\documentclass[reqno,11pt,a4paper]{amsart}
\usepackage[utf8]{inputenc}  
\usepackage[T1]{fontenc}  
\usepackage{amscd,amssymb,amsmath,latexsym,url,mathrsfs,upgreek,graphicx,mathtools,color}
\usepackage[all]{xy}
\usepackage{hyperref} \hypersetup{breaklinks=true}

\def\ov#1{{\overline{#1}}}

\def\wt#1{{\widetilde{#1}}}

\newcommand{\res}{\operatorname{res}}
\newcommand{\Res}{{\rm Res}}
\newcommand{\supp}{\operatorname{supp}}
\newcommand{\Gal}{\operatorname{Gal}}

\newcommand{\m}{\operatorname{m}}
\newcommand{\dd}{\hspace{1pt}\operatorname{d}\hspace{-1pt}}
\renewcommand{\and}{{\quad \text{ and } \quad}}

\newcommand{\coeff}{\operatorname{coeff}}
\newcommand{\id}{\operatorname{id}}

\newcommand{\Spec}{\operatorname{Spec}}

\newcommand{\e}{\operatorname{e}}

\newcommand{\h}{\operatorname{h}}
\newcommand{\hcan}{\h}

\newcommand{\sing}{{\rm sing}}
\newcommand{\reg}{{\rm reg}}

\newcommand{\ball}{{B}}
\newcommand{\ResX}{\Res_{X}}

\newcommand{\ResXC}{\Res_{X_{\C}}}
\newcommand{\Tr}{{\operatorname{Tr}}}

\newcommand{\KK}{\operatorname{K}}
\renewcommand{\Re}{\operatorname{Re}}

\newcommand{\Cr}{\operatorname{Cr}}

\newcommand{\diag}{\operatorname{diag}}
\renewcommand{\div}{\operatorname{div}}

\makeatletter
\newcommand{\pushright}[1]{\ifmeasuring@#1\else\omit\hfill$\displaystyle#1$\fi}
\newcommand{\pushleft}[1]{\ifmeasuring@#1\else\omit$\displaystyle#1$\hfill\fi\ignorespaces}
\makeatother

\def \A{\mathbb{A}}

\def \C{\mathbb{C}}

\def \K{\mathbb{K}}

\def \N{\mathbb{N}}
\def \P{\mathbb{P}}
\def \Q{\mathbb{Q}}
\def \R{\mathbb{R}}

\def \Z{\mathbb{Z}}

\def\cB {{\mathcal B}}

\def\cO {{\mathcal O}}

\def\cS {{\mathcal S}}

\def\cX {{\mathcal X}}

\def\Qbar {{\overline{\Q}}}

\newcommand{\bfb}{{\boldsymbol{b}}}

\newcommand{\bfd}{{\boldsymbol{d}}}
\newcommand{\bfe}{{\boldsymbol{e}}}
\newcommand{\bff}{{\boldsymbol{f}}}

\newcommand{\bfh}{{\boldsymbol{h}}}
\newcommand{\bfl}{{\boldsymbol{l}}}
\newcommand{\bfell}{{\boldsymbol{\ell}}}
\newcommand{\bfm}{{\boldsymbol{m}}}

\newcommand{\bft}{{\boldsymbol{t}}}
\newcommand{\bfu}{{\boldsymbol{u}}}

\newcommand{\bfx}{{\boldsymbol{x}}}
\newcommand{\bfy}{{\boldsymbol{y}}}
\newcommand{\bfz}{{\boldsymbol{z}}}

\newcommand{\bfalpha}{{\boldsymbol{\alpha}}}
\newcommand{\bfbeta}{{\boldsymbol{\beta}}}

\newcommand{\bfdelta}{{\boldsymbol{\delta}}}

\newcommand{\bfvarepsilon}{{\boldsymbol{\varepsilon}}}

\newcommand{\bfphi}{{\boldsymbol{\phi}}}
\newcommand{\bfxi}{{\boldsymbol{\xi}}}

\newcommand{\bfzero}{\boldsymbol{0}}
\newcommand{\bfone}{\boldsymbol{1}}

\numberwithin{equation}{section}
\theoremstyle{definition}
\newtheorem{defn}{Definition}
\numberwithin{defn}{section}

\newtheorem{rem}[defn]{Remark}
\newtheorem{exmpl}[defn]{Example}

\theoremstyle{plain}
\newtheorem{lem}[defn]{Lemma}

\newtheorem{prop}[defn]{Proposition}
\newtheorem{thm}[defn]{Theorem}
\newtheorem{cor}[defn]{Corollary}
\newtheorem{prop-def}[defn]{Proposition-Definition}

\begin{document}

\title[Multivariate residues and the elimination theorem]{Bounds for
  multivariate residues and for the polynomials in the elimination
  theorem}

\author[Sombra]{Mart{\'\i}n~Sombra} 
\address{Instituci\'o Catalana de Recerca
  i Estudis Avan\c{c}ats (ICREA). Passeig Llu{\'\i}s Companys~23,
  08010 Barcelona, Spain  \vspace*{-2.5mm}} 
\address{Departament de Matem\`atiques i
  Inform\`atica, Universitat de Barcelona. Gran Via 585, 08007
  Bar\-ce\-lo\-na, Spain} 
\email{sombra@ub.edu}
\urladdr{\url{http://www.maia.ub.edu/~sombra}}

\author[Yger]{Alain Yger}
\address{ Institut de Mathématiques, Université de Bordeaux.
351 cours de la Libération, 33405 Talence, France}
\email{Alain.Yger@math.u-bordeaux.fr}
\urladdr{\url{http://www.math.u-bordeaux.fr/~ayger/}}

\date{\today} 
 \thanks{Sombra was
  partially supported by the MINECO research projects
  MTM2012-38122-C03-02, MTM2015-65361-P and PID2019-104047GB-I00. Yger
  was partially supported by the CNRS research project PICS 6381
  ``G\'eom\'etrie diophantienne et calcul formel''.}

\begin{abstract}
  We present several upper bounds for the height of global residues of
  rational forms on an affine variety defined over $\Q$. As an
  application, we deduce upper bounds for the height of the
  coefficients in the Bergman-Weil trace formula.

  We also present upper bounds for the degree and the height of the
  polynomials in the elimination theorem on an affine variety defined
  over $\Q$. This is an arithmetic analogue of Jelonek's effective
  elimination theorem, and it plays a crucial role in the proof of our
  bounds for the height of global residues.
\end{abstract}

\maketitle

\vspace*{-8mm}

\setcounter{tocdepth}{1}
\tableofcontents

\vspace*{-8mm}

 \section[Introduction]{Introduction}

Given a regular sequence $\bff=(f_1,...,f_n)$ in the $n$-dimensional
local ring $\mathscr O_{\C^n,0}$ of germs at 0 of holomorphic
functions of $\C^{n}$ and
$\omega = g \dd x_1 \wedge \dots \wedge \dd x_n$ a germ at $0$ of a
holomorphic $n$-form on $\C^{n}$, the Grothendieck local residue at $0$
of $\omega$ with respect to $\bff$ is defined as

\begin{equation*}
{\Res}_0 \begin{bmatrix} \omega \\ 
  \bff \end{bmatrix} = \frac{1}{(2i\pi)^n}
\int_{|f_1|=\varepsilon_1,...,|f_n|=\varepsilon_n} \frac{g \dd
  x_1 \wedge \dots \wedge \dd x_n}{f_1\cdots f_n}
\end{equation*}  
for
$\bfvarepsilon=(\varepsilon_{1},\dots, \varepsilon_{n})\in
(\R_{\ge0})^{n}$ sufficiently small. 

This notion can be transposed to the more general situation of an
$r$-dimensional local ring $\mathscr O_{X,\bfx_0}$ for a (reduced)
complex space $X$ of pure dimension $r\ge 1$, $\bff=(f_1,...,f_r)$ a
regular sequence in $\mathscr O_{X,\bfx_0}$, and $\omega $ a germ at
$\bfx_{0}$ of a holomorphic $n$-form on $X$, thus leading to the
definition of the local residue, denoted as
$$
{\Res}_{X,\bfx_0}\begin{bmatrix} \omega \\
  \bff \end{bmatrix}.
$$
When $X\subset \mathbb \C^n$ is an algebraic variety of pure dimension
$r\ge 1$, $\bff=(f_1,...,f_r)$ a family of polynomials in
$\C[x_{1},\dots, x_{n}]$ defining a complete intersection on $X$, and
$\omega$ a rational $r$-form on $\C^{n}$ that is regular on
$X\cap V(\bff)$, the zero set of the system $\bff$ on $X$, the
\emph{(global) residue} on $X$ of $\omega$ with respect to $\bff$
might be defined as of the total sum the local residues at the points
of $X\cap V(\bff)$, that is
\begin{equation}
\label{defres1} 
{\Res}_{X}\begin{bmatrix} \omega \\
  \bff \end{bmatrix}
=\sum\limits_{\bfx_0\in X\, \cap\,  
V(\bff)} {\Res}_{X,\bfx_0}\begin{bmatrix} \omega \\
  \bff \end{bmatrix}. 
\end{equation}
 
These residues play an important role in division formulae on
polynomial rings. An example of these is the Bergman-Weil trace
formula for the case when $X=\C^{n}$ \cite{AizYu:gnus, Tsikh:tsikh,
  BoyHi2:gnus}.  This formula was the key tool towards the first
versions of the arithmetic Nullstellensatz, giving bounds for the
degree and the height of the polynomials in the Nullstellensatz
\cite{BerensteinGayVidrasYger:bgvy,BY99:gnus}. 
Global residues have also a deep connection with duality for algebraic
varieties \cite{K08:rdpav} and, moreover, they
appear in several fields of mathematics and theoretical physics, like
scattering amplitude computations beyond Feynman diagrams
\cite{SoY15:segdr}.

Many of these applications involve not just one residue, but the whole
residue multi-sequence associated to the triple $(X,\bff,\omega)$,
given by
\begin{equation}
  \label{eq:83}
\Big( \ResX \Big[\begin{matrix} \omega\\
  \bff^{\bfalpha+\bfone} \end{matrix} \Big]\Big)_{ \bfalpha\in \N^{n}}
\end{equation}
with
$\bff^{\bfalpha+\bfone}:=(f_1^{\alpha_1+1},...,f_r^{\alpha_r+1})\in
\C[x_{1},\dots, x_{n}]^{r}$
for $\bfalpha\in \N^{n}$. For instance, this multi-sequence is
relevant for the Bergman-Weil formula and for the representation of
traces in terms of residues, as explained in
\S\ref{sec:relat-trac-divis}.

In the arithmetic setting, that is, when the triple $(X,\bff,\omega)$
is defined over the field of rational numbers, the corresponding
residues are rational numbers. In spite of their applications, up to
our knowledge there are no available results allowing to control the
arithmetic complexity or \emph{height} of global residues, that is, to
bound the numerator and denominator of these rational numbers.  It is
precisely our aim in this paper to provide this kind of bounds.

To state our results, suppose for the rest of this introduction that
$X$ is a subvariety of $\A^{n}_{\Q}$ of pure dimension $r$, $\bff$ a
system of nonconstant polynomials with integer coefficients defining a
complete intersection on $X$, and $\omega$ a rational $r$-form defined
over $\Q$ that is regular on $X\cap V( \bff)$.  For
$\bfalpha\in \N^{n}$, write the corresponding global residue as
 \begin{displaymath}
\ResXC \Big[\begin{matrix} \omega\\
  \bff^{\bfalpha+\bfone} \end{matrix} \Big] = \frac{a_{\bfalpha}}{b_{\bfalpha}}
 \end{displaymath}
 with $a_{\bfalpha}\in \Z$ and $ b_{\bfalpha}\in \Z\setminus \{0\}$
 coprime.

 For a polynomial $f\in \Z[x_{1},\dots, x_{n}]$, its \emph{height},
 denoted by $\h(f)$, is the logarithm of the maximum of the absolute
 value of its coefficients.  A \emph{polynomial} $r$-form $\tau$ on
 $\C^{n}$ defined over $\Z$ is a holomorphic $r$-form that writes down
 as
\begin{displaymath}
  \tau=\sum_{I} g_{I}\dd \bfx_{I},
\end{displaymath}
the sum being over the subsets $I\subset \{1,\dots, n\}$ of cardinality
$r$, with  $g_{I}\in \Z[x_{1},\dots, x_{n}]$ and 
\begin{equation*}
  \dd \bfx_{I}=\bigwedge_{j=1}^{r}\dd x_{i_{j}}.
\end{equation*}
for $1\le i_{1}<\dots<i_{r}\le n$ such that 
$I=\{i_{1},\dots, i_{r}\}$.
The \emph{degree} and the \emph{length} of $\tau$, denoted by $\deg(\tau)$ and by
$\h_{1}(\tau)$, are respectively defined as the maximum degree of the
$g_{I}$'s and as the logarithm of the $\ell^{1}$-norm of the
coefficient list of all the $g_{I}$'s.

The \emph{degree} and \emph{height} of the subvariety $X$, denoted by
$\deg(X)$ and by $\h(X)$, are respectively defined as the degree and
the canonical height of the closure of its image with respect to the
standard inclusion $\mathbb A_\Q^n \hookrightarrow \mathbb
P^n_\Q$.
The degree and the height of an affine variety are measures of its
geometric and arithmetic complexity, see for instance
\cite[\S1.2]{KPS:sean} or \cite[\S 2.3]{DKS:hvmsaN} for more details.

The following result corresponds to Theorem \ref{thm:10}, and bounds
the numerator and denominators in the residue multi-sequence
\eqref{defres1} in  the most general situation considered in this
paper.

\begin{thm}
\label{thm:7}
With notation as above, 
write $ \omega= \tau /f_0$ with $\tau$ a
polynomial $r$-form on $\C^{n}$ defined over $\Z$ and
$f_0\in \Z[x_{1},\dots, x_{n}]$ not vanishing on $X\cap V(\bff)$.  Set
$d_{i}=\deg(f_{i})$, $i=0,\dots, r$, and $e =\deg(\tau)$.  Set also 
\begin{equation*}
  D_{X,\bff}=\deg(X)
\prod_{j=1}^r d_j 
\  {and}\  \kappa_{X,\bff}= \frac{\hcan( X)}{\deg(X)}
   + \sum_{j=1}^{r}\frac{\h( f_{j})}{d_{j}}+ 4(n+5)^{2} \log((n+1)
\deg(X)) .
\end{equation*}
Then, for $\bfalpha\in \N^{r}$, 
\begin{multline} \label{eq:111}
\log|a_{\bfalpha} |,
  \log| b_{\bfalpha}|+ \h_{1}(\tau)
\le \binom{n}{r}\big(  \h_{1}(\tau) + 
e\, (r+1) \, D_{X,\bff}\, \kappa_{X,\bff} 
\\ + (|\bfalpha|+1) \big( 2  (r+1) \,  D_{X,\bff} \h(f_{0}) +
     (3d_{0}+r+1) \, D_{X,\bff}^{2} \, 
 \kappa_{X,\bff} \big)\big) .  
\end{multline}
\end{thm}

In Theorem \ref{thm:11}, we give a sharper bound for the case when
$X=\A_{\Q}^{n}$: set now
\begin{equation*}
  D_{\bff}=\prod_{j=1}^n d_j \and \kappa_{\bff}'=
  \sum_{j=1}^{n}\frac{\h(f_{j})}{d_{j}} + (4n+8)\log(n+3). 
\end{equation*}
Then, for $\bfalpha\in\N^{n}$,
\begin{multline*}
\log|a_{\bfalpha} |,
  \log| b_{\bfalpha}|+ \h_{1}(\tau)
\le
\h_1(\tau) + e\, n\, D_{\bff}\, \kappa'_{\bff} \\ + (|\bfalpha|+1)\, \big(D_{\bff}\, \h(f_0) +
(d_0+n)(n D_{\bff}+1)\, D_{\bff}\, \kappa'_{\bff}\big).     
\end{multline*}

Also, in Theorems \ref{thm:8} and \ref{thm:9} we present sharper
bounds for the special case when $X$ is in good position with respect
to the system of coordinates, in the sense that
$\# X\cap V(x_{1},\dots, x_{r})=\deg(X)$, and $\omega$ is a polynomial
$r$-form. Nevertheless, all of these bounds have a quadratic
dependence on the B\'ezout number $D_{X,\bff}$ that does not seem
optimal, although at this moment we cannot not tell if this is the
case or not.

For the affine line $X=\A^1_\Q$, residue calculus and
Euclidean division are deeply correlated. The following result, 
corresponding to  Theorem \ref{thm:4}, bounds the
numerators and denominators of the residue sequence of to a
polynomial $1$-form on the affine line.




\begin{thm}
\label{thm:13}
Let $f\in \Z[x]\setminus \Z$ and $g\in \Z[x]$. Set
$d=\deg(f)$ and $e =\deg(g)$, and let $f_{d}$ be the leading
coefficient of $f$.  Then, for $\alpha \in \N$,
\begin{equation*}
  f_d^{e+1- (\alpha+1)(d-1)}  \Res  \Big[\begin{matrix} {g} \dd x \cr f^{\alpha+1}\end{matrix}\Big] \in \Z
\end{equation*}
and
\begin{equation*}
   \log \Big|f_d^{e+1 - (\alpha+1)(d-1)}   \Res  \Big[\begin{matrix} g \dd x \cr 
f^{\alpha+1}\end{matrix}\Big]\Big| 
 \leq \h_1(g) + (e +1-(\alpha+1)d)\h(f) + (e -d+1)  \log (2) .
\end{equation*}
If $e < (\alpha+1) d -1$, 
then $ \Res  \Big[\begin{matrix} g \dd x \cr 
f^{\alpha+1}\end{matrix}\Big] =0$. 
\end{thm}

In Theorems \ref{thm:1} and \ref{thm:6}, we generalize this result to
residues on a higher dimensional affine variety and a system of
univariate polynomials in separated variables.  In contrast to the
general case, these upper bounds do seem to be sharp, as shown by
Example~\ref{exm:1} for the case of the affine line.

Our approach to these results is based on the  arithmetic
membership problem that we next explain, since it might be of independent interest.
Recall that
$X\subset \A_{\Q}^{n}$ is a variety of pure dimension $r$ and
$\bff=(f_{1},\dots,f_{r})$ a system of polynomials in
$\Z[x_{1},\dots, x_{n}]\setminus \Z$ with a finite number of common zeros in
$X$. A classical method to solve the system of equations given
by
\begin{displaymath}
  f_{1}(\bfx)=\dots=f_{r}(\bfx)=0 \quad \text{ for } \bfx \in X
\end{displaymath}
is to eliminate variables, that is, find $\phi_{l}\in \Z[x_{l}]
\setminus \{0\}$, $l=1,\dots, r$, and $a_{l,i}\in \Z[x_{1},\dots,
x_{n}]$, $l,i=1,\dots, {r}$, such that
\begin{equation}\label{eq:110}
  \phi_{l}=\sum_{i=1}^{{r}}a_{l,i}f_{i} \quad \text{ on } X,
\end{equation}
in the sense that this polynomial relation holds modulo the ideal of
definition of $X$.

In \cite{Jelonek:eN}, Jelonek obtained an optimal upper bound for the
degrees of these polynomials, using a variant of his approach to the
effective Nullstellensatz. Here we prove the following arithmetic
analogue of this result, corresponding to Corollary~\ref{cor:1} in the
body of the paper.  Its proof proceeds by adapting Jelonek's approach
with the tools from geometric and arithmetic intersection theory from
\cite{DKS:hvmsaN}.

\begin{thm}
\label{thm:14}
With notation as above, there are
$\phi_{l}\in \Z[x_{l}]\setminus \{0\}$ and
$a_{l,1},\dots, a_{l,r}\in \Z[x_{1},\dots, x_{n}]$, $l,i=1,\dots, {r}$, satisfying the identity \eqref{eq:110} with
\begin{align*}
&\deg(\phi_{l}), \deg (a_{l,i}) +d_{i} \le \Big( \prod_{j=1}^{r}
  d_{j}\Big) \deg(X), \\[-2mm]
& 
\h(\phi_{l}), \h(a_{l,i})+\h(f_{i})\le  \Big(\prod_{j=1}^{r}d_{j}
  \Big)\Big(\h(X) + \deg(X)\Big(
  \sum_{j=1}^{r}\frac{\h(f_{j})}{d_{j}}\hspace{40mm}\\[-2mm]
&\pushright{+ (r+1) \log(2(r+2)(n+1)^{2})  \Big)\Big).}  
\end{align*}
\end{thm}

In the case $X=\A^{n}_{\Q}$, we have $r=n$,  $\deg(\A^{n}_{\Q})=1$ and
$\h(\A^{n}_{\Q})=0$. If moreover $\deg(f_{j})\le d$ and
$\h(f_{j})\le h$ for all $j$, then the previous bound specializes to
\begin{align*}
& \deg(\phi_l), \deg (a_{l,i}) + \deg(f_{i}) \le d^{n}, \\
& \h(\phi_{l}), \h(a_{l,i})+\h(f_{i})\le
n d^{n-1} h + (n+2) \log(2(n+2)(n+1)^{2}) d^{n}.   
\end{align*}

The proof of Theorem \ref{thm:7} and \ref{thm:11} is incremental. We
first obtain the bounds for residues on the affine line (Theorems
\ref{thm:13} and \ref{thm:5}) by a recurrence scheme based on the
relationship of these residues with the Euclidean division.  The
treatment of residues on affine varieties and a system of univariate
polynomials in separated variables (Theorems \ref{thm:1} and
\ref{thm:6}) is based on the arithmetic Perron's theorem from
\cite{DKS:hvmsaN} and the relationship between residues and traces on
polynomial algebras.  The general case is then treated by reducing to
the case of univariate polynomials by applying the transformation law
(Theorem \ref{propTL}) and Theorem \ref{thm:14}.


As an application of our results on the heights in the residue
multi-sequence, we derive a bound for the coefficients in the
Bergman-Weil trace formula.  To formulate it, let
$ \bff=(f_{1},\dots, f_{n})$ be a family of polynomials in
$\Z[x_{1},\dots, x_{n}]\setminus \Z$ defining a complete intersection
on $\A^{n}_{\Q}$ and, for simplicity, suppose that the map
$\A^{n}_{\Q}\to \A^{n}_{\Q}$ defined by $\bfx\mapsto \bff(\bfx)$ is
proper. Set $\bfd=(d_{1},\dots, d_{n})\in \N^{n}$ with
$d_{i}=\deg(f_{i})$.  For $p\in \Z[x_{1},\dots, x_{n}]$, the
Bergman-Weil trace formula gives an explicit
polynomial identity
  \begin{displaymath}
    p=\sum_{\bfalpha\in \N^{n}}p_{\bfalpha} \, \bff^{\bfalpha}
  \end{displaymath}
  with $p_{\bfalpha}\in \Q[x_{1},\dots, x_{n}]$ of degree bounded
  by $|\bfd|-n$, that are zero except for a finite number of
  $\bfalpha$'s (Theorem \ref{thm:3}) .  The next result corresponds to Corollary \ref{cor:3}.

\begin{thm} \label{thm:15}
With notation as above, set
\begin{displaymath}
D_{\bff}= \prod_{j=1}^{n}d_{j} \and  \kappa''_{\bff}=\sum_{j=1}^{n}\frac{\h(f_{j})}{d_{j}}+ 3(n+2)\log (n+2).
\end{displaymath}
Set also $e=\deg(p)$. Then there exists
$\vartheta\in \Z\setminus \{0\}$ with
$ \log|\vartheta| \le n \,\kappa''_{\bff} $ such that, for
$\bfalpha\in \N^{n}$, we have that
$\vartheta^{e+|\bfd|+(|\bfalpha|+1) (nD_{\bff}+1)} p_{\bfalpha}\in \Z
[x_{1},\dots, x_{n}]$ and
\begin{displaymath}
  \h\big(\vartheta^{e+|\bfd|+(|\bfalpha|+1) (nD_{\bff}+1)} p_{\bfalpha}\big) \le
  \h_{1}(g) + (e+|\bfd|+(|\bfalpha|+1) (nD_{\bff}+1)) \, n\,D_{\bff}
  \, \kappa''_{\bff}.
  \end{displaymath} 
\end{thm}

In a general way, we expect that the result from this
paper might be useful to obtain arithmetic versions of other problems
from effective commutative algebra allowing an analytic treatment,
like the Brian\c{c}on-Skoda theorem or the Artin-Rees lemma, among
others.

  The paper is organized as follows. In \S\ref{sectionresidue}, we
  recall the definition of the global residue in the algebraic setting
  and its basic properties, including the transformation law and other
  results from multivariate residue calculus. In \S
  \ref{sec:1-dimensional-case}, we study in detail global residues on
  the affine line.  Section \ref{sec:tensor univariate case} is
  devoted to the case of an arbitrary affine variety and univariate
  polynomials in separated variables.  In \S \ref{sec:an-arithm-elim},
  we present the arithmetic analogue of Jelonek's theorem, bounding
  the degree and the height of the polynomials in the elimination
  theorem. Finally, in \S\ref{sectionestimates} we exploit these
  arithmetic constructions in accordance with multivariate residue
  calculus as described in \S\ref{sectionresidue}, to achieve the
  stated bounds for the height of multivariate residues.
  
\medskip \noindent {\bf Acknowledgments.}  Part of this work was done
while the authors met at the Universitat de Barcelona and at the
Institut de Math\'ematiques de Bordeaux. We thank these institutions
for their hospitality.

\section{Global residues on affine varieties}\label{sectionresidue} 

In this section, we introduce global residues of meromorphic forms on
 affine varieties and recall its basic properties. This material is
classical, and we base most of our exposition on the book
\cite{ColeffHerrera:crafm} and on the paper \cite{BVY:gnus}. We refer
to these sources for precisions and the proof of the stated results.

Boldface letters and symbols denote finite sets or sequences of
objects, where the type and number should be clear from the context:
for instance, for $n\ge 1$ we denote by $\bfx$ the group of variables
$(x_1,\dots,x_n)$, so that $\C[\bfx]=\C[x_1,\dots,x_n]$.

\subsection{Definition and basic properties} \label{sec:defin-basic-prop} Let $X\subset \A_{\C}^{n}$
be a variety of pure dimension $r\ge 1$ and
$\bff=(f_1,...,f_r)\in \C[x_{1},\dots, x_{n}]^{r}=\C[\bfx]^{r}$ a
family of $r$ polynomials in $n$ variables defining a complete
intersection on $X$. To simplify the exposition, we identify $X$ with
its set of complex points $X(\C)$.

We denote by $X^{\sing}$ and $X^{\reg}$ the subsets of $X$ of singular
and regular points, respectively. Since the family $\bff$ defines a
complete intersection on $X$, its Jacobian locus is proper closed
subset of $X$. We denote by $W$ an algebraic hypersurface of $X$
containing both $X^{\sing}$ and this Jacobian locus.
We also denote by 
\begin{displaymath}
  X\cap V(\bff)=\{\bfx\in X \mid f_{1}(\bfx)=\dots= f_{r}(\bfx)=0\}
\end{displaymath}
the finite set of zeros of the system $\bff$ on $X$. 

Let $\|\cdot\|$ denote the Euclidean norm of $\C^{n}$ and fix $R>0$
such that the open ball
\begin{displaymath}
  \ball_{R}=\{\bfx\in \C^n \mid \|x\| < R\}
\end{displaymath}
contains $X\cap V(\bff)$.  Let $\eta>0$ and
$\bfvarepsilon=(\varepsilon_1,...,\varepsilon_r)\in (\R_{\ge 0})^{n}$
with $\varepsilon_{i}\le \eta$ for all $i$, and consider the tube
around $X\cap V(\bff)$ given by
\begin{equation}
  \label{eq:96}
  \Gamma_{\bfvarepsilon}=
  \ball_{R}\cap
  \{\bfx\in X\mid \, |f_1(\bfx)|= \varepsilon_1 , \dots, |f_r(\bfx)|= \varepsilon_r \}. 
\end{equation}
When $\eta$ is sufficiently small, this is a compact, not necessarily
connected, semianalytic set of dimension~$r$, without components
contained in $W$ and smooth outside this hypersurface.  We orient the
smooth semianalytic set $\Gamma_{\bfvarepsilon}\setminus W$ so that
the inverse image to it of the differential $r$-form
$\bigwedge_{j=1}^r \dd {\arg} (f_j)$ is positive.

Let $\omega$ be a {meromorphic} $r$-form on $\C^{n}$ that is
regular on $X\cap V(\bff)$. For
$\bfalpha=(\alpha_{1},\dots, \alpha_{r}) \in \N^r$, the integral
\begin{equation*}
  \int_{\Gamma_{\bfvarepsilon}} \frac{\omega}{f_{1}^{\alpha_1+1}\cdots
    f_{r}^{\alpha_r+1}} 
\end{equation*}
is defined as the integral of a regular $(r,0)$-form on the
$r$-dimensional smooth semianalytic chain
$\Gamma_{\bfvarepsilon}\setminus W$. Its value does not depend on the
choice of $W$. It does neither depend on the choice of
$\bfvarepsilon$ by Stokes' theorem on semianalytic chains, see
\cite[\textsection II.B]{Her:bsmf} or \cite{P74:these}.

\begin{defn} \label{def:2}Let $X\subset \A_\C^n$ be a variety of pure dimension $r\ge 1$,
  $\bff=(f_1,...,f_r)\in \C[\bfx]^{r}$ a family of $r$ polynomials
  defining a complete intersection on $X$, and $\omega$ a meromorphic
  $r$-form on $\C^{n}$ that is regular on $X\cap V(\bff)$. With
  notation as above, given $\bfalpha\in \N^r$, the \emph{(global)
    residue} on $X$ of $\omega$ with respect
  to~$\bff^{\bfalpha+\bfone}:=(f_1^{\alpha_1+1},...,f_r^{\alpha_r+1})$
  is defined as
  \begin{equation*}
\ResX \Big[\begin{matrix} \omega\\ \bff^{\bfalpha + \bfone} \end{matrix} \Big] =     
\frac{1}{(2\pi i)^{r}} \int_{\Gamma_{\bfvarepsilon}} \frac{\omega}{f_{1}^{\alpha_1+1}\cdots f_{r}^{\alpha_r+1}}
  \end{equation*}
  for any $\eta>0$ sufficiently small and
  $\bfvarepsilon=(\varepsilon_{1},\dots, \varepsilon_{r})\in
  (\R_{\ge0})^{n}$ with $\varepsilon_{i}\le \eta$ for all $i$.
\end{defn}

\begin{rem}
  \label{rem:2}
This notion coincides with that   in
\eqref{defres1}. Since in this paper we are only concerned with global
residues, here we define them directly without passing through the
local case. 
\end{rem}

To profit from the flexibility of analysis, as well as to emphasize
the action of the residue instead of the result of this action on an
specific form, it is often convenient to enlarge this notion with a
currential approach.  Following Coleff and Herrera \cite[\S
4.1]{ColeffHerrera:crafm}, we can define  a residual current by
considering the limit of residue integrals along special, so-called
``admissible'', paths of the form
\begin{displaymath}
  s\mapsto \bfvarepsilon(s)=(s^{\beta_{1}}, \dots, s^{\beta_{r}})
\end{displaymath}
for some fixed positive numbers $\beta_{1}\gg\dots\gg \beta_{r}$.  Given a
compactly supported $(r,0)$-form~$\eta$ and $\bfalpha \in \N^r$, the limit
\begin{displaymath}
    \bigg\langle \bigwedge\limits_{j=1}^r \overline\partial 
  \bigg(\frac{1}{f_j^{\alpha_j+1}}\bigg) \wedge [X], \eta \bigg\rangle
  =\lim_{s\to 0}    \frac{1}{(2\pi i)^{r}} \int_{\Gamma_{\bfvarepsilon(s)}} \frac{\eta}{f_{1}^{\alpha_1+1}\cdots f_{r}^{\alpha_r+1}}
\end{displaymath}
exists and defines a $(0,r)$-current which is independent of the choice
of the admissible path \cite{ColeffHerrera:crafm}.  Using  an
arbitrary 
$C^\infty$-function $\chi\colon \mathbb C^n \to \R$ with compact
support that is identically equal to $1$ on a neighborhood of $X\cap
V(\bff)$, the residues in Definition~\ref{def:2} can be then written
as 
\begin{equation}
  \label{eq:95}
  \ResX \bigg[\begin{matrix} \omega\\ \bff^{\bfalpha + \bfone} \end{matrix}\bigg] =     
  \bigg\langle \bigwedge\limits_{j=1}^r \overline\partial 
  \bigg(\frac{1}{f_j^{\alpha_j+1}}\bigg) \wedge [X], \chi \, \omega \bigg\rangle .
\end{equation}

Residues of holomorphic $r$-forms can also be represented as integrals
of the Bochner-Martinelli type over a $(2r-1)$-dimensional cycle.
For  $\bfalpha\in \N^r$, we set 
\begin{displaymath}
|\bfalpha|=\alpha_1 + \dots + \alpha_r \and \bfalpha!=\alpha_1!\dots \alpha_r!. 
\end{displaymath}

\begin{prop}
\label{prop:10}
With notation as in Definition \ref{def:2},  suppose that $\omega$
is a holomorphic $r$-form on $\C^{n}$ and let $R>0$ such that
$X\cap V(\bff) \subset \ball_{R}$. Then
$\ResX \Big[\begin{matrix} \omega \\ \bff^{\bfalpha +
    \bfone} \end{matrix} \Big]$ is equal to
\begin{equation*}
\frac{(-1)^{\frac{r(r-1)}{2}} (|\bfalpha|+r-1)!}{(2i\pi)^r \bfalpha!}
\, \int_{ \partial (X\cap \ball_R)} 
 \bigg(\prod_{j=1}^{r} \overline f_{j}^{\alpha_{j}}\bigg) \, \frac{\sum_{j=1}^r (-1)^{j-1} 
\overline f_j \bigwedge_{{l =1 \atop l \not=j}}^r 
\overline {\dd f_l}}{(\sum_{j=1}^r |f_j|^2)^{r+|\bfalpha|}} 
\wedge \omega
\end{equation*} 
\end{prop}

\begin{proof}
  When $X=\C^n$, such Bochner-Martinelli type integral representation
  formulae are also known as Andreotti-Norguet formulae, see for instance
  \cite[\S 2.4]{BerensteinGayVidrasYger:bgvy} or 
  \cite[\S 3.1 and 3.2]{TsY:rescur}.  We
  adapt here the proof to the relative case, when $X$ is a
  $r$-dimensional  subvariety in $\A^n_\C$ with $1\leq r <n$.

  Within this proof, we set for short
  $\bff^{\bfalpha}= \prod_{j=1}^{r}f_{j}^{\alpha_{j}}$. Consider the
  $(0,r-1)$-form
  \begin{equation*}
    \Omega_{\bff,\bfalpha}= \overline \bff^\bfalpha\, \sum_{j=1}^r
  (-1)^{j-1} \overline f_j \bigwedge_{{l =1 \atop l \not=j}}^r
  \overline {\dd f_l}
  \end{equation*}
  on $\C^{n}$, and let $\lambda$ be a formal parameter. A formal
  computation shows that, if $[X]$ denotes the integration current on
  $X$, then
\begin{equation}\label{bochnermartproof1}
\dd \bigg( \frac{\prod_{j=1}^{r}|f_j|^{2\lambda}}{\|\bff\|^{2(r+|\bfalpha|)}}
\, \Omega_{\bff,\bfalpha} \wedge 
\omega \wedge [X]\bigg) 
= r\, \lambda
\frac{\big(\prod_{j=1}^{r}|f_j|^{2\lambda} \big)\bff^{\bfalpha}}{\|\bff\|^{2(r+|\bfalpha|)}}\, 
\bigwedge_{l=1}^r  \overline{\dd f_l} \wedge \omega \wedge [X].
\end{equation} 

For $\bft\in (\R_{\ge 0})^{n}$, set
$ \bft_{\bfalpha} =
\big(t_1^{1/(2(\alpha_1+1))},...,t_r^{1/(2(\alpha_r+1))}\big)$
and let $\Gamma_{\bft_{\bfalpha}}$ be the associated $r$-dimensional
semianalytic chain on $X$ as in \eqref{eq:96}.  For $\eta>0$ small
enough (depending on $R$ and $\bfalpha$) and $\bft\in (0,\eta]^r$,
we have that $\displaystyle {
\frac{1}{(2i\pi)^r} \int_{\Gamma_{\bft_\bfalpha}} 
\frac{\omega}{f_1^{\alpha_1+1}\cdots f_r^{\alpha_r+1}} }= 
\ResX \Big[\begin{matrix} \omega \\ \bff^{\bfalpha + \bfone} \end{matrix}\Big]. 
$
Consider the subset 
\begin{displaymath}
  I(R,\eta)= \big\{\bfx \in X\cap B_R\,\big|\ |f_j(\bfx)|\leq \eta^{\frac{1}{2(\alpha_j+1)}},\ j=1,...,r\big\}
\end{displaymath}
and let now $\lambda$ be a complex parameter with
$\Re(\lambda) \gg 1$.  It follows from Stokes' theorem and from
\eqref{bochnermartproof1} that
\begin{multline}\label{bochnermartproof1bis} 
\int_{\partial (X  \cap {\ball_R})}
 \frac{\prod_{j=1}^{r}|f_j|^{2\lambda}}{\|\bff\|^{2(r+|\bfalpha|)}}
\, \Omega_{\bff,\bfalpha} \wedge 
\omega\\
= \int_{\partial I(R,\eta)}  \frac{\prod_{j=1}^{r}|f_j|^{2\lambda}}{\|\bff\|^{2(r+|\bfalpha|)}}
\, \Omega_{\bff,\bfalpha} \wedge 
\omega
+ r\, \lambda \int_{X\setminus I(R,\eta)} 
\frac{\big(\prod_{j=1}^{r}|f_j|^{2\lambda} \big)\bff^{\bfalpha}}{\|\bff\|^{2(r+|\bfalpha|)}}\, 
\bigwedge_{l=1}^r  \overline{\dd f_l} \wedge \omega.
\end{multline} 
It follows again from Stokes' theorem 
and \eqref{bochnermartproof1} that 
\begin{align*} 
\int_{ \partial I(R,\eta)} 
\frac{\prod_{j=1}^{r}|f_j|^{2\lambda}}{\|\bff\|^{2(r+|\bfalpha|)}}
\, \Omega_{\bff,\bfalpha} \wedge 
\omega 
&=r\, \lambda \int_{I(R,\eta)} 
\frac{\big(\prod_{j=1}^{r}|f_j|^{2\lambda} \big)\bff^{\bfalpha}}{\|\bff\|^{2(r+|\bfalpha|)}}\, 
\bigwedge_{l=1}^r \overline{\dd f_l} \wedge \omega \\
&= \frac{r\lambda}{\prod_{l=1}^r (\alpha_l+1)}\, \int_{ I(R,\eta)}
\frac{\big(\prod_{j=1}^{r}|f_j|^{2\lambda} \big)\bff^{\bfalpha}}{\|\bff\|^{2(r+|\bfalpha|)}}\, 
\bigwedge_{l=1}^r \overline{\dd f_l^{\alpha_l+1}} \wedge \omega .  
\end{align*}
We have that
$\bigwedge_{l} \dd \big(|f_l|^{2(\alpha_l+1)}\big) = \bff^{\bfalpha
  +1}\, \bigwedge_{l} (\overline\partial f_l)^{\alpha_l+1}$.
It follows then from Lebesgue's domination and Fubini's theorems and
using the map
$\boldsymbol t=(|f_1|^{2(\alpha_1+1)},...,|f_r|^{2(\alpha_r+1)})$ to
define the slicing locally about each point in $X\setminus W$, that
since $W$ has Lebesgue measure $0$ with respect to the $r$-dimensional
Lebesgue measure on $X$ and the set $\Cr(\bft)$ of critical values of
the map $\bft|_{X}$ has Lebesgue measure $0$ in $(\R_{\ge 0})^r$
thanks to Sard's lemma,
\begin{align*}
& \frac{r\lambda}{\prod_{l=1}^r (\alpha_l+1)}\, \int_{ I(R,\eta)}
\frac{\big(\prod_{j=1}^{r}|f_j|^{2\lambda} \big)\bff^{\bfalpha}}{\|\bff\|^{2(r+|\bfalpha|)}}\, 
\bigwedge_{l=1}^r \overline{\dd f_l^{\alpha_l+1}} \wedge \omega \\
& = \frac{r\lambda}{\prod_{l=1}^r (\alpha_l+1)}\, \int_{ 
\{\bfx\in (X \cap \ball_R)\setminus W \mid \, t_j(\bfx)\leq \eta,\ j=1,...,r\}}
\frac{\prod_{l=1}^r 
|t_l(\bfx)|^{{\lambda}/{(\alpha_l+1)}}}
{\bff^{\bfalpha+1}\, 
\|\bff\|^{2(r+|\bfalpha|)}
}\, {\bigwedge_{l=1}^r 
\dd t_l(\bfx)} 
 \wedge \omega \\
& = (-1)^{(r(r-1)/2} \frac{r \lambda}{\prod_{l=1}^r (\alpha_l+1)} 
\int_{(0,\eta]^r\setminus \Cr(\bft)} \bigg(\int_{\Gamma_{{\boldsymbol t}_\bfalpha}} 
\frac{\omega}{\bff^{\bfalpha + \bfone}}\bigg)\, 
\frac{\prod_{l=1}^r t_l^{{\lambda}/{(\alpha_l+1)}}}{\big(
\sum_{l =1}^r t_l^{{1}/{(\alpha_l+1)}}\big)^{r+|\bfalpha|}} \dd \bft
\\ 
& = (-1)^{(r(r-1)/2} \, (2i\pi)^r \, \ResX \Big[\begin{matrix} \omega \\
  \bff^{\bfalpha + \bfone} \end{matrix} \Big]    
\times \frac{r\lambda}{\prod_{l=1}^r (\alpha_l+1)}\, \int_{(0,\eta]^r}
\frac{\prod_{l=1}^r t_l^{{\lambda}/{(\alpha_l+1)}}}{\big(
\sum_{l =1}^r t_l^{{1}/{(\alpha_l+1)}}\big)^{r+|\bfalpha|}} \dd \bft.
\end{align*} 
We now consider both sides of \eqref{bochnermartproof1bis} as meromorphic
functions of $\lambda$ having no poles in $\Re(\lambda) >-\kappa$
for some sufficiently small value of $\kappa>0$. Identifying the
values at $\lambda=0$ of both sides of \eqref{bochnermartproof1bis}, we
get
\begin{align*} 
& \int_{\partial (X \cap \ball_{R}{)}}
 \frac{ \Omega_{\bff,\bfalpha}}{\|\bff\|^{2(r+|\bfalpha|)}}
 \wedge \omega \\
&= 
(-1)^{(r(r-1)/2} \, (2i\pi)^r  \Bigg[ \frac{r\lambda}{\prod_{l=1}^r (\alpha_l+1)}\, \int_{(0,\eta]^r}
\frac{\prod_{l=1}^r t_l^{{\lambda}/{(\alpha_l+1)}}}{\big(
\sum_{l =1}^r t_l^{{1}/{(\alpha_l+1)}}\big)^{r+|\bfalpha|}} \dd \bft
  \Bigg]_{\lambda=0}  \hspace{-3mm}\ResX \Big[\begin{matrix} \omega \\
  \bff^{\bfalpha + \bfone} \end{matrix} \Big]  \\ 
& = (-1)^{(r(r-1)/2} \, (2i\pi)^r \,
  \frac{\bfalpha!}{(|\bfalpha|+r-1)!}     \ResX \Big[\begin{matrix} \omega \\ \bff^{\bfalpha + \bfone} \end{matrix} \Big],
\end{align*} 
which leads to \eqref{bochnermartproof1}.
\end{proof}

For each subset $I\subset \{1,\dots, n\}$ of cardinality $r$, write
$I=\{i_{1},\dots, i_{r}\}$ with $i_{1}<\dots<i_{r}$ and consider the
holomorphic $r$-form on~$\C^{n}$ given by
\begin{equation}\label{eq:89}
  \dd \bfx_{I}=\bigwedge_{j=1}^{r}\dd x_{i_{j}}.
\end{equation}

\begin{defn}
  \label{def:9}
  A holomorphic $r$-form $\omega$ on $\C^{n}$ is \emph{polynomial} if it writes
  down as 
\begin{equation*}
   \omega=\sum_{I}{g_{I}} \dd \bfx_{I}, 
\end{equation*}
the sum being over the subsets $I\subset\{1,\dots, n\}$ of
cardinality~$r$, with $g_{I}\in \C[\bfx]$ for all~$I$.  It is
\emph{defined over $\Z$} if $g_{I}\in \Z[\bfx]$ for all $I$.

  A meromorphic $r$-form $\omega$ on $\C^{n}$ is \emph{rational}
  if there is $h\in \C[\bfx]\setminus \{0\}$ such that $h\, \omega$ is
  a polynomial $r$-form on $\C^{n}$.  It is \emph{defined over
    $\Q$} if there is $h\in \Z[\bfx] \setminus \{0\}$ such that $h\, \omega$ is a
  polynomial $r$-form defined over $\Z$.
\end{defn}

We next list the basic properties of residues on affine varieties. We
will restrict to the algebraic setting and in particular, we will only
consider polynomial or rational forms, although several of these
properties hold in greater generality. As before, we assume that
$X\subset \A_{\C}^{n}$ is a variety of pure
dimension $r\ge 1$ and that $\bff=(f_1,...,f_r)\in \C[\bfx]^{r}$ is a
complete intersection on $X$.

Residues on affine varieties vanish on the ideal generated by $\bff$
in the ring of regular functions of $X$.

\begin{prop}
  \label{prop:16}
  Let $\omega$ be a rational $r$-form that is regular on an open
  subset $U\subset X$ containing $X\cap V(\bff)$ and $p \in (\bff)$,
  the ideal generated by $\bff$ in $\cO_{X}(U)$. Then
\begin{displaymath}
  \ResX \Big[\begin{matrix} p \, \omega \\ \bff \end{matrix} \Big]= 0.
\end{displaymath}
\end{prop}

\begin{proof} See instance \cite[\S 4.4, Theorem
  4.4.1(2)]{ColeffHerrera:crafm}.
\end{proof} 

Also, these residues are invariant under linear change of
variables. 

\begin{prop}
\label{changebasis} 
Let $\ell\colon \A^n_\C \mapsto \A^n_\C$ be an invertible affine
map. Then
  \begin{displaymath}
\ResX
\Big[\begin{matrix} 
\omega\\ \bff \end{matrix} \Big]
 = \Res_{\ell^{-1}(X)}
\Big[\begin{matrix} 
\ell^{*}\omega\\ \ell^{*}\bff \end{matrix} \Big].    
  \end{displaymath}
\end{prop} 

\begin{proof} See for example \cite[page 25]{BVY:gnus}.
\end{proof}   

Another important property is the Lagrange-Jacobi vanishing theorem.
Consider the map
\begin{equation}
  \label{eq:23}
  \varphi_{\bff}\colon X\longmapsto \A^{r}_{\C}, \quad \bfx\longmapsto
  \bff(\bfx).
\end{equation}
This map is proper if and only if there exist $\delta_{i}>0$,
$i=1,\dots, r$, and $C, \tau>0$ such that, for all $\bfx\in X$ with
$\|\bfx\|\ge C$,
\begin{equation}
  \label{eq:55}
  \sum_{i=1}^{r}\frac{|f_{i}(\bfx)|}{\|\bfx\|^{\delta_{i}}} \ge \tau,
\end{equation}
see for instance \cite[Theorem 5.2]{H:scbyioci}. In the case when the
homogenizations $f_j^{\h}$, $j=1,...,r$, have no common zeros in the
intersection of the hyperplane at infinity $\{x_0=0\}$ with the
Zariski closure of $X$ in $\P^n_\C$, in the inequality~\eqref{eq:55}
we can take
\begin{equation}\label{eq:71}
 \delta_i=\deg(f_{i}), \quad  i=1,...,r.
\end{equation}

\begin{thm}\label{jacobi} Suppose that
  the map $\varphi_{\bff}$ in \eqref{eq:23} is proper and let
  $\bfdelta=(\delta_{1},\dots, \delta_{r})$ be as in \eqref{eq:55}.
  Let $\omega$ be a polynomial $r$-form on $\C^{n}$ such that
  $\deg(\omega) < |\bfdelta|-r$. Then
\begin{displaymath} 
  \ResX
  \Big[\begin{matrix} 
    \omega\\ \bff \end{matrix} \Big]=0. 
\end{displaymath} 
\end{thm}

\begin{proof}
See for instance \cite[Proposition 4.1]{BVY:gnus}.
\end{proof}

We also need  the next extension of the transformation law for
affine residues. 

\begin{thm}
  \label{propTL}
  Let $\bfphi=(\phi_{i})_{1\le i\le r} \in \C[\bfx]^{r}$ and
  $A=(a_{i,j})_{1\le i,j\le r}\in \C[\bfx]^{r\times r}$ such that
  $\bfphi$ is a complete intersection on $X$ and $A\cdot
  \bff=\bfphi$.
  Let $\bfu=(u_{1},\dots, u_{r})$ be a group of $r$ variables and set
  $a_{l}=\sum_{i=1}^{r}a_{l,i} u_{i} \in \C[\bfu,\bfx]$, $i=1,\dots, r$.
  For $\bfalpha \in \N^r$, set
\begin{equation}
\label{eq:107}
  H= \det(A) \cdot \prod_{l=1}^{r}\Big( \sum_{k=0}^{|\bfalpha|} \phi_{l}^{k}\,
a_{l}^{|\bfalpha|-k} \Big) \in
\C[\bfu,\bfx] 
\end{equation}
and $G= \coeff_{\bfu^{\bfalpha}}(H) \in \C[\bfx]$, the coefficient of
$\bfu^{\bfalpha}$ in the monomial expansion of $H$ with respect to the
group of variables $ \bfu$. Let $\omega$ be a polynomial $r$-form on
$\C^{n}$. Then
\begin{equation}\label{eq:63}
\ResX \Big[\begin{matrix}
 \omega \cr 
f_1^{\alpha_{1}+1},...,f_{r}^{\alpha_{r}+1}\end{matrix}\Big]
= \ResX \bigg[\begin{matrix}
G \, \omega \cr 
\phi_1^{|\bfalpha|+1},...,\phi_{r}^{|\bfalpha|+1}\end{matrix}\bigg].
\end{equation}
\end{thm}

\begin{proof}
  The case when $\bfalpha=\bfzero$ is done in \cite[Proposition
  3.2]{BVY:gnus}. The general case when $\bfalpha\in \N^r$ is
  arbitrary, can be similarly proven by transposing the proof of
  \cite{BoyHi1:gnus} on $\C^{n}$ to the case of an arbitrary affine
  variety, using the Bochner-Martinelli integral representation of affine
  residues from Proposition \ref{prop:10}.
\end{proof}

When $X$ is the affine line, we can compute the residue of a
polynomial $1$-form as a coefficient in the Laurent expansion around
the point at infinity of a rational function.

\begin{prop}
  \label{prop:12}
  Let $f\in \C[x]\setminus \{0\}$ and $\omega= g\dd x$ be a polynomial
  $1$-form on $\A^{1}_{\C}$. Then
\begin{displaymath}
  \Res_{\A^{1}_{\C}}  \Big[\begin{matrix} \omega \\
f\end{matrix}\Big]  
\end{displaymath}
equals the coefficient of degree $-1$ in the expansion of $g/f$ as a
Laurent series around the point at infinity. 
\end{prop}

\begin{proof}
  Consider $\omega/f$ as a rational 1-form on $\P^1(\C)$. As usual, we
  identify the complex plane $\C$ with the open subset
  $\P^{1}(\C)\setminus \{\infty\}$, with $\infty=(0:1) $ the point at
  infinity.  With this identification,
\begin{equation*}
  \Res_{\A_{\C}^{1}}  \Big[\begin{matrix} \omega \cr 
f\end{matrix}\Big]=\sum_{\xi \in V(f)} 
\res_\xi \Big(\frac{\omega}{f}\Big),  
\end{equation*}
where
$\res_\xi$  the action of the local residue
at a point $\xi$.

The sum of the local residues of the rational 1-form $\omega/f$ on
$\P^{1}(\C)$ vanishes, as a consequence of Stokes' theorem on this
compact manifold. Hence
  \begin{equation*}
  \Res_{\A_{\C}^{1}}  \Big[\begin{matrix} \omega \cr 
f\end{matrix}\Big]= - \res_{\infty}\Big(\frac{\omega}{f}\Big) .
  \end{equation*}
The local residue is invariant under changes of
coordinates.  Putting
$y=x^{-1}$, we get
\begin{equation*}
 \res_{\infty}\Big(\frac{\omega}{f}\Big)  
  =-\res_{0}\Big(\frac{g(y^{-1})}{f(y^{-1})} \frac{\dd y}{y^{2}}
  \Big).
\end{equation*}
By Cauchy's integral formula, this coincides with the coefficient of
degree $1$ in the expansion of $g(y^{-1})/f(y^{-1})$ as a Laurent
series around the origin or, equivalently, with the coefficient of
degree $-1$ in the expansion of $g/f$ as a Laurent series around the
point at infinity.
\end{proof}

Reciprocally, we can compute  the Laurent expansion around the point at
infinity of the inverse of a polynomial, in terms of residues over the
affine line.

\begin{cor}
  \label{cor:4}
  Let $f\in \C[x]\setminus \C$ and set $d=\deg(f)$. Then, for $x\in
  \C$ with $|x|>\max_{\xi\in V(f)}|\xi|$, 
\begin{equation*}
\frac{1}{f(x)}=\sum_{l\in \N} \Res_{\A^{1}_{\C}}\Big[\begin{matrix} x^{d +l-1} \dd x \cr 
  f\end{matrix}\Big] \, {x^{-d-l}}.
\end{equation*}
\end{cor}

\begin{proof}
  We can write $f = (f_{d} +x^{-1}q(x^{-1})) x^{d} $ with
  $f_{d}\in \C^{\times}$ and $q\in \C[x^{-1}]$. Hence, the Laurent
  expansion of $1/f$ around the point at infinity is of the form
  \begin{displaymath}
    \frac{1}{f}=\sum_{l\in \N} c_{l}\,  x^{-d-l}
  \end{displaymath}
  with $c_{l}\in \C$.  Since $1/f$ is holomorphic for $x\in \C$ with
  $|x|>\max_{\xi\in V(f)}|\xi|$, this expansion is convergent on this
  region, and the expression for the $c_{l}$'s in terms of residues
  follows from Proposition~\ref{prop:12}.
\end{proof}

\subsection{Relationship to traces and division formulae in polynomial
  rings} \label{sec:relat-trac-divis}
Multivariate residue calculus is deeply related to the concept of
trace. In particular, traces over reduced $0$-dimensional $\C$-algebras can be
expressed in terms of residues. 

\begin{defn}
\label{def:6}
Let $K$ be a field, $L$ a finite-dimensional $K$-algebra, and
$q\in L$. The \emph{trace} of $q$, denoted by $\Tr_{L/K}(q)$, is
defined as the trace of the multiplication map $m_{g}\colon L\to L$ given by
$m_{q}(p)= q\cdot p$.
\end{defn}

Let $X\subset \A_{\C}^{n}$ be a  variety of pure
dimension $r\ge 1$ and $\bff=(f_1,...,f_r)$ a complete intersection on
$X$. Set
\begin{equation*}
\dd\bff =  \bigwedge_{i=1}^{r} \dd f_i.
\end{equation*}

\begin{prop}
\label{prop:14}
Suppose that  the finite-dimensional $\C$-algebra
$B=\C[\bfx]/(I(X)+(\bff))$ is reduced, and let $q\in \C(x_{1},\dots, x_{n})$ be a
rational function that is regular on $X\cap V(\bff)$. Then
$$
{\Tr}_{B/\C} (q)= \ResX
\Big[\begin{matrix} 
q \dd\bff \\
\bff \end{matrix} \Big].
$$  
\end{prop}

\begin{proof}
 Write $X\cap V(\bff)=\{\bfxi_{1},\dots, \bfxi_{L }\}$ with
$L =\# (X\cap V(\bff))$ and $\bfxi_{l}\in \C^{n}$. 
Since $B$ is reduced, the map
\begin{displaymath}
\psi\colon  B\longrightarrow \C^{L }, \quad
  q\longmapsto (q(\bfxi_{1}),\dots, q(\bfxi_{L }))
\end{displaymath}
is an isomorphism of $\C$-algebras. For $q\in B$, the matrix of the multiplication
map $m_{q}$ in the standard basis $\cS$ of $\C^{L }$ is diagonal, namely
\begin{math}
  (m_{q})_{\cS}=\diag(q(\bfxi_{1}),\dots, q(\bfxi_{L })) \in
  \C^{L \times L }.
\end{math}
Hence 
\begin{equation} \label{eq:33}
  {\Tr}_{B/\C} (q)= \sum_{l=1}^{L } q(\bfxi_{l}). 
\end{equation}

Set $Z(\bff)= X\cdot \prod_{j=1}^{r}\div(f_{j})$ for the 0-dimensional
intersection cycle of $\bff$ on $X$, and let 
$[Z(\bff)]$ be the integration current on this cycle. 
By  \cite[\S
1.9 and 3.6]{ColeffHerrera:crafm}, we have the currential identity 
\begin{displaymath}
\bigg(\bigwedge_{j=1}^r \overline\partial
\bigg(\frac{1}{f_j}\bigg)\wedge X\bigg) \wedge \dd \bff =
[Z(\bff)].
\end{displaymath}
Since $B$ is reduced, so is $Z(\bff)$ and hence
$ [Z(\bff)]= \sum_{l=1}^{L } \delta_{\bfxi_{l}}$, where
$ \delta_{\bfxi_{l}}$ denotes the Dirac delta measure at the point
$\bfxi_{l}$. By \eqref{eq:95},
\begin{equation*}
  \ResX
  \Big[\begin{matrix} 
    q \dd\bff \\
    \bff \end{matrix} \Big] = 
  \bigg\langle \bigwedge\limits_{j=1}^r \overline\partial 
  \bigg(\frac{1}{f_j}\bigg) \wedge [X], \chi \, q \dd
  \bff \bigg\rangle 
  =\int_{X} \chi \, q \dd[Z(\bff)]
  =\sum_{l=1}^{L} q(\bfxi_{l}) , 
\end{equation*}
where $\chi\colon \mathbb C^n \to \R$ is an arbitrary
$C^\infty$-function with compact support that is identically equal to
$1$ on a neighborhood of $X\cap V(\bff)$. The statement follows from
this equality together with \eqref{eq:33}.
\end{proof}

We can also consider traces of rational functions on $X$, which are rational
functions over the base space $\C^{r}$. The hypothesis that $\bff$ is
complete intersection over $X$ implies that the map $\varphi_{\bff}$
in \eqref{eq:23} is dominant and generically finite. Let
$K=\KK(\A_{\C}^{r})$ and $L=\KK(X)$ respectively denote the function
fields of $\A_{\C}^{r}$ and of $X$, and let 
$\varphi_{\bff}^{\#}\colon K\hookrightarrow L$ be the finite field
extension induced by this map. We identify $K$ with the field
$ \C(\bfy)$, where $ \bfy=(y_{1},\dots, y_{r})$ denotes a group of $r$
variables.

Let $g,h\in \C[\bfx]$ such that $h\notin I(X)$. Then $q=g/h$ is a
rational function on $X$, and the trace
\begin{equation}
  \label{eq:31}
  \Theta_{X,\bff,q}:=\Tr_{L/K}(q) \in K = \C(\bfy)
\end{equation}
is a rational function on $\C^{r}$. Under suitable hypothesis, the
Taylor expansion of this rational function can be computed in terms of
affine residues. Set
\begin{displaymath}
D=\bigg(\prod_{i=1}^{r}d_{i}\bigg) \deg(X)
\end{displaymath}
for the B\'ezout number of $\bff$ on $X$.

\begin{prop}
  \label{prop:15}
  Suppose that $\# (X\cap V(\bff))= D$ and let $g\in \C[\bfx]$.  Then
  $\Theta_{X,\bff, g} \in \C[\bfy]$ and
\begin{equation*}
  \Theta_{X,\bff,g}=\sum_{\bfalpha}{\ResX}
  \Big[\begin{matrix} g \dd \bff \cr
    \bff^{\bfalpha+\bfone}\end{matrix} \Big] \, \bfy^{\bfalpha},
\end{equation*}
the sum being over the vectors $\bfalpha \in \N^{r}$ such that
$ \sum_{j=1}^{r} \alpha_{j} \deg(f_{j}) \le \deg(g) $.
\end{prop}

We give the proof of this result after the next lemma. This lemma shows that, on
a nonempty open subset of $\C^{r}$, the function $\Theta_{X,\bff,g}$
can be computed in terms of traces over ``fiber'' algebras.

\begin{lem}
  \label{lemm:6}
Let  notation be as in Proposition \ref{prop:15} and set, for $\bfy\in \C^{r}$, 
  \begin{displaymath}
   B_{\bfy}=\C[\bfx]/(I(X)+(\bff-\bfy)). 
  \end{displaymath}
  Then there is nonempty open subset $U\subset \C^{r}$ such that, for
  $\bfy\in W$, the $\C$-algebra $ B_{\bfy}$ is reduced,
  $\dim_{\C}( B_{ \bfy})= D$, and
\begin{equation*} 
  \Theta_{X,\bff,q}(\bfy)=\Tr_{B_{\bfy}/\C}(q).
\end{equation*}
\end{lem}

\begin{proof}
The hypothesis that $  \# (X\cap V(\bff))= D$  is equivalent to the
fact that the fiber of $\varphi_{\bff}$ at the point $\bfzero\in \C^{r}$ has exactly
$D$ points. 
By B\'ezout's theorem, this fiber is reduced and moreover, there is a nonempty open
subset $U_{1}\subset\C^{r}$ with $\bfzero \in U_{1}$ such that, for $\bfy\in U_{1}$,
\begin{equation*}
  \# \varphi_{\bff}^{-1}(\bfy)= \#(X\cap V(\bff-\bfy))= D 
\end{equation*}
and $\varphi_{\bff}^{-1}(\bfy)$ is also reduced. 

For the third statement, let $\ell\in \C[\bfx]$ be a linear form such
that $\ell(\xi)\ne \ell(\xi')$ for all
$\xi,\xi'\in \varphi_{\bff}^{-1}(\bfzero)$ with $\xi\ne \xi'$. Then
$\cB=(\ell^{k})_{0\le k\le D-1}$ gives a basis for the $\C$-algebra
$B_{\bfzero}$.

Let $U_{2}\subset U_{1}$ be nonempty open subset  $\bfzero \in U_{2}$
where the fibers of
$\varphi_{\bff}$ have cardinality $D$ and the linear form $\ell$
separates the point of these fibers. Similarly, the collection $\cB$
is a $\C$-basis of the fiber algebra $B_{\bfy}$ for all
$\bfy\in U_{2}$. In particular, this also implies that $\cB$ is a
$K$-basis of the $K$-algebra $L$.

For $0\le j,k\le D-1$, write 
\begin{equation}\label{eq:45}
  \ell^{j} \cdot \ell^{k}=\sum_{l}\gamma_{j,k,l} \, \ell^{l}
\end{equation}
with $\gamma_{j,k,l}\in K$. Then we choose $U$ as any nonempty open
subset of $U_{2}$ such that $\bfzero \in U$ and where all the rational
functions $\gamma_{j,k,l}$ are regular. Hence, the relations in
\eqref{eq:45} also hold in the $\C$-algebra $B_{\bfy}$, for
$\bfy\in U$.

Let $M_{q}=(m_{q})_{\cB}\in K^{D\times D}$ be the matrix of the
multiplication map of $q$ over $L$ with respect to this basis. For
$\bfy\in U$, this matrix specializes into the matrix of the
multiplication map of $q$ over $B_{\bfy}$ with respect to the basis
$\cB$. Hence
\begin{displaymath}
  \Theta_{X,\bff,q}(\bfy)=\Tr (M_{q}(\bfy))= \Tr_{B_{\bfy}/\C}(q),   
\end{displaymath}
as stated.
\end{proof}

\begin{proof}[Proof of Proposition \ref{prop:15}]
  By Lemma~\ref{lemm:6} and Proposition \ref{prop:14}, there is a
  nonempty open subset $U\subset \C^{r}$ with $\bfzero \in U$ such
  that, for $\bfy\in U$,
\begin{equation*}
    \Theta_{X,\bff,g}(\bfy)= {\ResX}
  \Big[\begin{matrix} g \dd \bff \cr \bff-\bfy\end{matrix} \Big].
\end{equation*}
Hence, the rational function $\Theta_{X,\bff,g}$ is regular at
$\bfzero\in\C^{r}$, and we can consider its Taylor expansion around
this point.

Since the set-valued function $\bfy\mapsto X\cap V(\bff-\bfy)$ varies
continuously on a neighborhood of $\bfzero\in \C^{r}$, there exist
$R>0$ and $\eta>0$ such that
$X \cap V(\bff-\bfy)\subset X\, \cap \, \ball_R$ for all
$\bfy\in \C^{r}$ with $\|\bfy\|<\eta$.  By Proposition \ref{prop:10},
\begin{displaymath}
  {\ResX}
  \Big[\begin{matrix} g \dd \bff \cr \bff-\bfy\end{matrix} \Big]= 
 \frac{(-1)^{r(r-1)/2} (r-1)!} {(2i\pi)^r} 
 \int_{ \partial (X\cap \ball_R)} 
 \frac{\Omega_{\bff-\bfy}}{\|\bff - \bfy\|^{2r}}\wedge g \dd \bff   
\end{displaymath}
with 
  \begin{math}
    \Omega_{\bff-\bfy}= \sum_{j=1}^r
  (-1)^{j-1} (\overline{ f_j}-\overline{y_{j}}) \bigwedge_{{l =1 \atop l \not=j}}^r
  \overline {\dd f_l}.
  \end{math}
  Differentiating this identity, it follows from Lebesgue's
  differentiation theorem that
  \begin{displaymath}
\frac{\partial^{|\bfalpha|} \Theta_{X,\bff,g}}{\partial \bfy^{\bfalpha}} 
=
\frac{(-1)^{(r(r-1)/2} (r+|\bfalpha|-1)!}{(2i\pi)^r}\,  
\int_{\partial (X \cap  B_R)} 
(\overline\bff  - \overline\bfy)^{\bfalpha} \frac{\Omega_{\bff-\bfy}}{\|\bff\|^{2(r+|\bfalpha|)}}\wedge g\dd \bff.      
  \end{displaymath}
Evaluating this identity at  $\bfy =\bfzero$, we get from the integral representation in Proposition \ref{prop:10} that
\begin{equation}
  \label{eq:94}
\frac{\partial^{|\bfalpha|} \Theta_{X,\bff,g}}{\partial \bfy^{\bfalpha}} (\bfzero) = 
\bfalpha! \ResX \Big[\begin{matrix} g \dd \bff \\ \bff^{\bfalpha + \bfone} \end{matrix} \Big]. 
\end{equation}

The hypothesis that $\# (X\cap V(\bff))= D$ implies that the system
$\bff$ on $X$ has no zeros at infinity. By \eqref{eq:71} and the
Lagrange-Jacobi theorem~\ref{jacobi}, the residues in \eqref{eq:94}
vanish for $\bfalpha \in \N^{r}$ such that
\begin{equation*}
 \deg(g \dd \bff) < \bigg( \sum_{j=1}^{r} (\alpha_{j}+1) \deg(f_{j}) \bigg) - r .  
\end{equation*}
We have that
$\deg(g \dd \bff)= \deg(g) + \big( \sum_{j=1}^{r} \deg(f_{j})\big)
-r$.
Hence, the residues in \eqref{eq:94} vanish whenever
$ \sum_{j=1}^{r} \alpha_{j} \deg(f_{j}) > \deg(g) $, which finishes
the proof.
\end{proof}

Residues also play an important role in division formulae in polynomial
rings. An example of this connection is the Bergman-Weil trace formula
for the case when $X=\A_{\C}^{n}$, see \cite[II \S9]{AizYu:gnus} and
\cite[IV]{Tsikh:tsikh}, or \cite{BoyHi2:gnus} for an extended
bibliography on this subject as well as a presentation of Weil's
formula and Bergman-Weil's developments within an algebraic setting.

To describe this formula, let $\bfz=(z_{1},\dots,
z_{n})$ be a group of variables, fix $1\le i\le n$, let $h_{i,j}\in
\C[\bfx,\bfz]$, $1\le j\le n$, be a family of $n$ polynomials such that
\begin{equation}\label{eq:105}
  f_{i}(\bfz)-f_{i}(\bfx)= \sum_{j=1}^{n}h_{i,j}(\bfx,\bfz)
  (z_{k}-x_{k}),
\end{equation}
and set $h_{i}=\sum_{j=1}^{n}h_{i,j}\dd z_{j}$.  This is a polynomial
1-form in the variables $\bfz$ whose coefficients are polynomials in
$\C[\bfx]$.

\begin{thm}
  \label{thm:3}
With notation as above, let $p\in \C[\bfx]$ and choose $R>0$ such that
$X\, \cap \, V(f)$ is
contained in the ball $ \ball_R$. Then, for $\bfx\in \C^{n}$ such that 
$\|\bff(\bfx)\|<\min_{\bfy\in X\, \cap\, \partial B_R} \|\bff(\bfy)\|$, 
\begin{equation}\label{eq:102}
p(\bfx) 
= \sum_{\bfalpha \in \N^n} \Res_{\mathbb A^n_\C}
\Bigg[\begin{matrix}  p(\bfz) \,  
\bigwedge\limits_{i=1}^n h_i(\bfx,\bfz)  \\
f_1(\bfz)^{\alpha_1+1},...,f_n(\bfz)^{\alpha_n+1}\end{matrix}\Bigg] \, f(\bfx)^{\bfalpha}.  
\end{equation}
\end{thm}

When the map $\varphi_{\bff}$ in \eqref{eq:23} is proper, the
Lagrange-Jacobi theorem \ref{jacobi} implies that all but a finite
number of residues in the expansion \eqref{eq:102} vanish. Hence in
the proper case, this expansion becomes a polynomial identity.

As an application  of the Bergman-Weil formula, we can express the
coefficients of the $f$-adic expansion of a univariate polynomial in
terms of residues on the affine line. 

\begin{defn} \label{def:1}
Given $f\in \C[x]\setminus \C$,
the \emph{$f$-adic expansion} of a polynomial $p\in \C[x]$ is its
unique finite representation as 
\begin{equation*}
p =\sum_{\alpha\in \N}p_{f,\alpha}f^{\alpha}
\end{equation*}
with $p_{f,\alpha}\in \C[x]$ satisfying $\deg(p_{f,\alpha})\le \deg(f) -1$ for all
$\alpha$.  
\end{defn}

\begin{cor}
  \label{cor:5} With notation as in Definition \ref{def:1}, 
  coefficients of the $f$-adic expression of $p$ are given, for
  $\alpha\in \N$, by
\begin{equation}
\label{eq:49}
  p_{f,\alpha}(x)=   \Res_{\A^{1}_{\C}}  \Bigg[\begin{matrix}
\displaystyle{{p(z)}\, \frac{f(z)-f(x)}{z-x}  \dd z }\cr 
\displaystyle{f(z)^{\alpha+1}}\end{matrix}\Bigg].
\end{equation}
\end{cor}

\begin{proof} 
For $n=1$, the formula in Theorem \ref{thm:3} reduces to 
\begin{equation}
\label{eq:104}
p(x) =\sum_{\alpha\in \N}\Res_{\A^{1}_{\C}}  \Bigg[\begin{matrix}
\displaystyle{{p(z)}\, \frac{f(z)-f(x)}{z-x}  \dd y }\cr 
\displaystyle{f(z)^{\alpha+1}}\end{matrix}\Bigg] f(x)^{\alpha}.
\end{equation}

Set $d=\deg(f)$. The quotient $(f(z)-f(x))/(z-x)$ is a polynomial in
$\C[x,y]$ of degree bounded by $d-1$, and so the residues in the
right-hand side of \eqref{eq:49} are also polynomials in $\C[x]$ of
degree bounded by $d-1$.

By the Lagrange-Jacobi theorem \ref{jacobi}, these residues vanish when
$\deg(p)+d-1<(\alpha+1) d -1$ or, equivalently, when $\deg(p)<\alpha
d$. Hence, the representation \eqref{eq:104} is finite, and so it
gives the $f$-adic expansion of $p$, as stated. 
\end{proof}

\section[Residues on the affine line]{Residues on the affine line} \label{sec:1-dimensional-case}

Here we consider the problem of bounding the residues of polynomial
and rational 1-forms on the affine line. More precisely, let
$f \in \Z[x]\setminus \Z$ and $g,h\in \Z[x]$ with $h$ coprime with
$f$.  For $\alpha\in \N$, the residue
\begin{equation*}
 \Res_{\A^{1}_{\C}} \Big[\begin{matrix} g/h  \dd x \cr f^{\alpha+1}\end{matrix}\Big]
\end{equation*} 
is a rational number, and we want to bound its numerator and
denominator. In this section, we only consider residues of this type
and, for ease of notation, we omit the variety $\A^{1}_{\C}$ when
denoting them.

\begin{defn}
  \label{def:3} Let $f=\sum_{i=0}^{d}f_{i}x^{i}\in \Z[x]$. The
  \emph{(logarithmic) height} and the \emph{(logarithmic) length} of $f$ are respectively defined as
  \begin{displaymath}
     \h(f)= \log\big(\max_{0\le i\le d} |f_{i}|\big) \and \h_{1}(f) = \log\bigg( \sum_{i=0}^{d}|f_{i}|\bigg).
  \end{displaymath}
\end{defn}

These quantities are related by the inequalities
\begin{displaymath}
  \h(f)\le \h_{1}(f)\le \h(f) +\log(d+1).
\end{displaymath}
The length is submultiplicative, in the sense that, for
$f_{1},f_{2}\in \Z[x]$, 
\begin{equation*}
  \h_{1}(f_{1}f_{2})\le \h_{1}(f_{1})+\h_{1}(f_{2}).
\end{equation*}

The following is the main result of this section. It bounds the
numerators and denominators of the residue sequence of to a
polynomial $1$-form on the affine line.

\begin{thm}
\label{thm:4}  
Let $f\in \Z[x]\setminus \Z$ and $g\in \Z[x]$. Set
$d=\deg(f)$ and $e =\deg(g)$, and let $f_{d}$ be the leading
coefficient of $f$.  Then, for $\alpha \in \N$,
\begin{equation*}
  f_d^{e+1- (\alpha+1)(d-1)}  \Res  \Big[\begin{matrix} {g} \dd x \cr f^{\alpha+1}\end{matrix}\Big] \in \Z
\end{equation*}
and
\begin{equation*}
   \log \Big|f_d^{e+1 - (\alpha+1)(d-1)}   \Res  \Big[\begin{matrix} g \dd x \cr 
f^{\alpha+1}\end{matrix}\Big]\Big| 
 \leq \h_1(g) + (e +1-(\alpha+1)d)\h(f) + (e -d+1)  \log (2) .
\end{equation*}
If $e < (\alpha+1) d -1$, 
then $ \Res  \Big[\begin{matrix} g \dd x \cr 
f^{\alpha+1}\end{matrix}\Big] =0$. 
\end{thm}

We give the proof of this theorem  after some preliminary results. 
  Let $f\in \C[x]\setminus \C$ and, for $j,\alpha\in
  \N$, set 
  \begin{equation}
    \label{eq:21}
    \varrho_f(j,\alpha)=\Res \Big[\begin{matrix} x^{j}
    \dd x \cr f^{\alpha+1}\end{matrix}\Big].
  \end{equation}
Set also $\varrho_f(j,-1)=0$ for all $j\in \N$. 

\begin{prop}
  \label{prop:2}
  With notation as in \eqref{eq:21}, write
  $f=\sum_{i=0}^{d}f_{i}x^{i}$ with $d=\deg(f)$ and $f_{i}\in
  \C$. Then, for all $j,\alpha \in \N$,
\begin{equation}
  \label{eq:1}
  \sum_{i=0}^{d} f_{i} \, \varrho(j+i,\alpha) 
  =  \varrho(j,\alpha-1).
\end{equation}
\end{prop}

\begin{proof}
  By linearity of the residue, the sum in the left-hand side of
  \eqref{eq:1} is equal to
\begin{math}
    \Res  \Big[\begin{matrix} f x^{j} \dd x \cr f^{\alpha+1}\end{matrix}\Big].
\end{math} 
The identity
$ \Res \Big[\begin{matrix} f x^{j} \dd x \cr
  f^{\alpha+1}\end{matrix}\Big] = \varrho(j,\alpha-1)$
follows easily from the definition of the residue.
\end{proof}

Using the formula in the previous proposition, we obtain a recursive
algorithm for computing the residue sequence of a monomial $1$-form. 

\begin{prop}
  \label{prop:3}
Let notation be as in Proposition
  \ref{prop:2}. Then, for all $j,\alpha\in \N$,
\begin{equation*}
  \varrho_f(j,\alpha)= 
  \begin{cases}
    0 & \text{ if } j\le (\alpha+1) d  -2,\\
    f_{d}^{-\alpha-1} & \text{ if } j= (\alpha+1) d  -1,\\
\varrho_f(j-d,\alpha-1) -\sum_{i=1}^{d} f_{d}^{-1}{f_{d-i}} \, \varrho_f(j-i,\alpha)   & \text{ if } j\ge (\alpha+1) d.
  \end{cases}
\end{equation*}
\end{prop}

\begin{proof}
  By Proposition \ref{prop:12}, the residue $ \varrho_f(j,\alpha)$
  coincides with the coefficient of degree~$-1$ in the expansion of
  $ x^{j}/f^{\alpha+1}$ as a Laurent series around the point at
  infinity. Write $f = (f_{d} +x^{-1}q ) x^{d} $ with
  $f_{d}\in \C^{\times}$ and $q\in \C[x^{-1}]$. Hence
\begin{displaymath}
  \frac{x^{j}}{f^{\alpha+1}} = 
  {f_{d}^{-\alpha-1}} x^{j-(\alpha+1) d} + \text{higher order terms},
\end{displaymath}
which implies first and second equalities in the statement. The third
equality follows from Proposition \ref{prop:2}.
\end{proof}

When $f$ has integral coefficients, we can apply the recursive
formulae in Proposition~\ref{prop:3} to bound the numerator and the
denominator in the residue sequence of a monomial $1$-form.

\begin{prop}
  \label{prop:4}
  With notation as in \eqref{eq:21}, suppose that
  $f\in \Z[x]\setminus \Z$, and denote by $d$ and $f_{d}$ its degree
  and leading coefficient. Then, for all $j,\alpha\in \N$,
\begin{equation} 
\label{eq:3bis} 
f_d^{j+1 -(\alpha+1) (d-1)} \varrho_{f}(j,\alpha) \in \Z
\end{equation} 
and 
\begin{equation}\label{eq:3q} 
 \log \big|f_d^{j+1 -(\alpha+1) (d-1)} \varrho_{f}(j,\alpha) \big| \\
 \leq (j +1 - (\alpha+1) d)\h(f) + (j-d+1)  \log (2). 
\end{equation} 
If $j <(\alpha+1) d-1$, then $ \varrho_{f}(j,\alpha) =0$.
\end{prop}

\begin{proof}
  Set
  $\wt \varrho_f(j,\alpha)=f_d^{j+1 -(\alpha+1) (d-1)} \varrho_f(j,\alpha)$.
  The recursive formulae in Proposition~\ref{prop:3} then translate into
  the relations, for  $j,\alpha\in \N$,
\begin{equation*}
  \wt\varrho_f(j,\alpha)= 
  \begin{cases}
    0 & \text{ if } j\le (\alpha+1) d -2,\\
1  & \text{ if } j= (\alpha+1) d  -1,\\
\wt \varrho_f(j-d,\alpha-1) -\sum_{i=1}^{d}{f_{d}^{i-1}f_{d-i}} \, \wt \varrho_f(j-i,\alpha)   & \text{ if } j\ge (\alpha+1) d.
  \end{cases}
\end{equation*}

Set $H=\e^{\h(f)} = \max_{i}|f_{i}|$. The statements \eqref{eq:3bis}
and \eqref{eq:3q} 
amount to the conditions
\begin{displaymath}
  \wt\varrho_f(j,\alpha)\in \Z \and |\wt\varrho_f(j,\alpha)| \le
  2^{j-d+1}H^{j+1 - (\alpha+1) d}.
\end{displaymath}
We prove them by induction on the quantity
$j+\alpha$. If $j+\alpha=0$, then $j=\alpha=0$ and both statements
follow from the recursive formulae above.  

Suppose that $j+\alpha\ge 1$. If $j\le (\alpha+1)d-1$, this also
follows from these formulae. So we assume that $j\geq (\alpha+1)d$.
By the inductive hypothesis, $\wt \varrho_{f}(j,\alpha)\in \Z$, which
gives~\eqref{eq:3bis}. For the bound \eqref{eq:3q}, we use again the
inductive hypothesis and these formulae to obtain that
\begin{align*}
|\wt\varrho_f(j,\alpha)| & \le |\wt \varrho_f(j-d,\alpha-1)| +\sum_{i=1}^{d}
H^{i} \, |\wt \varrho_f(j-i,\alpha)| \\
& \le 2^{j-2d+1}H^{j-d+1-\alpha d} +\sum_{i=1}^{d} 
H^{i} \, 2^{(j-i)-d+1} H^{(j-i) +1-(\alpha+1) d}\\
& = \Big( 2^{j-2d+1} +\sum_{i=1}^{d} 
2^{(j-i)-d+1} \Big) H^{j+1-(\alpha+1) d} \\
&= 2^{j-d+1}H^{j  +1-(\alpha+1) d}.
\end{align*}
The last statement follows directly from Proposition \ref{prop:3} or,
alternatively, it can be derived from the Lagrange-Jacobi vanishing
theorem \ref{jacobi}.
\end{proof}

\begin{proof}[Proof of Theorem \ref{thm:4}]
  Write $g=\sum_{j=0}^{e}g_{j}x^{j}$ with $g_{j}\in \Z$. By linearity,
\begin{displaymath}
\Res  \Big[\begin{matrix} {g} \dd x \cr
  f^{\alpha+1}\end{matrix}\Big]
=\sum_{j=0}^{e}g_{j} \Res  \Big[\begin{matrix} {x^{j}} \dd x \cr
  f^{\alpha+1}\end{matrix}\Big].  
\end{displaymath}
Hence, with the notation in the proof of Proposition \ref{prop:4},
\begin{equation}\label{eq:38}
  f_d^{e+1- (\alpha+1) (d-1)}  \Res  \Big[\begin{matrix} {g}
    \dd x \cr f^{\alpha+1}\end{matrix}\Big]
=\sum_{j=0}^{e}f_{d}^{e-j}g_{j}\, \wt \varrho_f(j,\alpha).  
\end{equation}
It follows from this proposition that this quantity lies in
$\Z$, proving the first statement. 
By the same result, 
\begin{multline*}
\bigg|\sum_{j=0}^{e}f_{d}^{e-j}g_{j}\, \wt \varrho_f(j,\alpha) \bigg | \le
\bigg(\sum_{j=0}^{e}|g_{j}|\bigg)  \max_{j} \big( |f_{d}|^{e-j}
2^{j-d+1}H^{j+1 -(\alpha+1) d}\big) \\ \le 
\bigg(\sum_{j=0}^{e}|g_{j}|\bigg)  
2^{e-d+1}H^{e+1-(\alpha+1) d},
\end{multline*}
with $H=\e^{\h(f)}= \max_{i}|f_{i}|$. The second statement follows
from this and~\eqref{eq:38}. The last statement follows from
\eqref{eq:38} and Proposition \ref{prop:4}.
\end{proof}

The bounds in Theorem \ref{thm:4} are essentially optimal, as shown by
the next example.

 \begin{exmpl}
   \label{exm:1}
Let $d\ge 1$, $\alpha\ge 0$,  $e \ge (\alpha+1) d$,  $H_{2}\ge
H_{1}\ge 1$ and $H_{3}\ge 1$. Set    
\begin{displaymath}
     f=H_{1}x^{d}-H_{2}x^{d-1},  \quad  g=H_{3}x^{e } \and \rho=   \Res  \Big[\begin{matrix} g \dd x \cr 
f^{\alpha+1}\end{matrix}\Big].
   \end{displaymath}
We have that 
\begin{displaymath}
\rho=   \Res  \Big[\begin{matrix} H_{3}x^{e} \dd x \cr 
(x^{d-1}(H_{1}x-H_{2}))^{\alpha+1}\end{matrix}\Big] = 
H_{3}  \Res  \Big[\begin{matrix} x^{e-(\alpha+1) (d -1)} \dd x \cr 
H_{1}x-H_{2}\end{matrix}\Big].
\end{displaymath}
Making the change of variables $y=H_{1}x-H_{2}$ and applying
Propositions \ref{changebasis} and \ref{prop:12}, we get
\begin{multline}
\label{eq:51}
\rho=
H_{3}   \Res  \Big[\begin{matrix} (H_{1}^{-1}(y+H_{2}))^{e-(\alpha+1)(d-1)} H_{1}^{-1}\dd y \cr 
y^{\alpha+1}\end{matrix}\Big] \\ =
{e- (\alpha+1) d \choose \alpha } \frac{H_{3}\, H_{2}^{e+1 -(\alpha+1)    d}}{H_{1}^{e+1 -(\alpha+1) (d-1)}}.
\end{multline}

We have that $\deg(f)=d$, $\h(f)=\log( H_{2})$, $\deg(g)=e $, and
$\h_{1}(g)=\log(H_{3})$. In this case,  Theorem
\ref{thm:4} says that
\begin{displaymath}
H_{1}^{e+1-(\alpha+1) (d-1)}\rho \in \Z \and 
\big|H_{1}^{e+1-(\alpha+1) (d-1)}\rho \big| \le 2^{e-d+1} H_{3}\,
H_{2}^{e+1 -(\alpha+1)    d},
\end{displaymath}
which can be compared with the explicit expression for $\rho$ in
\eqref{eq:51}.
\end{exmpl}

As an application of these results, we derive a bound for the
coefficients of the Laurent expansion around the point at infinity of
the inverse of a polynomial. For a polynomial $f\in \Z[x]\setminus \Z$
of degree $d$ and $\alpha\in \N$, by Corollary \ref{cor:4} the Laurent
series
\begin{equation}\label{eq:65}
{f^{-\alpha-1}}=\sum_{l\in \N}
{c_{f,\alpha,l}}\, {x^{-(\alpha+1)d-l}}
\end{equation}
 is convergent when $|x|>\max_{\xi\in V(f)}|\xi|$, and its
coefficients can be expressed in terms of residues on the affine line as
\begin{equation}
  \label{eq:66}
  c_{f,\alpha,l}= \Res  \Big[\begin{matrix} x^{(\alpha+1) d +l-1} \dd x \cr 
f^{\alpha+1}\end{matrix}\Big].
\end{equation}

\begin{cor}
  \label{cor:2}
  Let $f\in \Z[x]\setminus \Z$ and $\alpha\in \N$.  Set $d=\deg(f)$
  and let $f_{d}$ be the leading coefficient of $f$. Then the
  coefficients of the Laurent expansion of $f^{-\alpha-1}$
  in~\eqref{eq:65} satisfy
\begin{equation*}
  f_{d}^{l+\alpha+1} c_{f,\alpha,l}\in \Z \and \log|
  f_{d}^{l+\alpha+1} c_{f,\alpha,l}|\le l\h(f)+(l+\alpha d)\log(2).
\end{equation*}
\end{cor}

\begin{proof}
  This follows  directly from the formula \eqref{eq:66} and Proposition
  \ref{prop:4}.
\end{proof}

As a second application, we bound the coefficients of the $f$-adic
expansion of a polynomial (Definition \ref{def:1}). 

\begin{prop}
  \label{prop:5}
  Let $f\in \Z[x]\setminus \Z$ and $p\in \Z[x]$.  Set $d=\deg(f)$ and
  $e=\deg(p)$, and let $f_{d}$ be the leading coefficient of
  $f$. Then
\begin{equation*}
p =\sum_{\alpha=0}^{\lfloor e/d\rfloor}p_{f,\alpha}f^{\alpha}
\end{equation*}
with  $p_{f,\alpha}\in \Q[x]$ such that $\deg(p_{f,\alpha})\le d -1$ 
and
\begin{displaymath}
  f_{d}^{e+1-\alpha (d-1)} p_{f,\alpha}\in \Z[x] \and \h_{1}(  f_{d}^{e+1-\alpha (d-1)} p_{f,\alpha})\le \h_{1}(p)+
(e-\alpha d)\h_{1}(f)+ e\log (2), 
\end{displaymath}
for all $\alpha$. 
\end{prop}

\begin{proof}
  For $\alpha\in \N$, let
  $p_{f,\alpha}\in \Q[x]$ be the
  $\alpha$-th coefficient in the $f$-adic expansion of $p$.  Clearly,
  $p_{f,\alpha}=0$ for $\alpha>e/d$.

  For $0\le \alpha \le e/d$, write
  $p_{f,\alpha}=\sum_{i=0}^{d-1}p_{f,\alpha,i}x^{i}$ and
  $f=\sum_{i=0}^{d}f_{i} x^{i}$. Then
\begin{equation*}
  \frac{f(z)-f(x)}{z-x}= \sum_{i=0}^{d}f_{i}\frac{z^{i}-x^{i}}{z-x}=
  \sum_{i=0}^{d}q_{i}(z)x^{i}
\end{equation*}
with $q_{i}=\sum_{l=0}^{d-1-i} f_{i+l+1}z^{l}$.  By
Corollary~\ref{cor:4},
$p_{f,\alpha,i}= \Res \Big[\begin{matrix} \displaystyle{{p} \, q_{i}
    \dd z }\cr \displaystyle{f^{\alpha+1}}\end{matrix}\Big]$
for $i=0,\dots, d-1$.  We have that $\deg(q_{i})\le d-i-1$ and
$\h_{1}(q_{i})\le \h_{1}(f)$.  Hence $\deg(p\, q_{i})\le e+d-i-1$ and
$\h_{1}(p\, q_{i})\le \h_{1}(p)+\h_{1}(f)$. By Theorem \ref{thm:4},
$ f_{d}^{e+1-\alpha (d-1)}p_{f,\alpha,i}\in \Z$ and
\begin{equation*}
\log\big|f_{d}^{e+1-\alpha (d-1)}p_{f,\alpha,i} \big|\le
\h_{1}(p)+ \h_{1}(f)+
(e-i-\alpha d)\h(f)+ (e-i)\log (2),
\end{equation*}
as stated. 
\end{proof}

To extend the bounds in Theorem \ref{thm:4} to rational functions, we
need the following particular case of the arithmetic Nullstellensatz.
Given two polynomials $f_0,f_1\in \Z[x]$ of respective degrees $d_{0}$
and $d_{1}$, we denote by $S(f_0,f_1) \in \Q^{(d_{0}+d_{1})\times
  (d_{0}+d_{1})}$ their Sylvester matrix, and by
$\sigma(f_{0},f_{1})=\det(S(f_{0},f_{1}))$ their Sylvester resultant.
If $f_{0}$ and $f_{1}$ are coprime, this Sylvester resultant is
nonzero.

\begin{lem}
  \label{lemm:1}
Let $f_0,f_1\in \Z[x]$ be coprime polynomials. Then  there exist
$p_{0},p_{1} \in \Z[x]$ such that
\begin{displaymath}
  \sigma(f_{0},f_{1})=p_0 f_0 + p_1 f_1
\end{displaymath}
satisfying, for $i=0,1$, 
\begin{enumerate}
\item \label{item:10}   $\deg(p_{i}) + \deg(f_{i}) \le \deg(f_{0})+\deg(f_{1})-1$, 
\item \label{item:13} $\h_{1}(p_{i}) + \h_{1}(f_{i}) \le  \deg(f_{1})\h_{1}(f_{0})+\deg(f_{0}) \h_{1}(f_{1})$.
\end{enumerate}
Moreover, $\log|\sigma(f_{0},f_{1})| \le
\deg(f_{1})\h_{1}(f_{0})+\deg(f_{0}) \h_{1}(f_{1})$.
\end{lem}

\begin{proof}
  Set $d_{i}=\deg(f_{i})$ for short.  Since $f_{0}$ and $f_{1}$ are
  coprime, there are unique $q_{0}, q_{1}\in \Q[x]$ with $\deg(q_{0})
  \le d_{1}-1$ and $\deg(q_{1}) \le d_{0}-1$ such that
  \begin{equation}
    \label{eq:35}
1=q_{0}f_{0}+q_{1}f_{1}.    
  \end{equation}
This B\'ezout identity translates into a square system of linear
equations over~$\Q$. Such system can be written  down as 
$$
S(f_0,f_1)\cdot Q= b
$$
with $b\in \Q^{d_{0}+d_{1}}$ a vector with an entry equal to 1, that
corresponding to the constant term in \eqref{eq:35}, and all others
ones equal to 0, and $Q\in \Q^{d_{0}+d_{1}}$ the vector of
coefficients of $q_{0}$ and $q_{1}$.

Set $p_{i}=\sigma(f_{0},f_{1}) q_{i}$, $i=0,1$.  The upper bound for
the degree of $p_{i}$ follows from this construction. By Cramer's
rule, the coefficients of $p_{i}$ are $(d_{0}+d_{1}-1)$-minors of
$S(f_0,f_1)$. The upper bounds for $\sigma(f_{0},f_{1}) $ and for the
length of the $p_{i}$'s can be verified by analyzing these minors.
\end{proof}

\begin{thm}
  \label{thm:5}
  Let $f\in \Z[x]\setminus \Z$ and $f_{0}, g \in \Z[x]$ with $f_{0}$
  coprime with $f_{0}$. Set $d=\deg(f)$, $d_{0}=\deg(f_{0})$ and
  $e=\deg(g)$, and let $f_{d}$ be the leading coefficient of~$f$.
  Then, for $\alpha \in \N$,
\begin{equation*}
\sigma(f,f_{0})^{\alpha+1}  f_{d}^{e+ \alpha +1 }  \Res  \Big[\begin{matrix}
g/f_{0} \dd x \cr 
f^{\alpha+1}\end{matrix}\Big] \in \Z
\end{equation*}
and
\begin{multline}
\label{eq:37}
   \log \Big|
\sigma(f,f_{0})^{\alpha+1}  f_d^{e+ \alpha +1 }  \Res  \Big[\begin{matrix}
g/f_{0} \dd x \cr 
f^{\alpha+1}\end{matrix}\Big]\Big| 
 \leq \h_1(g) + ((\alpha+1) d -1)\h_{1}(f_{0}) 
\\ +(e +(\alpha+1) d_{0})\h_{1}(f) + (e +\alpha d)  \log (2) .
\end{multline}
\end{thm}

\begin{proof}
The polynomials $f_{0}$ and $f^{\alpha+1}$ are coprime of degrees
$d_{0}$ and $(\alpha+1)d$, and lengths $\h_{1}(f_{0})$ and
$\h_{1}(f^{\alpha+1}) \le (\alpha+1)\h_{1}(f)$. 

By the multiplicativity of the Sylvester resultant, 
\begin{displaymath}
\sigma(f^{\alpha+1},f_{0})=
\sigma(f,f_{0})^{\alpha+1}.
\end{displaymath}
Set $\gamma= \sigma(f, f_{0})$ for short.  By Lemma \ref{lemm:1},
there are $p_{\alpha,0},p_{\alpha,1}\in \Z[x]$ such that
\begin{equation}
  \label{eq:39}
  \gamma^{\alpha+1}=p_{\alpha,0}f_{0}+p_{\alpha,1}f^{\alpha+1}
\end{equation}
with $\deg(p_{\alpha,0}) \le (\alpha+1)d-1 $ and
$\h_{1}(p_{\alpha,0})\le ((\alpha+1)d
-1)\h_{1}(f_{0})+(\alpha+1)d_{0}\h_{1}(f)$. By \eqref{eq:39}, we have the
congruence on the open subset $\A_{\Q}^{1}\setminus V(f_{0})$ 
\begin{displaymath}
  \frac{g}{f_{0}}\equiv \frac{p_{\alpha,0}g}{\gamma^{\alpha+1}}
  \pmod{f^{\alpha+1}}.
\end{displaymath}
Since this open subset is a neighborhood of $V(f)$, by Proposition \ref{prop:16},
\begin{equation}\label{eq:40}
\Res  \Big[\begin{matrix}
g/f_{0} \dd x \cr 
f^{\alpha+1}\end{matrix}\Big]
=\frac{1}{\gamma^{\alpha+1}} \Res  \Big[\begin{matrix} {p_{\alpha,0}\,g}  \dd x \cr 
f^{\alpha+1}\end{matrix}\Big].
\end{equation}
Using that $\deg(p_{\alpha,0}\,g) \le (\alpha+1) d-1+e$, it follows from Theorem
\ref{thm:4} that $\gamma^{\alpha+1} f_d^{e+ \alpha +1 } $ is a
denominator for the residue in the left-hand side of
\eqref{eq:40}.
Similarly, the  bound in \eqref{eq:37}  follows also from Theorem
\ref{thm:4}, using that
\begin{displaymath}
\h_{1}(p_{\alpha,0}\, g)\le
\h_{1}(p_{\alpha,0})+\h_{1}(g) \le
((\alpha+1)d
-1)\h_{1}(f_{0})+(\alpha+1)d_{0}\h_{1}(f)+\h_{1}(g).
\end{displaymath}
\end{proof}      

\section{Residues on an affine variety:  polynomials in separated variables} \label{sec:tensor univariate case}

In this section, we bound the residues on an affine variety with
respect to a family of univariate polynomials in separated variables.

We first extend the different notions of size of polynomials to the
multivariate case. As in \S~\ref{sectionresidue}, we denote by $\bfx$
the group of variables $(x_1,\dots,x_n)$. For $f\in \C[\bfx]$, we
adopt the usual notation
$$
f=\sum_\bfalpha f_\bfalpha\bfx^\bfalpha$$
where, for each index $\bfalpha=(\alpha_1,\dots,\alpha_n)\in \N^n$,
$f_{\bfalpha}$ denotes an element of $ \C$ and $\bfx^\bfalpha$ the
monomial $x_1^{\alpha_1}\dots x_n^{\alpha_n}$.  The \emph{support} of
$f$, denoted by $\supp(f)$, is the finite subset $\bfalpha$'s such
that $f_{\bfalpha} \ne 0$. For each $\bfalpha\in \N^{n}$, we set
$|\bfalpha|=\sum_{i=1}^{n}\alpha_{i}$ for its length and
$\coeff_{\bfalpha}(f)=f_{\bfalpha}$ for the coefficient of
$\bfx^{\bfalpha}$ in the monomial expansion of $f$.  For
$\bfalpha,\bfbeta\in \N^{r}$, we denote by
$\langle \bfalpha,\bfbeta\rangle = \sum_{i=1}^{r} \alpha_{i}\beta_{i}$
their scalar product.

\begin{defn}
  \label{def:4}
  Let $f=\sum_{\bfalpha}f_{\bfalpha}\bfx^{\bfalpha}\in \Z[\bfx]$.  The
  \emph{(logarithmic) height},  \emph{length} and 
  \emph{Mahler measure} of $f$ are respectively defined
  as
  \begin{displaymath}
     \h(f)= \log(\max_{\bfalpha} |f_{\bfalpha}|) , \quad  \h_{1}(f) = \log\Big(
     \sum_{\bfalpha}|f_{\bfalpha}|\Big) \and
\m(f)= \int_{(S^{1})^{n}}\log|f|\dd \mu^{n},
  \end{displaymath}
  where $S^{1}=\{z\in \C\mid |z|=1\}$ is the unit circle of $\C$, and $\mu$ the
  probability Haar measure on~it.
\end{defn}

These quantities are related by the inequalities
\begin{equation}\label{eq:9}
  \h(f)\le \h_{1}(f)\le \h(f)+\log(n+1) \deg(f) , \quad  | \m(f)- \h(f)|\le\log(n+1) \deg(f),
\end{equation}
see for instance \cite[Lemma 2.30]{DKS:hvmsaN}.

The projective space $\P^{n}_{\Q}$ has a standard structure of toric
variety. Using Arakelov geometry, one can define a notion of canonical
height for its equidimensional subvarieties, as explained in
\cite[Chapter 1]{BurgosPhilipponSombra:agtvmmh}. Alternatively, this
height can be defining using Chow forms and a limit procedure
\emph{\`a la Tate} as in \cite{DavidPhilippon:mhn}, see also
\cite[\S2.3]{DKS:hvmsaN}.

For an equidimensional projective variety $V\subset \P^{n}_{\Q}$, its
canonical height is denoted by $\hcan(V)$. It is a nonnegative real
number that measures the arithmetic complexity of $V$, and that can
also be considered as an arithmetic analogue of its degree.

When $V$ is of dimension zero, its canonical height coincides with the
sum of the Weil heights of its points. In the other extreme,
$\hcan(\P^{n}_{\Q})=0$. When $V$ is a hypersurface, its canonical
height is the Mahler measure of its primitive defining polynomial, see
for instance~\cite[Proposition 2.39]{DKS:hvmsaN}.

\begin{defn}
  \label{def:8}
For an equidimensional affine variety $X\subset \A^{n}_{\Q}$, we
define its \emph{degree} and \emph{height}, respectively denote by
\begin{displaymath}
\deg(X) \and \hcan(X),
\end{displaymath}
as the degree and the canonical height of the closure of the image of $X$
by the standard inclusion $\iota\colon \A^{n}_{\Q}\hookrightarrow \P^{n}_{\Q}$
given by $\iota(x_{1},\dots, x_{n})=(1:x_{1}:\dots:x_{n})$. 
\end{defn}

Given a field extension $\Q\hookrightarrow K$, we set
\begin{displaymath}
  X_{K}=X\times {\Spec(K)} 
\end{displaymath}
for the subvariety of $\A^{n}_{K}$ obtained from $X$ by base change.
Set also $\bfone=(1,\dots, 1)\in \N^{r}$ and, for $i=1,\dots,r$, let
$\epsilon_{i}\in \N^{r}$ denote the $i$-th vector in the standard
basis of $\R^{r}$.

The following is the main result of this section. 

\begin{thm}
  \label{thm:1}
  Let $X\subset \A^n_\Q$ be a variety of pure dimension $r\ge 1$ such that
  \begin{displaymath}
\#(X_{\C}\cap V(x_{1},\dots, x_{r}))=\deg(X).    
  \end{displaymath}
Let $f_{i}\in
  \Z[x_{i}]\setminus \Z$, $i=1,\dots, r$, and $g\in \Z[x_{1},\dots,
  x_{n}]$. Set $\bfd=(d_{1},\dots, d_{r})\in \N^{r}$ with $d_{i}=\deg(f_{i})$
  and $e=\deg(g)$, and let $f_{i,d_{i}}$ be the leading coefficient of
  $f_{i}$, $i=1,\dots, r$. Then there exists $\gamma\in
  \Z\setminus \{0\}$ with
  $$
 \log|\gamma|\leq  e\, ( \hcan(X) + \deg(X) (n+2)\log(2n+3))
  $$
  such that, for   $\bfalpha=(\alpha_{1},\dots,
  \alpha_{r})\in \N^{r}$, 
  \begin{displaymath}
   \gamma\cdot \Big(\prod_{i=1}^{r}f_{i,d_{i}}^{e+r-\langle
  \bfalpha+\bfone,\bfd-\epsilon_{i}\rangle }\Big) \cdot
\Res_{X_{\C}} \Big[\begin{matrix}
g \dd x_1\wedge \dots \wedge \dd x_r \cr 
f_1^{\alpha_1+1},...,f_{r}^{\alpha_r+1}\end{matrix}\Big]
\in \Z 
  \end{displaymath}
 and 
\begin{multline*}
\log\bigg|\gamma\cdot \Big(\prod_{i=1}^{r}f_{i,d_{i}}^{e+r-\langle
  \bfalpha+\bfone,\bfd-\epsilon_{i}\rangle }\Big)
\cdot\Res_{X_{\C}} \Big[\begin{matrix}
g \dd x_1\wedge \dots \wedge \dd x_r \cr 
f_1^{\alpha_1+1},...,f_{r}^{\alpha_r+1}\end{matrix}\Big]\bigg| 
 \le  \h_{1}(g)\\ + (e+r-\langle
  \bfalpha+\bfone,\bfd\rangle)\sum_{i=1}^{r}\h(f_{i})+ e\,\hcan(X)  + e\, \deg(X) (n+3)\log(2n+3).
\end{multline*}
If $e<  
  \langle \bfalpha+\bfone,\bfd\rangle -r $, then 
  \begin{math}
  \Res_{X_{\C}} \Big[\begin{matrix}
g \dd x_1\wedge \dots \wedge \dd x_r \cr 
f_1^{\alpha_1+1},...,f_{r}^{\alpha_r+1}\end{matrix}\Big] = 0.    
  \end{math}
\end{thm}

When $X$ is  the affine space, we have the following more
precise result.

\begin{thm}
\label{thm:6}
Let $f_{i}\in \Z[x_{i}]\setminus \Z$, $i=1,\dots, n$, $g\in
\Z[x_{1},\dots, x_{n}]$ and $\bfalpha=(\alpha_{1},\dots,
\alpha_{n})\in \N^{n}$.  Set $\bfd=(d_{1},\dots, d_{n})\in \N^{n}$ with
$d_{i}=\deg(f_{i})$ and $e=\deg(g)$, and let $f_{i,d_{i}}$ be the
leading coefficient of $f_{i}$, $i=1,\dots, n$. Then
  \begin{equation*}
\Big(\prod_{i=1}^{n}f_{i,d_{i}}^{e+n-\langle
  \bfalpha+\bfone,\bfd-\epsilon_{i}\rangle } \Big) \cdot
\Res_{\A^{n}_{\C}}  \Big[\begin{matrix}
g \dd x_1\wedge \dots \wedge \dd x_n \cr 
f_1^{\alpha_1+1},...,f_{n}^{\alpha_n+1}\end{matrix}\Big]
\in \Z
\end{equation*}
and 
\begin{multline*}
\log\bigg|\Big(\prod_{i=1}^{n}f_{i,d_{i}}^{e+n-\langle
  \bfalpha+\bfone,\bfd-\epsilon_{i}\rangle}\Big) 
\cdot\Res_{\A^{n}_{\C}} \Big[\begin{matrix}
g \dd x_1\wedge \dots \wedge \dd x_n \cr 
f_1^{\alpha_1+1},...,f_{n}^{\alpha_n+1}\end{matrix}\Big]\bigg| 
 \\\le \h_{1}(g)+ 
(e+n-\langle
  \bfalpha+\bfone,\bfd\rangle)\sum_{i=1}^{r}\h(f_{i})+ (e-|\bfd|+n)\log(2).
\end{multline*}
If $e<  \langle \bfalpha+\bfone,\bfd\rangle -n $, then 
  \begin{math}
  \Res_{\A^{n}_{\C}} \Big[\begin{matrix}
g \dd x_1\wedge \dots \wedge \dd x_n \cr 
f_1^{\alpha_1+1},...,f_{n}^{\alpha_n+1}\end{matrix}\Big]= 0.    
  \end{math}
\end{thm}

We give the proof of these theorems after some definitions and
auxiliary results.  We first extend the bound for the coefficients of
Laurent expansions (Corollary \ref{cor:2}) to our current multivariate
setting.  Let $\bff=(f_{1},\dots, f_{r})$ with
$f_{i}\in \Z[x_{i}]\setminus\Z$, and
$\bfalpha=(\alpha_{1},\dots, \alpha_{r})\in \N^{r}$. Consider the
multivariate Laurent series of ${\bff^{-\bfalpha-\bfone}}$ given by
the product of univariate Laurent series of the
$f_{i}^{-\alpha_{i}-1}$'s around the point at infinity: with notation
as in~\eqref{eq:65},
\begin{multline}
  \label{eq:68}
{\bff^{-\bfalpha-\bfone}}
= \prod_{i=1}^{r}f_{i}^{-\alpha_{i}-1}= \prod_{i=1}^{r}\sum_{l_{i}\in
  \N}c_{f_{i},\alpha_{i},l_{i}} \, x_{i}^{-\alpha_{i}d_{i}-l_{i}} \\=
\sum_{\bfl\in \N^{r}} {c_{\bff,\bfalpha,\bfl}}\,
{x_{1}^{-(\alpha_{1}+1)d_{1}-l_{1}}\dots
  x_{r}^{-(\alpha_{r}+1)d_{r}-l_{r}}}
\end{multline}
with $c_{\bff, \bfalpha,\bfl}= \prod_{i=1}^{r}
c_{f_{i},\alpha_{i},l_{i}}$. This series is convergent for
$(x_{1},\dots, x_{r})\in \C^{r}$ such that 
$|x_{i}|>\max_{\xi\in V(f_{i})}|\xi|$ for all $i$. 

The following bound is a direct consequence of Corollary~\ref{cor:2}.

\begin{prop}
\label{prop:9}
Let $\bff=(f_{1},\dots, f_{r})$ with $f_{i}\in \Z[x_{i}]\setminus\Z$
and $\bfalpha=(\alpha_{1},\dots, \alpha_{r})\in \N^{r}$. Set
$\bfd=(d_{1},\dots, d_{r})\in \N^{r}$ with $d_{i}=\deg(f_{i})$, and
let $f_{i,d_{i}}$ be the leading coefficient of $f_{i}$, $i=1,\dots,
r$. Then the coefficients of the Laurent expansion of
$\bff^{-\bfalpha-\bfone}$ in \eqref{eq:68} satisfy
\begin{equation*}
\Big(\prod_{i=1}^{r}  f_{i,d_{i}}^{l_{i}+\alpha_{i}+1} \Big) c_{\bff,\bfalpha,\bfl}\in \Z
\end{equation*}
and 
\begin{equation*}
  \log\Big| \Big(\prod_{i=1}^{r}  f_{i,d_{i}}^{l_{i}+\alpha_{i}+1}\Big)
  c_{\bff,\bfalpha,\bfl}\Big|\le \sum_{i=1}^{r}l_{i}\h(f_{i})+(|\bfl|
   +\langle \bfalpha, \bfd\rangle)\log(2).
\end{equation*}
\end{prop}

For the remainder of this section, we assume that $X\subset \A^n_\Q$
is a variety of pure dimension $r\ge 1$ and that
\begin{equation}
  \label{eq:52}
  \#(X_{\C}\cap V(x_{1},\dots,
x_{r}))=\deg(X).
\end{equation}
Set also $\bfx'=(x_{1},\dots, x_{r})$.

For $g\in \Z[x_{1},\dots, x_{n}]$, we consider the trace function
$\Theta_{X,\bfx',g}$ as in \eqref{eq:31}.  By Proposition
\ref{prop:15}, this is a polynomial in the variables
$\bfy=(y_{1},\dots, y_{r})$ with rational coefficients, and these
coefficients are given by the residue multi-sequence of the polynomial
$r$-form~$g\dd \bfx'$: for
$\bfalpha=(\alpha_{1},\dots, \alpha_{r})\in \N^{r}$,
\begin{equation*}
\coeff_{\bfalpha}(\Theta_{X,\bfx',g})=
 {\ResXC}  \bigg[\begin{matrix} g \dd x_1\wedge \dots \wedge \dd x_r \cr 
x_1^{\alpha_1 + 1},\dots,x_r^{\alpha_r+1}\end{matrix} \bigg].
\end{equation*}   
Moreover, $\deg(\Theta_{X,\bfx',g})\le \deg(g)$.  We next bound
these coefficients, by applying the arithmetic Perron theorem in
\cite{DKS:hvmsaN}.

\begin{lem}
  \label{lemm:4}
  Let $X$, $\bfx'$ and $g$ be as above, and
  $\Theta_{X,\bfx'g}\in \Q[\bfy]$ the associated trace function. Then
  there exists $\gamma\in \Z\setminus \{0\}$ with
  $\gamma\, \Theta_{X,\bfx'g}\in \Z[x_{1},\dots,x_{r}]$ such that
\begin{align*}
\log|\gamma| &\le \deg(g)\, ( \hcan(X) + \deg(X) (n+2)\log(2n+3) ), \\
 \h(\gamma\,  \Theta_{X,\bfx'g}) & \le \h_{1}(g)+ \deg(g)\, ( \hcan(X) + \deg(X)
 (n+2)\log(2n+3) ).
\end{align*}
\end{lem}

\begin{proof}
  Set $e=\deg(g)$ and $D=\deg(X)$ for short.  If $e=0$, then $g\in \Z$
  and $\Theta_{X,\bfx',g}=Dg$, and the statement is clear. Hence, we
  can suppose without loss of generality that~$e\ge 1$.

Write $\bfx'=(x_{1},\dots, x_{r})$ and
$\bfx''=(x_{r+1},\dots, x_{n})$, so that $ \bfx=(\bfx',\bfx''). $
Consider the map
\begin{displaymath}
  \pi\colon X\longrightarrow \A^{r}_{\Q}, \quad \bfx\longmapsto \bfx'.
\end{displaymath}
Its fiber at $0$ coincides with the intersection $X\cap V(\bfx')$. By
hypothesis, this fiber is of cardinality $\deg(X)$ and so this is a
finite map, see for instance \cite[Lemma 2.14]{KPS:sean}.  Hence, the
map
\begin{displaymath}
  \id_{\A^{r}}\times \pi\colon \A_{\Q}^{n-r}\times X
\longrightarrow \A_{\Q}^{n-r}\times \A^{r}_{\Q}
\end{displaymath}
is also finite of degree $D$.  Let $\bfu=(u_{r+1},\dots, u_{n})$ be a
group of $n-r$ variables and set
\begin{displaymath}
 p= g+\sum_{i=r+1}^{n}u_{i}x_{i}\in \Z[\bfu,\bfx]. 
\end{displaymath}
Set $\bfy=(y_{1},\dots, y_{r})$ and let $E\in \Z[\bfu,\bfy, t]$ be an
irreducible polynomial giving a minimal equation of integral
dependence for $p$ with respect to this map.

We have that $\deg_{t}(E)\le D$.  On the other hand, $p$ separates the
points of the fiber of $\pi$ at 0, which implies that $\deg_{t}(E)\ge
D$. Hence,
\begin{displaymath}
\deg_{t}(E)= D
\end{displaymath}
and so this minimal polynomial coincides, up to a scalar factor, with
the characteristic polynomial of $p$ with respect to the map
$\id_{\A^{r}}\times \pi$.

Precisely, if we write
\begin{math}
  E=\sum_{j=0}^{D}E_{j} t^{j}
\end{math}
with $E_{j}\in \Z[\bfu,\bfy]$, $j=0,\dots, D-1$, and $E_{D}\in
\Z\setminus\{0\}$, then $E_{D}^{-1}E$ is the characteristic polynomial
of $p$ with respect to the map $\id_{\A^{r}}\times \pi$ and
$E_{D}^{-1}E(\bfzero, \bfy)$ is the characteristic polynomial of $g$
with respect to the map $\pi$.  In particular,
\begin{displaymath}
  \Theta_{X,\bfx'g}=-\frac{E_{D-1}(\bfzero,\bfy)}{E_{D}}\in
  \Q[\bfy]. 
\end{displaymath}
Setting $\gamma=E_{D}\in\Z\setminus\{0\}$, we have that $\gamma\,
\Theta_{X,\bfx',g}\in \Z[\bfy]$ and $\h(\gamma\, \Theta_{X,\bfx',g})
\le \h(E_{D-1})$.

Consider now the map
\begin{equation*}
 \A_{\Q}^{n-r}\times X \longrightarrow \A_{\Q}^{n-r}\times  \A^{r+1}_{\Q}, \quad (\bfu,\bfx)=(\bfu,\bfx',\bfx'')\longmapsto (\bfu,\bfx',p(\bfu,\bfx)).
\end{equation*}
Its image coincides with the hypersurface defined by $E$. Then
\cite[Theorem 3.15]{DKS:hvmsaN} implies that
\begin{multline}
\label{eq:77}
  \h(E_{D}) + D\h(g), \h( E_{D-1}) + (D-1)\h(g)\\ \le 
  D \h(g)+ e\, (
  \hcan(\A^{r}_{\Q}\times X) + \deg(\A^{r}_{\Q}\times X)
  (n+2)\log(2n+3) ).
\end{multline}
By \cite[Lemma 3.16]{DKS:hvmsaN}, $\deg(\A^{r}_{\Q}\times
X)=\deg(X)=D$ and $\hcan(\A^{r}_{\Q}\times X)= \hcan(X)$. Combining
this with \eqref{eq:77}, we easily derive stated bounds for $\gamma$ and
$\Theta_{X,\bfx'g}$.
\end{proof}

The following lemma reduces the computation of residues on an affine
variety with respect to univariate polynomials in separated variables to the
monomial case.

\begin{lem}
\label{lemm:5}
With  notation  as in Theorem \ref{thm:1}, let 
$c_{\bff,\bfalpha,\bfl}$, $\bfl\in\N^{r}$, be the coefficients of the
Laurent expansion of $\bff^{-\bfalpha-\bfone}$ as in \eqref{eq:68}.
Then
\begin{equation}
\label{eq:80}
 \Res_{X_{\C}} \Big[\begin{matrix}
g \dd x_1\wedge \dots \wedge \dd x_r  \cr 
f_1^{\alpha_1+1},\dots,f_{r}^{\alpha_r+1}\end{matrix}\Big]=
\sum_{\bfl}
c_{\bff,\bfalpha,\bfl} \, \Res_{X_{\C}} \Big[\begin{matrix}
g \dd x_1\wedge \dots \wedge \dd x_r  \cr 
x_{1}^{(\alpha_{1}+1)d_{1}+l_{1}},\dots, x_{r}^{(\alpha_{r}+1)d_{r}+l_{r}}\end{matrix}\Big],
\end{equation} 
the sum being over $\bfl\in \N^{r} $ such that $ |\bfl| \le e- \langle
\bfalpha+\bfone, \bfd\rangle +r$.
\end{lem}

\begin{proof}
Let $R>\max_{i}\max_{\xi\in
  V(f_{i})}|\xi|$. The multivariate Laurent series in \eqref{eq:68}
converges uniformly on the $r$-dimensional analytic cycle
$$
\Gamma_{R}=\{
\bfxi\in X(\C) \mid |\xi_{1}|=\dots=|\xi_{r}|=R\}.
$$ 
Set for short $\dd \bfx'=\dd x_1\wedge \dots \wedge \dd x_r$. Then 
\begin{align*}
  \Res_{X_{\C}} \Big[\begin{matrix}
g
 \dd \bfx'  \cr 
\bff^{\bfalpha+\bfone}\end{matrix}\Big]
&=\frac{1}{(2\pi i)^{r}} \int_{\Gamma_{R}} \frac{g(\bfx)\dd \bfx'}{f_{1}^{\alpha_{1}+1}(x_{1})\dots f_{r}^{\alpha_{r}+1}(x_{r})} 
\\ 
&=\sum_{\bfl\in \N^{r}} c_{\bff, \bfalpha,\bfl} 
\frac{1}{(2\pi i)^{r}} \int_{\Gamma_{R}} 
\frac{g(\bfx) \dd \bfx'}{x_{1}^{(\alpha_{1}+1)d_{1}+l_{1}},\dots, x_{r}^{(\alpha_{r}+1)d_{r}+l_{r}}}\\
&=\sum_{\bfl\in \N^{r}} c_{\bff, \bfalpha,\bfl} \,   \Res_{X_{\C}} \Big[\begin{matrix}
g
 \dd \bfx'  \cr 
x_{1}^{(\alpha_{1}+1)d_{1}+l_{1}},\dots, x_{r}^{(\alpha_{r}+1)d_{r}+l_{r}}
\end{matrix}\Big].
\end{align*}
The hypothesis \eqref{eq:52} implies that $\bfx'$ has no zeros on $X$
on the hyperplane at infinity. By \eqref{eq:71} and Theorem \ref{jacobi}, this implies
that the residues in this last sum vanish whenever
$e<\langle \bfalpha+\bfone,\bfd\rangle + |\bfl|-r$, concluding the
proof.
\end{proof}

\begin{proof}[Proof of Theorems \ref{thm:1} and \ref{thm:6}]
  We first consider the general case, when $X\subset \A^n_\Q$ is of
  pure dimension $r$ and $ \#(X\cap V(x_{1},\dots, x_{r}))=\deg(X)$.
Set
\begin{displaymath}
  \eta=\prod_{i=1}^{r}f_{i,d_{i}}^{e+r-\langle
  \bfalpha+\bfone,\bfd-\epsilon_{i}\rangle} \in \Z\setminus \{0\}.
\end{displaymath}
Let $\bfl\in \N^{r}$ with $ |\bfl|\le e+r- \langle\bfalpha +\bfone,
\bfd\rangle$. We have that $l_{i}+\alpha_{i}+1\le e+r-\langle
\bfalpha+\bfone,\bfd-\epsilon_{i}\rangle$ for all $i$.  Hence,
Proposition \ref{prop:9} implies that \begin{math}\eta \,
  c_{\bff,\bfalpha,\bfl}\in \Z
\end{math}
and 
\begin{equation} 
\label{eq:73}
\log| \eta\, 
  c_{\bff,\bfalpha,\bfl}|\le 
(e-\langle
  \bfalpha+\bfone,\bfd\rangle +r)\sum_{i=1}^{r}\h(f_{i})+ (e-|\bfd|+r)\log(2).
\end{equation} 

Let $\gamma\in \Z\setminus\{0\}$ with $  \log|\gamma| \le e\,
( \hcan(X) + \deg(X) (n+2)\log(2n+3))$ as in Lemma~\ref{lemm:4}.
Set for short 
\begin{displaymath}
\rho_{\bff}(g,\bfalpha)=  \Res_{X_{\C}} \Big[\begin{matrix}
g \dd x_1\wedge \dots \wedge \dd x_r  \cr 
f_1^{\alpha_1+1},\dots,f_{r}^{\alpha_r+1}\end{matrix}\Big].
\end{displaymath}
From the formula in Lemma \ref{lemm:5}  and the bounds in
Lemma \ref{lemm:4} and \eqref{eq:73}, we deduce that
$\gamma\,\eta \,\rho_{\bff}(g,\bfalpha)\in \Z  $
and that 
\begin{align*}
\log|\gamma\,\eta\,\rho_{\bff}(g,\bfalpha)|\le& 
\max_{\bfl} \log| \eta\, 
  c_{\bff,\bfalpha,\bfl}| + 
\h_{1}(g)+ e\, \big( \hcan(X) + \deg(X) (n+2)\log(2n+3) \big) 
\\ & + \log (\#\{\bfl \in \N^{r}\mid |\bfl|\le e+r-
\langle\bfalpha+\bfone, \bfd\rangle\}) \\
\le& 
(e+r-\langle
  \bfalpha+\bfone,\bfd\rangle)\sum_{i=1}^{r}\h(f_{i})+ (e-|\bfd|+r)\log(2)+
\h_{1}(g)\\ & + e\, \big( \hcan(X) + \deg(X) (n+2)\log(2n+3) \big) \\
& 
+(e+r-\langle
  \bfalpha + \bfone,\bfd\rangle) \log(r+1)\\
\le& 
\h_{1}(g)+ (e+r-\langle
  \bfalpha+\bfone,\bfd\rangle)\sum_{i=1}^{r}\h(f_{i})+ e\,\hcan(X) \\ & + e\, \deg(X) (n+3)\log(2n+3).
\end{align*}
If $e< \langle \bfalpha+\bfone,\bfd\rangle -r $, by Theorem
\ref{jacobi} all the residues in the sum in the right-hand side of
the the formula \eqref{eq:80} vanish. Hence
$\rho_{\bff}(g,\bfalpha)=0$, proving the last statement.

If $X=\A^{n}_{\Q}$, then $
\Theta_{X,\bff,g}=\Tr_{\C(\bfx)/\C(\bfx)}(g)= g(\bfy).$ Hence, in this
case we can take $\gamma=1$ and we have that
$\h_{1}(\Theta_{X,\bfx'g})=\h_{1}(g)$. Thus
  \begin{align*}
    \log|\eta\,\rho_{\bff}(g,\bfalpha)|\le&
 \h_{1}(g) + \max_{\bfl} \log| \eta\, 
  c_{\bff,\bfalpha,\bfl}| \\[-2mm] \le &
\h_{1}(g)+
(e+r-\langle
  \bfalpha+\bfone,\bfd\rangle)\sum_{i=1}^{r}\h(f_{i})+ (e-|\bfd|+r)\log(2),
  \end{align*}
as stated.
\end{proof} 

For completeness, we also extend the bounds for the coefficients of
the $f$-adic expansions (Proposition \ref{prop:5}) to our current
multivariate setting.  Let $\bff=(f_{1},\dots, f_{r})$ with
$f_{i}\in \Z[x_{i}]\setminus\Z$.  Given
$p\in\Z[\bfx']= \Z[x_{1},\dots,x_{r}]$, its \emph{$\bff$-adic
  expansion} is its unique finite representation as
\begin{equation*}
p =\sum_{\bfalpha\in \N^{r}}p_{\bff,\bfalpha}\bff^{\bfalpha}
\end{equation*}
with $p_{\bff,\bfalpha}\in \Q[\bfx']$ such that
$\deg_{x_{i}}(p_{\bff,\bfalpha})\le \deg(f_{i}) -1$ for all $\bfalpha$
and $i$.  Using the Bergman-Weil formula (Theorem \ref{thm:3}), these coefficients can be
expressed in terms of residues as
\begin{equation*}
p_{\bff,\bfalpha}(\bfx')=   \Res_{\A^{r}_{\C}}\hspace{-0.5mm}  \Bigg[\begin{matrix}
\displaystyle{{p(\bfz)}\prod_{i=1 }^{r}\frac{f_{i}(z_{i})-f(x_{i})}{z_{i}-x_{i}}  \dd z_{1}\wedge \dots
  \wedge \dd z_{r} }\cr 
\displaystyle{f_{1}(z_{1})^{\alpha_{1}+1}, \dots, f_{r}(z_{r})^{\alpha_{r}+1}}\end{matrix}\Bigg]
\in \Q[\bfx'].
\end{equation*}

\begin{prop}
\label{prop:6}
Let $\bff=(f_{1},\dots, f_{r})$ with $f_{i}\in \Z[x_{i}]\setminus\Z$
and $p\in \Z[x_{1},\dots,x_{r}]$.  Set $\bfd=(d_{1},\dots, d_{r})$
with $d_{i}=\deg(f_{i})$ and $\bfe=(e_{1},\dots, e_{r})$ with
$e_{i}=\deg_{x_{i}}(p)$, and let $f_{i,d_{i}}$ be the leading
coefficient of $f_{i}$, $i=1,\dots, r$. Then
\begin{equation}
  \label{eq:59}
\Big(\prod_{i=1}^{r}  f_{i,d_{i}}^{e_{i}+1-\alpha_{i} (d_{i}-1)}\Big)\cdot
 p_{\bff,\bfalpha}\in \Z[x_{1},\dots, x_{r}]
\end{equation}
and
\begin{equation}
  \label{eq:76}
\h_{1}\Big( \Big( \prod_{i=1}^{r}  f_{i,d_{i}}^{e_{i}+1-\alpha_{i} (d_{i}-1)}\Big)\cdot
p_{\bff,\bfalpha}\Big)\le \h_{1}(p)+
\sum_{i=1}^{r}(e_{i}-\alpha_{i}d_{i})\h_{1}(f_{i})+ |\bfe|\log (2). 
\end{equation}
If $e_{i}<\alpha_{i} d_{i}$ for some $i$, then
$p_{\bff,\bfalpha}=0 $.
\end{prop}

\begin{proof}
  Consider first the case when $p$ is a monomial, that is, 
  \begin{math}
  p=(\bfx')^{\bfbeta}
  \end{math}
  with $\bfbeta=(\beta_{1},\dots, \beta_{r})\in \N^{r}$.  Since this is a product
  of polynomials in separated variables, its $\bff$-adic expansion can
  be obtained by multiplying the $f_{i}$-adic expansion of its 
  factors.  Hence, for $\bfalpha=(\alpha_{1},\dots,
  \alpha_{r})\in \N^{r}$,
\begin{equation}
  \label{eq:56}
  p_{\bff,\bfalpha}=\prod_{i=1}^{r} (x^{\beta_{i}}_{i})_{f_{i},\alpha_{i}}.
\end{equation}
Set $\nu = \prod_{i=1}^{r} f_{i,d_{i}}^{\beta_{i}+1-\alpha_{i}
  (d_{i}-1)}\in \Z\setminus \{0\}$.  By \eqref{eq:56} and Proposition
\ref{prop:5}, $\nu \, p_{\bff,\bfalpha}= \prod_{i=1}^{r} \big(
f_{i,d_{i}}^{\beta_{i}+1-\alpha_{i} (d_{i}-1)}
(x^{\beta_{i}}_{i})_{f_{i},\alpha_{i}})\in \Z[\bfx']$ and
\begin{equation*}
\h_{1}(\nu \, p_{\bff,\bfalpha})\le  
\sum_{i=1}^{r}\h_{1}\big(   f_{i,d_{i}}^{\beta_{i}+1-\alpha_{i} (d_{i}-1)}
(x^{\beta_{i}}_{i})_{f_{i},\alpha_{i}}\big)
\le \sum_{i=1}^{r}(\beta_{i}-\alpha_{i}d_{i})\h_{1}(f_{i})+ \beta_{i}\log
(2),  
\end{equation*}
proving \eqref{eq:59} and \eqref{eq:76} in this case. Moreover, if
$e_{i}<\alpha_{i} d_{i}$ for some $i$, then $
(x_{i}^{\beta_{i}})_{f_{i},\alpha_{i}}=0 $ and so
$p_{\bff,\bfalpha}=0$, giving also the last statement in this case.

The case of an arbitrary $p$ follows by linearity from the monomial
one.
\end{proof}

\section{An arithmetic elimination theorem}\label{sec:an-arithm-elim}

Let $f_{1},\dots, f_{n}\in \Z[x_{1},\dots, x_{n}]$ be polynomials with
a finite number of common zeros in $\Qbar^{n}$. A classical method to solve the system of
equations
\begin{displaymath}
  f_{1}=\dots=f_{n}=0
\end{displaymath}
is to eliminate variables, that is, to find $\phi_{l}\in \Z[x_{l}]
\setminus \{0\}$, $l=1,\dots, n$, and $a_{l,i}\in \Z[x_{1},\dots,
x_{n}]$, $l,i=1,\dots, n$, such that
\begin{displaymath}
  \phi_{l}=\sum_{i=1}^{n}a_{l,i}f_{i}.
\end{displaymath}
Applying a variant of his approach to the effective Nullstellensatz,
Jelonek has obtained an optimal upper bound for the degrees of these
polynomials \cite[Theorem~1.6]{Jelonek:eN}.  
Here we prove an
arithmetic analogue of this result, bounding the height of the
$\phi_{l}$'s and the $a_{l,i}$'s. Our proof proceeds by adapting
Jelonek's approach and applying the tools from arithmetic intersection
theory in \cite{DKS:hvmsaN}.  Our main result in this section (Theorem
\ref{thm:2}) is an arithmetic analogue of the ``generalized
elimination theorem'' in~\cite[Theorem~4.3]{Jelonek:eN}.

Given a variety $X \subset \A^{n}_{\Q}$, a polynomial relation is said
to \emph{hold on} $X$ if it holds modulo the ideal of definition
$I(X)$ or, equivalently, if it holds for every point of $X(\Qbar)$.

\begin{thm}
  \label{thm:2}
  Let $X \subset \A^{n}_{\Q}$ be a variety of pure dimension $r\ge 0$,
  $f_{1},\dots,f_{s}\in \Z[x_{1},\dots, x_{n}]\setminus \Z$ with $s\le
  r$, and $q \in \Z[x_{1},\dots, x_{n}]$ a polynomial that is constant
  on every irreducible component of $X_{\Qbar}\cap V(f_{1},\dots,
  f_{s})$. Set $d_{j}=\deg(f_{j})$, $j=1,\dots,
  s$. Then there exist $\phi\in \Z[t]\setminus \{0\}$ and
  $a_{1},\dots, a_{s}\in \Z[x_{1},\dots, x_{n}]$ such that
\begin{equation}
\label{eq:14}
  \phi(q) = a_{1}f_{1}+\dots +a_{s}f_{s} \quad \text { on } X
\end{equation}
satisfying, for $i=1,...,s$,
\begin{align}
\label{eq:11}
&\deg(\phi)\le \Big( \prod_{j=1}^{s}
  d_{j}\Big) \deg(X), \\[-3mm]
\label{eq:25}
&\deg (a_{i}) + \deg(f_{i}) \le \deg(q) \Big( \prod_{j=1}^{s}
  d_{j}\Big) \deg(X),  \\[-2mm]
& \h(\phi), \h(a_{i})+\h(f_{i})\le \deg(q) \Big(\prod_{j=1}^{s}d_{j}
  \Big)\Big(\h(X) + \deg(X)\Big(
 \sum_{j=1}^{s}\frac{\h(f_{j})}{d_{j}} + \frac{\h(q)}{\deg(q)} \nonumber\\[-2mm]
\label{eq:29}
 &  \pushright{ + (r+1)\log(2(r+2)(n+1)^{2})  \Big)\Big)}.  
\end{align}
\end{thm}

We give the proof of this result after some  auxiliary lemmas. 

\begin{lem}
  \label{lemm:2}
  Let $X \subset \A^{n}_{\Q}$ be an equidimensional  variety
  and $q\in \Z[x_{1},\dots, x_{n}]$ a polynomial that is constant
  on every irreducible component of $X_{\Qbar}$. Then there exists
  $\phi\in \Z[t]\setminus \{0\}$ such that
  \begin{displaymath}
    \phi(q)=0 \quad \text{ on } X
  \end{displaymath}
  with
  \begin{equation}\label{eq:43}
\deg(\phi) \le \deg(X) \and   \m(\phi)\le \deg(q) \, \h(X) +
\deg(X)  \h_{1}(q-z),
  \end{equation}
where $z$ is an additional variable.  In particular, $ \h(\phi)\le
  \deg(q) \, \h(X) + \deg(X) ( \h(q)+ \deg(q) \log(n+2))$.
\end{lem}

\begin{proof}
  Let $Z=\{\zeta_{1},\dots,\zeta_{l}\}\subset \Qbar$ be the finite set
  of values of $q$ on $X({\Qbar})$. It is invariant under the action
  of the absolute Galois group $\Gal(\Qbar/\Q)$, and so we can
  consider a primitive polynomial with integer coefficients defined as
  \begin{displaymath}
    \phi =\gamma\prod_{i=1}^{l}(t-\zeta_{i})\in \Z[t]\setminus \{0\}
  \end{displaymath}
  for a suitable $\gamma\in \Q^{\times}$. By construction,  $\phi(q)=0 $ on $X$.
  The cardinality of $Z$ is bounded by the number of irreducible
  components of $X$ and, \emph{a fortiori}, by its degree. Hence $
  \deg(\phi)\le \deg(X)$, which gives the degree bound.

  To bound the Mahler measure of $\phi$, consider the hypersurface
  $V(q-z) \subset \A^{n}_{\Q}\times \A^{1}_{\Q}$, with $z$ the
  standard coordinate of $\A^{1}_{\Q}$. The variety
  \begin{displaymath}
   Y=(X\times
  \A^{1}_{\Q})\cap V(q-z) 
  \end{displaymath}
  is of dimension $r:=\dim(X)$ and its projection onto the last
  coordinate coincides with $Z$. Choose a subset $I\subset \{1,\dots,
  n\}$ of cardinality $r$ such that the projection $ \varpi\colon
  \A^{n}_{\Q}\times \A^{1}_{\Q}\to \A^{r}_{\Q}\times\A^{1}_{\Q}$
  defined by
\begin{displaymath}
\varpi(x_{1},\dots, x_{n}, z) \longmapsto ((x_{i})_{i\in I},z)
\end{displaymath}
verifies $\dim(\ov{\varpi(Y)})= r$. Let $\varrho\colon
\A^{r}_{\Q}\times\A^{1}_{\Q}\to \A^{1}_{\Q}$ be the projection onto
the second factor. Then $\varrho (\ov{\varpi(Y)})=Z$ is zero
dimensional. The theorem of dimension of fibers then implies that
\begin{displaymath}
\ov{\varpi(Y)} = \A_{\Q}^{r}\times Z, 
\end{displaymath}
and so $\ov{\varpi(Y)}$ is a hypersurface with defining polynomial
$\phi$.

By \cite[Corollary 2.61 and Lemma 3.16]{DKS:hvmsaN},
\begin{equation*}
  \h(Y)\le \deg(q)\h(X) + \deg(X)
  \h_{1}(q-z) .
\end{equation*}
By \cite[Proposition 2.64]{DKS:hvmsaN}, we have
$\h(\ov{\varpi(Y)})\le \h(Y)$. Since $\ov{\varpi(Y)}=V(\phi)$,
we have that $ \h(\ov{\varpi(Y)})=\m(\phi)$, which completes the
proof of \eqref{eq:43}. The last statement follows from the second
inequality in \eqref{eq:9}.
\end{proof}

\begin{defn}
  \label{def:5}
Let $\varphi\colon X\to Y$ be a generically finite dominant map of affine
varieties. We say that $\varphi$ is \emph{finite} at a point $\bfy\in Y$ if
there is an open neighborhood  $U$ of $ \bfy$ such that the
restriction of $\varphi $ to a map $ \varphi^{-1}(U) \to U$ is  finite. 
\end{defn}

The following lemma is a variant of \cite[Lemma 4.1]{Jelonek:eN}.

\begin{lem} 
\label{lemm:3} 
Let $X \subset \A^{n}_{\Q}$ be a variety of pure dimension $r\ge 0$,
$q\in \Z[x_{1},\dots, x_{n}]$ a polynomial that is not constant on any
irreducible component of $X_{\Qbar}$, and $u_{i,j}$, $i=1,\dots, r-1$,
$j=r, \dots, n$, a group of $(r-1)(n-r+1)$ variables. Consider the
transcendental field extension $\K=\Q((u_{i,j})_{i,j})$ and the map
\begin{displaymath}
  \pi\colon X_{\K}\longrightarrow  \A^{r}_{\K}, \quad \bfx=(x_{1},\dots, x_{n}) \longmapsto \bigg(x_{1}+ \sum_{j=r}^{n}u_{1,j}x_{j}, \dots, x_{r-1}+
  \sum_{j=r}^{n}u_{r,j}x_{j}, q(\bfx)\bigg).
\end{displaymath}
Then $\pi$ is dominant, and there exists $p\in \Z[t_{r}]\setminus
\{0\}$ such that $\pi$ is finite on $ \A^{r}_{\K}\setminus V(p)$.
\end{lem}

\begin{proof}
  Consider the map $\varrho\colon X\rightarrow \A^{1}_{\Q}$ given by
  $\bfx\mapsto q(\bfx)$ and let ${\cX}$ be its generic fiber.  In
  algebraic terms, this map corresponds to the morphism of
  $\Q$-algebras
\begin{math}
  \Q[z]\to \Q[x_{1},\dots, x_{n}]/I(X)
\end{math}
defined by $z\mapsto q$, and $\cX$ is the subvariety of $
\A^{n}_{\Q(g)}$ defined by the ideal
\begin{displaymath}
I({\cX}) = \Q(q) I(X)\subset \Q(q)[x_{1},\dots, x_{n}].  
\end{displaymath} 
The hypothesis that the polynomial $q$ is not constant on any
irreducible component of $X_{\Q}$ implies that the map $\varrho$ is
surjective and has no vertical fibers. Hence, $\cX$ is of pure
dimension $r-1$ and the natural morphism
\begin{displaymath}
   \Q[x_{1},\dots, x_{n}]/I(X) \longrightarrow \Q(q)[x_{1},\dots, x_{n}]/I({\cX})
\end{displaymath}
is an inclusion. Let $\cX_{\K}$ be the affine variety
obtained by base change, which is  the subvariety of $ \A^{n}_{\K(g)}$
corresponding to the ideal $I(\cX_{\K})= \K(q) I(X)\subset \K(q)[x_{1},\dots, x_{n}]$.  

For $i=1,\dots, r-1$, set $ \ell_{i}= u_{i,r}x_{r}+\dots
+u_{i,n}x_{n}\in \K[x_{r},\dots, x_{n}].  $ Then the map
\begin{equation} \label{eq:10}
{\cX}_{\K}  \longrightarrow \A^{r-1}_{ \K(q) }, \quad
(x_{1},\dots, x_{r-1}, \bfx')\longmapsto
(x_{1}+\ell_{1}(\bfx'), \dots, x_{r-1}+\ell_{r-1}(\bfx'))
\end{equation}
is dominant and finite, since it is the restriction to ${\cX}_{\K} $ of a
general linear map $\A^{n}_{ \K(q) }\to \A^{r-1}_{ \K(g) }$ in reduced
triangular form. 

Choose $\ell_{r}\in \Q[x_{r},\dots, x_{n}]$  a sufficiently generic linear
form, so that the map
\begin{equation}\label{eq:41}
X_{\K}  \longrightarrow \A^{r}_{ \K}, \quad (x_{1},\dots, x_{r-1}, \bfx')\longmapsto
(x_{1}+\ell_{1}(\bfx'), \dots, x_{r-1}+\ell_{r-1},\ell_{r}(\bfx'))
\end{equation}
is also finite, and let $P\in \Q(q)[t_{1},\dots, t_{r-1}][T]$ be a
polynomial giving an equation of integral dependence of $\ell_{r}$
with respect to the map in \eqref{eq:10}. Let $p\in
\Q[t_{r}]\setminus \{0\}$ so that $p(q)$ is a denominator of $P$. Then
$\ell_{r}$ integral with respect to the map $\pi$ on the open subset $
\A^{r}_{\K}\setminus V(p)$.

 Since the map in \eqref{eq:41} is finite, this implies that $\pi$ is
 finite outside $V(p)$.  Since $\dim(X_{ \K})=r=\dim( \A^{r}_{ \K})$,
 this map is also dominant, completing the proof.
\end{proof}

\begin{proof}[Proof of Theorem \ref{thm:2}]
  First we treat the case when $q$ is not constant on any of the
  irreducible components of $X_{\Qbar}$.  Let $z$ be an additional
  variable,  consider the map
\begin{equation}\label{eq:7}
\varphi\colon X\times \A^{1}_{\Q}\longrightarrow \A^{n}_{\Q}\times
\A^{s}_{\Q}, \quad (\bfx,z)\mapsto (\bfx, z f_{1}(\bfx), \dots, z f_{s}(\bfx))
\end{equation}
and let $W\subset \A^{n}_{\Q}\times \A^{s}_{\Q}$ be the closure of its
image.  The assumptions on the polynomial~$q$ imply that no
irreducible component of $X$ is contained in the zero set of the
$f_{j}$'s, and so $W$ is of pure dimension $r+1$. By construction,
$\varphi$ gives an isomorphism between $X$ and $W$ outside the zero
set of the $f_{j}$'s.

Let  $\{\zeta_{1},\dots, \zeta_{l}\} \subset \Qbar$ be the finite
 $\Gal(\Qbar/\Q)$-invariant subset of values of the polynomial $q$ on
the irreducible components of $X_{\Qbar}\cap
V(f_{1},\dots, f_{s})$ and set
\begin{displaymath}
  \theta_{1}=\gamma \prod_{i=1}^{l}(t-\zeta_{i})\in \Z[t]
\end{displaymath}
for a suitable $\gamma\in \Q^{\times}$ such that $\theta_{1}$ is
primitive. The hypersurface $H\coloneqq V(\theta_{1}\circ q)\subset \A^{n}_{\Q}$
contains the variety $ X\cap V(f_{1},\dots, f_{s})$ and so the
restricted map
\begin{displaymath}
\varphi \colon (X\setminus H)\times\A^{1}_{\Q}\longrightarrow
W\setminus (H \times \A^{s}_{\Q}) 
\end{displaymath}
is an isomorphism. Since $q$ is not constant on any of the irreducible
components of $X_{\Qbar}$, the hypersurface $H$
contains no irreducible component of $X$. 

For $i=1,\dots, r$, let $\bfu_{i}=(u_{i,1}, \dots, u_{i,n})$ be a
group of $n$ variables and consider the general linear form
\begin{displaymath}
  \ell_{i}=u_{i,1}x_{1}+\dots+ u_{i,n}x_{n}
  \in\Q(\bfu_{i})[x_{1},\dots, x_{n}]. 
\end{displaymath}
Set $\bfu=(\bfu_{1},\dots, \bfu_{r})$, and consider the field
$\K=\Q(\bfu)$ and the map
\begin{displaymath}
  \pi\colon W_{\K} \longrightarrow 
  \A^{r+1}_{\K}, \quad (\bfx, \bfz)\longmapsto (z_{1}+\ell_{1}(\bfx),
  \dots, z_{s}+\ell_{r}(\bfx),\ell_{s+1}(\bfx), \dots, \ell_{r}(\bfx),
  q(\bfx)).
\end{displaymath}
By Lemma \ref{lemm:3} this map is dominant, and there exists
$\theta_{2}\in \Z[y_{r+1}]\setminus \{0\}$ such that it is finite on
the open subset $ \A^{r+1}_{\K}\setminus V(\theta_{2})$.  

Let $\varphi_{\K}$ be the base change of the map in \eqref{eq:7} given
by the field extension $\Q\hookrightarrow \K$ and set
$\theta=\theta_{1}(y_{r+1}) \theta_{2}(y_{r+1})\in \Z[y_{r+1}]$. Then
the composition
\begin{equation}
  \label{eq:42}
\psi=\pi\circ \varphi_\K
\colon X_{\K}\times \A^{1}_{\K}\longrightarrow  \A^{r+1}_{\K}  
\end{equation}
is finite on the open subset $ \A^{r+1}_{\K}\setminus V(\theta)$.

Set $\bfy=(y_{1},\dots, y_{r+1})$. For convenience, we also set
$f_{j}=0$, $j=s+1,\dots, r$. Then the previous condition is equivalent
to the fact that the inclusion of $\Q(\bfu)$-algebras
\begin{equation} \label{eq:13}
  \psi^{\#}\colon \K[\bfy]_{\theta}  \longrightarrow
(\K[x]/I(X_{\K}))_{\theta\circ q} \otimes \K[z], 
\quad  y_{i}\longmapsto
  \begin{cases}
    zf_{i}+\ell_{i} & \text{ if } 1\le i\le r, \\
    q  & \text{ if } i=r+1, \\
  \end{cases}
\end{equation}
is integral.  Let $E\in\K[\bfy]_{\theta}[z]$ be the minimal polynomial
of the variable $z$ with respect to~$\psi^{\#}$. It is unique up to a
unit of the ring  $\K[\bfy]_{\theta}$.  We choose it as a primitive and
squarefree polynomial in $\Z[\bfu][\bfy, z]$ and write it is as
\begin{equation*}
  E=\sum_{j=0}^{\delta} \sum_{\bfalpha\in \N^{r+1}}E_{\bfalpha,j}\,  \bfy^{\bfalpha}z^{\delta-j}
\end{equation*}
with  $\delta=\deg_{z}(E)$ and $E_{\bfalpha,j}\in \Z[\bfu]$ for all
$\bfalpha, j$.  The coefficient of $z^{\delta}$
is 
\begin{displaymath}
  \wt \phi=\sum_{\bfalpha \in \N^{r+1}}E_{\bfalpha,0}\, \bfy^{\bfalpha}.
\end{displaymath}
Since the morphism $\psi^{\#}$ in
\eqref{eq:13} is integral,  $\wt \phi= \lambda
\theta^{k}(y_{r+1})$ with $\lambda\in \Z[\bfu]$ and $k\in \N$. 

By definition, 
\begin{multline}\label{eq:6}
  E(z,zf_{1}+\ell_{1}, \dots, zf_{r}+\ell_{r}, q) \\= 
\sum_{j=0}^{\delta} \sum_{\bfalpha\in \N^{r+1}}E_{\bfalpha,j}\Big(
\prod_{k=1}^{r}(zf_{k}+\ell_{k})^{\alpha_{i}} \Big) q^{\alpha_{r+1}}z^{\delta-j}
=0 \quad \text{ on
  } X_{\K}\times \A^{1}_{\K}.
\end{multline}
Hence, all the coefficients in the expansion of this expression with
respect to $z$ vanish identically on $X_{\K}$. We will extract the
relation \eqref{eq:14} from the coefficient of $z^{\delta}$. Set, for
$i=1,\dots, s$,
\begin{displaymath}
  \wt a_{i}= - \sum_{j=1}^{\delta}\sum_{\bfalpha\in
    \N^{r+1}}\sum_{\bfbeta \in B_{i,j,\bfalpha}}E_{\bfalpha,j}
    \bigg(\prod_{k=1}^{r}{\alpha_{k}\choose
    \beta_{k}} \ell_{k}^{\alpha_{k}-\beta_{k}} \bigg) f_{1}^{\beta_{1}}\dots
  f_{i-1}^{\beta_{i-1}}f_{i}^{\beta_{i}-1} q^{\alpha_{r+1}}\in \Z[\bfu,x],
\end{displaymath}
where  the indexing set $ B_{i,j,\bfalpha}$ consists of the  vectors $\bfbeta
=(\beta_{1},\dots, \beta_{r}) $ with $|\bfbeta| = j$ and
$\beta_k\leq \alpha_k$ for $1\leq k\leq i$, $\beta_i \geq 1$, and
$\beta_k =0 $ for $ i+1\leq k\leq r $.  It follows from \eqref{eq:6}
that
\begin{displaymath}
  \wt \phi(q) = \wt a_{1}f_{1} + \dots + \wt a_{s}f_{s} \quad \text{ on
  } X_{\K}. 
\end{displaymath}
The relation \eqref{eq:14} is obtained by taking the coefficient of some
nonzero term in the monomial expansion with respect to $\bfu$ of the
left-hand side of this relation.

To control the degree and height of the minimal polynomial $E$, we
relate it to the implicit equation of the image of a polynomial
map. Indeed, by \cite[Lemma 4.8]{DKS:hvmsaN} applied to the map
$\psi$ in \eqref{eq:42}, the polynomial $E$ gives also an equation for the closure of
the image of the map
\begin{equation} \label{eq:5}
  X_{\K(z)}\longrightarrow \A^{r+1}_{\K(z)}, \quad (\bfx,z)\longmapsto
  (zf_{1}(\bfx)+\ell_{1}(\bfx), \dots, zf_{r}(\bfx)+\ell_{r}(\bfx), q(\bfx)).
\end{equation}
Note that \cite[Lemma 4.8]{DKS:hvmsaN} is stated finite maps, but it
also holds for generically finite maps, with the same proof, and so we
can apply it to the map $\psi$ to prove the claim above.

Hence, we can bound the size of $E$ by applying the arithmetic Perron
theorem. To this end, we summarize the partial degrees, height and
number of monomials of the polynomials defining the map in
\eqref{eq:5}.  Set $d_{j}=\deg(f_{j})$ and $h_{j}=\h(f_{j})$,
$j=1,\dots, s$ and, for convenience, we set also $d_{j}=1$ and
$h_{j}=0$, $j=s+1,\dots, r$. For $j=1,\dots, r$, we have
\begin{enumerate}
\item \label{item:11} $\deg_{\bfx}(zf_{j}+\ell_{j}) \le d_{j}$,
\item \label{item:12} $\deg_{\bfu,z}(zf_{j}+\ell_{j}) =1 $,
\item \label{item:14}  $\h(zf_{j}+\ell_{j}) \le h_{j}$, 
\item \label{item:15} $\log(\# \supp(zf_{j}+\ell_{j})+2) \le d_{j} \log(2n+3)$.
\end{enumerate}
Set $e=\deg_{\bfx}(q)$ and $l=\h(q)$. We have $\deg_{\bfu,z}(q)=0$,
and the number of parameters, that is the auxiliary variables $\bfu,
z$, is $rn+1$.

Write
\begin{displaymath}
  E=\sum_{\bfalpha\in  \N^{r+1}} E_\bfalpha\, \bfy^{\bfalpha}
\end{displaymath}
with $E_{\bfalpha}\in \Z[\bfu,z]\setminus \{0\}$, and consider the
vectors in $\N^{r+1}$ given by
\begin{displaymath}
   \bfd=(d_{1},\dots, d_{s}, 1, \dots, 1, e), \quad \bfdelta=(1,\dots,
   1, 0), \quad \bfh=
   (h_{1},\dots, h_{s}, 0,\dots, 0,l). 
\end{displaymath}
By \cite[Theorem 3.15]{DKS:hvmsaN}, for all $\bfalpha\in \supp(E)$,
\begin{equation}
\label{eq:15}
  \langle \bfd, \bfalpha \rangle \le e \Big(\prod_{j=1}^{s}d_{j} \Big)
  \deg(X), \quad \deg_{\bfu,z}(E_{\bfalpha}) +  \langle \bfdelta,
  \bfalpha\rangle \le r e \Big(\prod_{j=1}^{s}d_{j} \Big) \deg(X) 
\end{equation}
and
\begin{align}
\label{eq:17}
\nonumber \h(E_\bfalpha) +  \langle \bfh, \bfalpha \rangle \leq &   e
\Big(\prod_{j=1}^{s}d_{j} \Big)\bigg(\h(X)  + \deg(X)\bigg(
\sum_{j=1}^{r}\frac{h_{j}+d_{j}\log(2n+3))+\log(rn+2)}{d_{j}} \\[-1mm]
\nonumber &+ \frac{l+\log(\#\supp(q)+2)}{e} +
\log(r+2) \bigg)\bigg)\\
\nonumber  \le& e
\Big(\prod_{j=1}^{s}d_{j} \Big)\bigg(\h(X)  + \deg(X)\bigg(
\sum_{j=1}^{s}\frac{h_{j}}{d_{j}} +
r \log((2n+3)(rn+2))  \\[-2mm]
&+ 
\frac{l}{e}+\log((n+3)(r+2))\bigg)\bigg).
\end{align}
From \eqref{eq:15} we get $e\deg_{y_{r+1}}(\wt \phi) \le
e\big(\prod_{j=1}^{s}d_{j} \big) \deg(X)$, which proves \eqref{eq:11}.
Again by~\eqref{eq:15}, for $i=1,...,s$,
\begin{displaymath}
 \deg_{\bfx}(\wt a_{i})+\deg(f_{i}) \le \max_{\bfalpha \in \supp (E)}\langle \bfd,\bfalpha \rangle \le e\Big(\prod_{j=1}^{s}d_{j} \Big) \deg(X),
\end{displaymath}
which proves \eqref{eq:25}.  The bound for the height of $\wt \phi$
follows  from \eqref{eq:17}. For the height of the~$\wt
a_{i}$'s  we have, for $i=1,...,s$, 
\begin{multline*}
  \sum_{j=1}^\delta \sum_{{\bfalpha \in \N^{r+1} \atop      E_{\bfalpha,j}\not=0}}
\sum_{\bfbeta \in B_{i,j,\bfalpha}} 
   \prod_{k=1}^{r}{\alpha_{k}\choose
    \beta_{k}} \le 
  \sum_{\bfalpha \in \supp(E)}\sum_{\beta \le \bfalpha} \prod_{k=1}^{r}{\alpha_{k}\choose
    \beta_{k}} \\
    = \sum_{\bfalpha \in \supp(E)} 2^{\alpha_{1}+\dots+\alpha_{r}} \le \#\supp(E)
  \, 2^{\deg_{\bfy}(E)}.
\end{multline*}
Hence, by \cite[Lemma 2.37(1)]{DKS:hvmsaN},
\begin{multline}
  \label{eq:19}
\h(\wt a_{i})+\h(f_{i})\le 
\max_{j, \bfalpha, \bfbeta}
\bigg\{ \h\bigg( E_{\bfalpha,j}\bigg(\prod_{k=1}^{r}\ell_{k}^{\alpha_{k}-\beta_{k}} \bigg) f_{1}^{\beta_{1}}\dots
  f_{i-1}^{\beta_{i-1}}f_{i}^{\beta_{i}-1}  q^{\alpha_{r+1}}\bigg)
  \bigg\} \\  + \h(f_{i})
  + \log\big(\#\supp(E)\, 2^{\deg_{\bfy}(E)}\big).
\end{multline}
Let $\epsilon_{i}\in\N^{r}$ be the $i$-th vector in the standard basis
of $\R^{r}$.
 By \cite[Lemma 2.37(2)]{DKS:hvmsaN}, for each
$j,\bfalpha,\beta$ in the maximum in the right-hand side of
\eqref{eq:19}, we have
\begin{multline}
  \label{eq:18}
  \h\bigg( E_{\bfalpha,j}\bigg(\prod_{k=1}^{r}\ell_{k}^{\alpha_{k}-\beta_{k}} \bigg) f_{1}^{\beta_{1}}\dots
  f_{i-1}^{\beta_{i-1}}f_{i}^{\beta_{i}-1}  q^{a_{r+1}}\bigg) \le
  \h(E_{\bfalpha,j})+ |\bfalpha|\log(n) \\
+ \langle \bfh-\epsilon_{i},\bfalpha\rangle + \langle \bfd, \bfalpha\rangle
\log(n+1),
\end{multline}
because $\h_{1}(\ell_{k})=\log(n)$ and $\h_{1}(f_{j}) \le
\h(f_{j})+ d_{j}\log(n+1)$. It follows from \eqref{eq:15},
\eqref{eq:17}, \eqref{eq:19} and \eqref{eq:18} that, for $i=1,...,s$,
the sum $ \h(\wt a_{i})+\h(f_{i})$ is bounded by
\begin{equation*}
 e \Big(\prod_{j=1}^{s}d_{j} \Big)\bigg(\h(X)  + \deg(X)\bigg(
\sum_{j=1}^{s}\frac{h_{j}}{d_{j}} + 
\frac{l}{e} +
(r+1) \log(2(r+2)(n+1)^{2})\bigg)\bigg),
\end{equation*}
which gives \eqref{eq:29}. This concludes the case when $q$ is
not constant on any of the irreducible components of $X_{\Qbar}$.

For the general case,  consider a splitting $X=X_{1}\cup X_{2}$, 
where $X_{1}$ and
  $X_{2}$ denote the union of the irreducible components of
  $X_{\Qbar}$ where $q$ is constant and
  not constant, respectively. Let 
  \begin{equation}\label{eq:12}
\phi_{1}(q) = 0 \text{ on } X_{1} \quad \text{ and } \quad     
\phi_{2}(q) = a_{1}f_{1}+\dots+a_{s}f_{s} \text{ on } X_{2}
  \end{equation}
be the equations obtained by applying Lemma \ref{lemm:2} on $X_{1}$ and the
previously considered case on $X_{2}$, respectively.
Then 
\begin{displaymath}
(\phi_{1}\phi_{2})(q) = \phi_{1}(q) a_{1}f_{1}+\dots+\phi_{2}(q) a_{s}
f_{s} \text{ on } X
\end{displaymath}
is the corresponding equation on $X$. The bounds for the polynomials
therein follow readily for those for the polynomials in \eqref{eq:12}.
\end{proof}

\begin{rem}
  \label{rem:1}
  The case when $s>r$, that is, when the number of $f_{i}$'s
  exceeds the dimension of the variety $X$, can be reduced to the case
  when $s\le r$ considered in Theorem \ref{thm:2}, by taking linear
  combinations of the $f_{i}$'s.  The resulting bounds are not so neat
  since, in particular, the influence of the different $f_{i}$'s mixes
  and we loose track of their individual contributions. This case when
  $s>r$ is not necessary for  our present applications, and we omit the
  formulation of the corresponding bounds.
\end{rem}

The next statement is the specialization of Theorem \ref{thm:2} to the
0-dimensional case and a variable $q=x_{l}$. 

\begin{cor}
\label{cor:1}
Let $X \subset \A^{n}_{\Q}$ be a variety of pure dimension $r\ge 0$,
$\bff=(f_{1},\dots,f_{r})$ with
$f_{i}\in \Z[x_{1},\dots, x_{n}]\setminus \Z$ a family of polynomials
defining a complete intersection on~$X$. Set $d_{j}=\deg(f_{j})$,
$j=1,\dots, r$. Then, for $l=1,\dots, n$, there exist
$\phi_{l}\in \Z[x_{l}]\setminus \{0\}$ and
$a_{l,1},\dots, a_{l,r}\in \Z[x_{1},\dots, x_{n}]$ such that
\begin{equation*}
  \phi_{l} = a_{l,1}f_{1}+\dots +a_{l,r}f_{r} \quad \text { on } X
\end{equation*}
satisfying, for $i=1,...,r$,
\begin{align*}
&\deg(\phi_{l}), \deg (a_{l,i}) + \deg(f_{i}) \le \Big( \prod_{j=1}^{r}
  d_{j}\Big) \deg(X), \\[-2mm]
& 
\h(\phi_{l}), \h(a_{l,i})+\h(f_{i})\le  \Big(\prod_{j=1}^{r}d_{j}
  \Big)\Big(\h(X) + \deg(X)\Big(
  \sum_{j=1}^{r}\frac{\h(f_{j})}{d_{j}}\hspace{40mm}\\[-2mm]
&\pushright{+ (r+1) \log(2(r+2)(n+1)^{2})  \Big)\Big).}  
\end{align*}
\end{cor}

In the particular case when $X=\A^{n}_{\Q}$, we have 
  $\deg(\A^{n}_{\Q})=1$ and $\h(\A^{n}_{\Q})=0$. 
Hence, the previous statement specializes to 
\begin{equation*}
  \begin{split}
& \deg(\phi_l), \deg (a_{l,i}) + \deg(f_{i}) \le \prod_{j=1}^{n}
  d_{j}, \\
& \h(\phi_{l}), \h(a_{l,i})+\h(f_{i})\le
\Big( \prod_{j=1}^{n}d_{j} \Big)
  \Big(\sum_{j=1}^{n}\frac{\h(f_{j})}{d_{j}}  + (n+1) \log(2(n+2)(n+1)^{2}) \Big).
  \end{split}
\end{equation*}
In particular, if $\deg(f_{j})\le d$ and $\h(f_{j})\le h$, then 
\begin{equation*}
  \begin{split}
& \deg(\phi_l), \deg (a_{l,i}) + \deg(f_{i}) \le d^{n}, \\
& \h(\phi_{l}), \h(a_{l,i})+\h(f_{i})\le
n d^{n-1} h + (n+2) \log(2(n+2)(n+1)^{2}) \, d^{n}. 
  \end{split}
\end{equation*}

To prove our bound for residues over an affine variety, we also need
the next lemma, allowing to put the variables in general position with
a linear change of controlled height.

\begin{prop}
  \label{prop:1}
  Let $X\subset \A_{\Q}^{n}$ be a variety of pure dimension $r$. Then
  there are affine polynomials $\ell_{i}\in \Z[x_{1},\dots, x_{n}]$
  with $\h(\ell_{i}) \le 2\log(\deg(X)) + n \log (2) $ such that, for
  every subset $I\subset\{1,\dots, n\}$ of cardinality $r$,
\begin{equation}\label{eq:27}
  \# \Big(X_{\Qbar}\cap \bigcap_{i\in I}V(\ell_{i})\Big)= \deg(X). 
\end{equation}
In particular, the map $X\to \A^{r}_{\Q}$ defined by $(x_{1},\dots,
x_{n})\mapsto (\ell_{i})_{i\in I}$ is finite of $\deg(X)$.
\end{prop}

\begin{proof}
  For $i=1,\dots, n$, let $\bfu=(u_{i,1}, \dots, u_{n,j})$ be a group
  of $n$ variables. Set $D=\deg(X) $ for short.  By \cite[Proposition
  4.5]{KPS:sean}, there is a polynomial $G\in \Q[\bfu_{1},\dots,
  \bfu_{r}]$ with $\deg_{\bfu_{i}}(G) \le 2 D^{2}$, $i=1,\dots, r$,
  such that the condition $ G(\bfb_{1},\dots, \bfb_{r})\ne 0$ for
  $\bfb_{i}\in \Q^{n+1}$, $i=1,\dots, r$, implies that
\begin{equation*}
  \# \Big(X_{\Qbar}\cap \bigcap_{i=1}^{r}V(b_{i,0}+b_{i,1}x_{1}+\dots+b_{i,n}x_{n})\Big)=D.
\end{equation*}
We then set
\begin{displaymath}
  F:=\hspace{-2mm}\prod_{1\le i_{1}<\dots <i_{r}\le n}\hspace*{-4mm} G(\bfu_{i_{1}},\dots,
  \bfu_{i_{r}})\in \Q[\bfu_{1},\dots, \bfu_{n}]. 
\end{displaymath}
This polynomial verifies, for $i=1,\dots, n$,
\begin{displaymath}
  \deg_{\bfu_{i}}(F)={n-1 \choose r-1} 2 D^{2}\le 2^{n}D^{2}.
\end{displaymath}
Hence there are $\bfb_{i}\in \Z^{n+1}$, $i=1,\dots, n$, such that
$|b_{i,j}|\le 2^{n}D^{2}$ for all $i,j$, such that
\begin{displaymath}
 F(\bfb_{1},\dots,
\bfb_{n})\ne 0. 
\end{displaymath}
Then the affine polynomials
$\ell_{i}=b_{i,0}+b_{i,1}x_{1}+\dots+b_{i,n}x_{n}$ verify that
$\h(\ell_{i}) \le \log \max_{j}|b_{i,j}|\le 2\log(D) + n \log (2) $ for all $i$ and satisfy the
condition \eqref{eq:27} for every subset subset $I\subset\{1,\dots,
n\}$ of cardinality $r$, proving the first statement.

The second statement follows from the first one, applying \cite[Lemma
2.14]{KPS:sean}.
\end{proof}

\section{Residues on an affine variety: the general case}\label{sectionestimates}

In this section, we bound the residue multi-sequence on an affine variety
$X\subset \A^{n}_{\Q}$ of dimension $r$, associated to a system of polynomials
$\bff=(f_{1},\dots, f_{r}) \in (\Z[x_{1},\dots, x_{n}]\setminus
\Z)^{r}$
defining a complete intersection on $X$ and a rational $r$-form
$\omega$ (Definition \ref{def:9}). We also apply these results to
bound the coefficients of the representation of a polynomial via the
Bergman-Weil trace formula \eqref{eq:102}.

\begin{defn}\label{def:7}
  Let $\omega$ be a polynomial $r$-form defined over $\Z$ (Definition
  \ref{def:9}). Write $\omega=\sum_{I}g_{I}\dd \bfx_{I}$ with
  $g_{I}\in\Z[x_{1},\dots, x_{n}]$ for each subset
  $I\subset \{1,\dots, n\}$ of cardinality $r$. For each multi-index
  $I$, write $g_{I}=\sum_{\bfbeta}g_{I,\bfbeta} \, \bfx^{\bfbeta}$
  with $g_{I,\bfbeta}\in\Z$, $\bfbeta\in \N^{n}$. The \emph{degree}
  and the \emph{(logarithmic) length} of $\omega$ are respectively
  defined by
\begin{displaymath}
  \deg(\omega)=\max_{I}\deg(g_{I}) \and \h_{1}(\omega)= \log\Big(\sum_{I,\bfbeta}|g_{I,\bfbeta}| \Big).
\end{displaymath}
\end{defn}

With the previous notation set, for $\bfalpha\in
\N^{r}$, 
 \begin{displaymath}
   \rho_{X,\bff} (\omega,\bfalpha)= \ResXC \Big[\begin{matrix} \omega \cr
\bff^{\bfalpha+\bfone}\end{matrix}\Big].
 \end{displaymath}
 We first consider these residues for the case when $X$ is good
 position with respect to a group of variables and $\omega$ is a
 polynomial multiple of the volume form associated to this group of
 variables.

\begin{thm}
  \label{thm:8}
  Let $X\subset \A^n_\Q$ be a variety of pure dimension $r\ge1$ 
such that 
\begin{equation} \label{eq:109}
  \#(X\cap V(x_{1},\dots, x_{r}))= \deg(X). 
\end{equation}
Let $ \bff=(f_{1},\dots, f_{r})$ be a family of polynomials in
$\Z[x_{1},\dots, x_{n}]\setminus \Z$ defining a complete intersection
on $X$, and $\omega= g\dd x_{1}\wedge \dots \wedge \dd x_{r}$ with
$g\in \Z[x_{1},\dots, x_{n}]$.  Set $d_{i}=\deg(f_{i})$,
$i=1,\dots, r$, $e=\deg (g)$ and 
  \begin{equation}\label{eq:24}
D_{X,\bff}=\deg(X) \prod_{j=1}^{r}d_{j} \and 
\kappa^{0}_{X,\bff}=\frac{\h(X)}{\deg(X)} 
+ \sum_{j=1}^{r}\frac{\h(f_{j})}{d_{j}}+ 3(n+2)\log (n+2).
  \end{equation}
Then, for $\bfalpha\in \N^{r}$, there exists $\zeta\in \Z\setminus
  \{0\}$ such that
  \begin{math}
    \zeta \cdot \rho_{X,\bff}(\omega,\bfalpha)\in \Z
  \end{math}
with
\begin{displaymath}
  \log| \zeta|, 
\log|\zeta \cdot \rho_{X,\bff}(\omega,\bfalpha)| 
 \le  \h_{1}(g) + (r+1) \big( e  +  (r+1) (|\bfalpha|+1)
\, D_{X,\bff} \big)  D_{X,\bff} \, \kappa^{0}_{X,\bff}.
\end{displaymath}
\end{thm}

When $X$ is the affine space, we have the following more precise
result.

\begin{thm}
  \label{thm:9}
  Let $ \bff=(f_{1},\dots, f_{n})$ be a family of polynomials in
  $\Z[x_{1},\dots, x_{n}]\setminus \Z$ defining a complete
  intersection on $\A^{n}_{\Q}$, and
  $\omega=g \dd x_1 \wedge \dots \wedge \dd x_n$ a polynomial $n$-form
  defined over~$\Z$. Set $d_{i}=\deg(f_{i})$, $i=1,\dots, n$, and
  $e=\deg(g)$. For $l=1,\dots, n$, denote by $\pi_{l}$ the projection
  $ \A^{n}_{\Q}\to \A^{1}_{\Q}$ to the $l$-th coordinate of
  $\A_{\Q}^{n}$ and set
  \begin{displaymath}
    D_{\bff}=\prod_{j=1}^{n} d_{j}, \quad 
    \kappa^{0}_{\bff}=\sum_{j=1}^{n}\frac{\h(f_{j})}{d_{j}}+
    3(n+2)\log (n+2), \quad    \Delta_{\bff}=n D_{\bff}
    -\sum_{l=1}^{n}\deg(\pi_{l}(V(\bff))).  
  \end{displaymath}
  Then there exists $\vartheta\in \Z\setminus \{0\}$ with
  $ \log|\vartheta| \le n \,\kappa^{0}_{\bff} $ such that, for all
  $\bfalpha\in \N^{n}$, we have that \begin{math} \vartheta^{e
      +(|\bfalpha|+1) (\Delta_{\bff}+1)} \rho_{\A^{n}_{\Q},\bff}
    (\omega,\bfalpha) \in \Z\end{math} and
  \begin{displaymath}
\log\big|     \vartheta^{e
      +(|\bfalpha|+1) (\Delta_{\bff}+1)} \rho_{\A^{n}_{\Q},\bff}(\omega,\bfalpha) \big| \le \h_{1}(g)+
    (e + (|\bfalpha|+1)  \Delta_{\bff}) \,n\,       D_{\bff}
    \, \kappa_{\bff}^{0}.
  \end{displaymath}
\end{thm}

\begin{proof}[Proof of Theorems  \ref{thm:8} and \ref{thm:9}]
  We first consider Theorem \ref{thm:8}, for an arbitrary affine
  variety $X\subset \A^n_\Q$ of pure dimension $r$ satisfying the
  condition \eqref{eq:109}. For short, we set
  $\bfd=(d_{1},\dots, d_{r})\in \N^{r}$, $D=D_{X,\bff}$ and
  $\kappa=\kappa_{X,\bff}^{0}$.

By Corollary \ref{cor:1}, there are polynomials
 $\phi_{l}\in \Z[x_{l}]\setminus
\{0\}$ and $a_{l,i} \in \Z[\bfx]$, for $i,l=1,...,r$,  such that $\phi_{l}=\sum_{i=1}^{r}
a_{l,i}f_{i}$ for all $l$, 
with $ \deg(\phi_{l}), \deg (a_{l,i}) + d_{i} \le D \deg(X) $ and 
\begin{multline}
  \label{eq:85}
\h(\phi_{l}), \h(a_{l,i})+\h(f_{i})\le D \Big(\frac{\h(X)}{\deg(X)}+ 
  \sum_{j=1}^{r}\frac{\h(f_{j})}{d_{j}}  {+ (r+1) \log(2(r+2)(n+1)^{2})  \Big).}  
\end{multline}
For each $l$, let $\pi_{l}$ be the projection $ \A^{n}_{\Q}\to
\A^{1}_{\Q}$ to the $l$-th coordinate of $\A_{\Q}^{n}$.  Since
$\phi_{l}$ vanishes on $\pi_{l}(X\cap V(\bff))$ and $\bff$ is a
complete intersection on $X$, 
\begin{equation}
\label{eq:54}
 \deg(\phi_{l})\ge \deg(\pi_{l}(X\cap V(\bff))) \ge 1.
\end{equation}

We next apply the transformation law (Theorem \ref{propTL}) to
reduce the study of the residue $\rho_{X,\bff}(\omega,\bfalpha)$ to the case of
separated variables, considered in \S\ref{sec:tensor univariate
  case}.  Let $\bfu=(u_{1},\dots, u_{r})$ be a group of $r$
variables. Set $A=(a_{l,i})_{l,i}\in \Z[\bfx]^{r\times r}$,
$a_{l}=\sum_{i}a_{l,i} u_{i}$ and 
\begin{equation}
  \label{eq:84}
 H_{l}= \sum_{k=0}^{|\bfalpha|} \phi_{l}^{k}\,
a_{l}^{|\bfalpha|-k} \in
\Z[\bfu,\bfx]. 
\end{equation}
Set $H=\det(A) \cdot \prod_{l=1}^{r}H_{l} \in \Z[\bfu, \bfx]$ as in
\eqref{eq:107} and
$ G= \coeff_{\bfu^{\bfalpha}}(H(\bfu, \bfx)) \in \Z[\bfx]$.  By
Theorem \ref{propTL},
\begin{equation*}
\rho_{X,\bff}(\omega,\bfalpha)
= \ResXC \Big[\begin{matrix}
g\,  G\dd x_1\wedge \dots \wedge \dd x_r \cr 
\phi_1^{|\bfalpha|+1},...,\phi_{r}^{|\bfalpha|+1}\end{matrix}\Big].
\end{equation*}

By the expression in \eqref{eq:84} and the bounds in \eqref{eq:85}, $
\deg_{\bfx}(H_{l}) \le |\bfalpha| D$ for all~$l$. We also have
that $
\deg(\det(A))\le r D - |\bfd|$. Thus
\begin{equation}\label{eq:62}
\deg(g\, G) \le e - |\bfd|+(|\bfalpha|+ 1) \, r\,  D.
\end{equation}
For the length, using the inequalities in \eqref{eq:9},  the bounds for
the height in \eqref{eq:85} and comparing it with the constant
$\kappa$ in \eqref{eq:24}, we obtain
\begin{equation*}
    \h(H_{l}) \le  \max_{k}(k\h_{1}(\phi_{l}) +
  (|\bfalpha|-k) \h_{1}(a_{l})) + \log(|\bfalpha|+1)   \le
  |\bfalpha| D\, \kappa
\end{equation*}
and 
\begin{math}
  \h_{1}(\det(A)) \le r \max_{l,i}\h_{1}(a_{l,i}) + \log (r!)\le r \,
  D\, \kappa.
\end{math}
We deduce that
\begin{equation}
  \label{eq:87}
  \h_{1}(g\, G) \le \h_{1}(g)+ (|\bfalpha|+1) \, r \, D\,  \kappa.
\end{equation}

We next apply Theorem \ref{thm:1} to the residue in the right-hand
side of \eqref{eq:63}, that is, to the $\phi_{i}$'s instead of the
$f_i$'s, the polynomial~$g\, G$ instead of $g$, and the vector
$(|\bfalpha|, \dots, |\bfalpha|)\in \N^{r}$. Set
$\delta_{l}=\deg(\phi_{l})$ and
$\bfdelta=(\delta_{1},\dots, \delta_{r})\in \N^{r}$, and let
$\phi_{l,\delta_{l}}$ be the leading coefficient of $\phi_{l}$. With
notation as in this result, set
\begin{displaymath}
  \zeta=\gamma \cdot 
  \prod_{l=1}^{r}\phi_{l,\delta_{l}}^{\deg(g\, G)+r - (|\bfalpha|+1)(|\bfdelta|-1)} \in \Z\setminus \{0\}.
\end{displaymath}
Hence $\zeta \cdot \rho_{X,\bff}(\omega,\bfalpha)\in\Z$. Moreover, setting
\begin{multline*}
    \kappa_{1}=  \big(\deg(g\, G) +r- (|\bfalpha|+1)|\bfdelta| \big)\, \sum_{l=1}^{r} \h(\phi_{l}) \\
    + \deg(g\, G)\, \big( \h(X) +
    \deg(X) (n+3)\log(2n+3)\big),
\end{multline*}
we have that $ \log|\zeta|\le \kappa_{1} $ and
$ \log|\zeta\cdot \rho_{X,\bff}(\omega,\bfalpha)| \le \h_{1}(g\,
G)+\kappa_{1}$.
We also have that $d_{i}\ge 1$ for all $i$ and, by \eqref{eq:54},  $\delta_{l}\ge 1$ for all $l$. Hence
$|\bfd|, |\bfdelta|\ge r$. Using this together with \eqref{eq:87}, we
obtain
\begin{align*}
\h_{1}(g\, G)+  \kappa_{1}\le& 
 \h_{1}(g)+ (|\bfalpha|+1) \, r \,  D\, \kappa \\ 
&+\big(e+(|\bfalpha|+1) \, r\, D\big) \, \big(r \, D\, \kappa +  \h(X) +
    \deg(X) (n+3)\log(2n+3)\big)\\
 \le& \h_{1}(g) +
  e\, (r+1) \,D\, \kappa + (|\bfalpha|+1) (r+1)^{2}
D^{2}\, \kappa,
\end{align*}
and similarly $\kappa_{1}\le e\, (r+1)D\, \kappa + (|\bfalpha|+1)
(r+1)^{2} D^{2}  \kappa$, which gives the bound in the theorem.

In the case $X=\A^{n}_{\Q}$, set
$\vartheta=\prod_{i=1}^{n}\phi_{i,\delta_{i}}\in \Z\setminus \{0\}$.
By \eqref{eq:85}, this quantity satisfies the inequality
$\log|\vartheta|\le n\, D\, \kappa$ as stated. 
Set 
$$ 
\lambda =\vartheta^{\deg(g\, G)+n
  -(|\bfalpha|+1)(|\bfdelta|-1)} \rho_{X,\bff}(\omega,\bfalpha).
$$ 
By Theorem \ref{thm:6}, $\lambda\in \Z$ and this integer can be
bounded by
\begin{multline*}
  \log| \lambda |\le 
  \h_{1}(g\, G)
+\big(\deg(g \, G)+n-(|\bfalpha|+1)|\bfdelta|\big)\, \sum_{l=1}^{n}\h(\phi_{l}) \\
  + \big(\deg(g\, G)-|\bfdelta|+n\big)\, \log(2). 
\end{multline*}
By the inequalities \eqref{eq:62} and \eqref{eq:54},
\begin{equation}
  \label{eq:64}
\deg(g\, G)+n-(|\bfalpha|+1)(|\bfdelta|-1)\le e
+(|\bfalpha|+1) (\Delta+1).
\end{equation}
Set also $\mu=\vartheta^{e +(|\bfalpha|+1) (\Delta+1)}
\rho_{X,\bff}(\omega,\bfalpha) = \lambda \, \vartheta^{c}$, where the
exponent $c\ge 0$ is the difference between both sides of the
inequality in \eqref{eq:64}. Hence $\mu\in \Z$ and
\begin{align*}
\log|\mu|& \le  \h_{1}(g)
 +\big(e +(|\bfalpha|+1)\Delta \big)\, \sum_{l=1}^{n}\h(\phi_{l})+\big (e-|\bfd|+(|\bfalpha|+1)\,
n\, D\big)\log(2)\\
&\le   \h_{1}(g)+
    \big(e + (|\bfalpha|+1)  \Delta\big) \,n\, D\,   \kappa,
\end{align*}
which concludes the proof of Theorem \ref{thm:9}.
\end{proof}

We apply these results to bound the coefficients in the Bergman-Weil
trace formula. Let $\bff=(f_{1},\dots, f_{n}) \in\C[\bfx]^{n}$ be a
complete intersection on $\A^{n}_{\Q}$. Let
$\bfz=(z_{1},\dots, z_{n})$ be a group of variables and, for
$i,j=1,\dots, n$, set
\begin{equation}\label{eq:106}
  h_{i,j}={(f_{j}(x_{1},\dots, x_{j},z_{j+1}, \dots, z_{n})-
    f_{j}(x_{1},\dots, x_{j-1},z_{j}, \dots,
    z_{n})})/({z_{j}-x_{j}}). 
\end{equation}
These are  polynomials in $\C[\bfx,\bfz]$ that verify the identity \eqref{eq:105}, namely 
\begin{displaymath}
  f_{i}(\bfz)-f_{i}(\bfx)= \sum_{j=1}^{n}h_{i,j}(\bfx,\bfz)
  (z_{j}-x_{j}). 
\end{displaymath}
Put $h_{i}=\sum_{j=1}^{n}h_{i,j}\dd y_{j}$ and, for $p\in \C[\bfx]$
and $\bfalpha\in \N^{n}$, set 
\begin{equation}
  \label{eq:78}
  p_{\bfalpha}= \rho_{\A^{n}_{\C}, \bff} \Big(g
  \bigwedge_{i=1}^{n}h_{i}, \bfalpha\Big)\in \C[\bfx] 
\end{equation}
for the coefficient corresponding to $\bfalpha$ in the representation
of $p$ given by the Bergman-Weil trace formula.  These are polynomials
of degree bounded by $\sum_{j=1}^{n}\deg(f_{j}) - n$.

\begin{cor}
  \label{cor:3}
  Let $ \bff=(f_{1},\dots, f_{n})$ be a family of polynomials in
  $\Z[x_{1},\dots, x_{n}]\setminus \Z$ defining a complete
  intersection on $\A^{n}_{\Q}$, and $p\in \Z[x_{1},\dots, x_{n}]$.
  Set $\bfd=(d_{1},\dots, d_{n})\in \N^{n}$ with $d_{i}=\deg(f_{i})$,
  and $e=\deg(p)$. Set also
\begin{displaymath}
D_{\bff}= \prod_{j=1}^{n}d_{j} \and  \kappa_{\bff}''=\sum_{j=1}^{n}\frac{\h(f_{j})}{d_{j}}+ 3(n+2)\log (n+2).
\end{displaymath}
Then there exists $\vartheta\in \Z\setminus \{0\}$ with $
  \log|\vartheta| \le n \,\kappa $ such that, for $\bfalpha\in
  \N^{n}$, the coefficient  $p_{\bfalpha}$  in \eqref{eq:78} satisfies
  that 
 $\vartheta^{e+|\bfd|+(|\bfalpha|+1) (nD_{\bff}+1)} p_{\bfalpha}\in \Z
[x_{1},\dots, x_{n}]$ and
\begin{displaymath}
 \h\big(\vartheta^{e+|\bfd|+(|\bfalpha|+1) (nD_{\bff}+1)} g_{\bfalpha}\big) \le
  \h_{1}(g) + (e+|\bfd|+(|\bfalpha|+1) (nD_{\bff}+1)) \, n\,
  D_{\bff}\, \kappa''_{\bff}.   
\end{displaymath}
\end{cor}

\begin{proof}
  With notation as in \eqref{eq:106}, consider the polynomial $n$-form $
  \omega= p\, \bigwedge_{j=1}^{n} h_{j}$ in the variables $\bfz$ with
  coefficients in $\Z[\bfx]$.  From \eqref{eq:106}, we verify  that
  $\deg(h_{i,j})\le d_{j}-1$ and $\h_{1}(h_{i,j})\le \h(f_{j})$. This
  implies that
\begin{displaymath}
\deg(\omega)\le |\bfd|-n \and \h_{1}(\omega)\le \h_{1}(p)+ 
\sum_{k=1}^{n}\h_{1}(f_{j})+ |\bfd|\log(2n+1),
\end{displaymath}
and the result follows then from Theorem \ref{thm:9}.
\end{proof}

To extend our bounds for residues to rational forms, we need the
following version of the arithmetic Nullstellensatz.  It is a direct
application of \cite[Theorem~0.1]{DKS:hvmsaN}.

\begin{lem}
  \label{lemm:8}
  Let $X\subset \A^{n}_{\Q}$ be a variety of pure dimension $r$ and
  $f_{i}\in \Z[x_{1},\dots, x_{n}]$, $i=0,\dots, r$, polynomials
  without common zeros in $X$. Set  $d_i=\deg(f_{i})$,
  $i=0,\dots, r$. Then there exist $a\in \Z\setminus \{0\}$ and
  $p\in \Z[x_{1},\dots, x_{n}]$ such that
\begin{equation*}
p f_{0} \equiv  a 
\mod{(f_{1},\dots, f_{r})} \quad \text{ on } X 
\end{equation*}
with $\deg(p)+ d_0 \le \big(\prod_{i=0}^{r}d_{i}\big)\, \deg(X)$ and
\begin{displaymath}
  \log|a|,   \h(p)+\h(f_{0}) 
  \le \Big(\prod_{j=0}^{r}d_{j}\Big)
\hcan(X)+
 \deg(X) \Big(\sum_{i=0}^{r} \Big(\prod_{j\ne i}d_{j}\Big){\h(f_{i})} + (4r+8)\log(n+3)\Big).
\end{displaymath}
\end{lem}

The following result bounds the numerator and denominators in a
residue multi-sequence, in the most general case considered in this
paper.

\begin{thm}
\label{thm:10}
Let $X\subset \A^n_\Q$ be a variety of pure dimension $r\ge 1$ and $
\bff=(f_{1},\dots, f_{r})$ a family of polynomials in
$\Z[x_{1},\dots, x_{n}]\setminus \Z$ defining a complete intersection
on $X$.  Let $\omega $ be a rational $r$-form defined over $\Q$
that is  regular on
$X\cap V(\bff)$. 
Write $ \omega= \tau /f_0$ with $\tau$ a polynomial $r$-form
defined over $\Z$ and $f_0\in \Z[x_{1},\dots, x_{n}]$ not vanishing on
$X\cap V(\bff)$.  Set $d_{i}=\deg(f_{i})$, $i=0,\dots, r$,  and $e =\deg(\tau)$. 
Set also
\begin{equation*}
D_{X,\bff}=\deg(X) \prod_{j=1}^r d_j \and \kappa_{X,\bff}=\frac{\hcan(
  X)}{\deg( X)} + 
\sum_{j=1}^{r}\frac{\h( f_{j})}{d_{j}}+ 4(n+5)^{2} \log((n+1)
\deg(X)).
\end{equation*}
Then, for $\bfalpha\in \N^{r}$, there is $\zeta\in \Z\setminus
  \{0\}$ such that
  \begin{math}
\zeta \cdot \rho_{X,\bff}(\omega,\bfalpha) \in \Z 
  \end{math}
with
\begin{multline*}
  \log| \zeta|+ \h_{1}(\tau), 
\log|\zeta \cdot \rho_{X,\bff}(\omega,\bfalpha) |
\le  \binom{n}{r}\big(  \h_{1}(\tau) + 
e\, (r+1)  \, D_{X,\bff}\, \kappa_{X,\bff} 
\\  + (\bfalpha+1) \big( 2  (r+1) \,  D_{X,\bff} \h(f_{0}) +
     (3d_{0}+r+1)\, D_{X,\bff}^{2}
 \kappa_{X,\bff} \big)\big).  
\end{multline*}
\end{thm}

When $X$ is the affine space, we have the following more
precise result. 

\begin{thm}
\label{thm:11}
Let $ \bff=(f_{1},\dots, f_{n})$ be a family of polynomials in
$\Z[x_{1},\dots, x_{n}]\setminus \Z$ defining a complete intersection
on $\A^{n}_{\Q}$, and $\omega $ a rational $n$-form that is regular on
$V(\bff)$.  Write
$ \omega= (g/f_0) \dd x_{1}\wedge \dots\wedge \dd x_{n}$ with
$g,f_0\in \Z[x_{1},\dots, x_{n}]$ and $f_0$ not vanishing on
$V(\bff)$. Set $d_{i}=\deg(f_{i})$, $i=0,\dots, n$, and
$e= \deg(g)$. 
Set also
\begin{equation*}
  D_{\bff}=\prod_{j=1}^n d_j \and \kappa_{\bff}= \sum_{i=1}^{n}\frac{\h(f_{i})}{d_{i}} + (4n+{9})\log(n+3).
\end{equation*}
Then, for $\bfalpha\in\N^{n}$, there exists $\zeta\in \Z\setminus
\{0\}$ 
such that $\zeta\cdot \rho_{\A_{\Q}^{n},\bff}(\omega, \bfalpha) \in \Z$
with 
\begin{multline*}
  \log|\zeta|\,,\log|\zeta\cdot \rho_{\A_{\Q}^{n},\bff}(\omega, \bfalpha) | \le
\h_1(g) + e\, n\, D_{\bff}\kappa_{\bff} \\+ (|\bfalpha|+1)\, \big(D_{\bff}\, \h(f_0) + (d_0+n)(n D_{\bff}+1)\,D_{\bff} \,\kappa_{\bff}\big).   
    \end{multline*} 
\end{thm}

\begin{proof}[Proof of Theorems \ref{thm:10} and \ref{thm:11}]
  Set for short $D=D_{X,\bff}$, $\kappa=\kappa_{X,\bff}$ and
\begin{equation*}
\theta_X=  \frac{\hcan(X)}{ \deg(X)}+
  \sum_{i=1}^{r}\frac{\h(f_{i})}{d_{i}} + (4r+{9})\log(n+3).
\end{equation*}
By Lemma \ref{lemm:8} {and the comparison between the height
  and the length}, there exist $a\in \Z\setminus \{0\}$ and
$p\in \Z[x_{1},\dots, x_{n}]$ such that
$p \, f_0 \equiv a \mod{(f_{1},\dots, f_{r})}$ on $X$ with
$\deg(p)+d_0\,\le d_0\, D$ and
\begin{equation}
\label{eq:75}
\log|a|, {\h_{1}}(p)+{\h_{1}}(f_0)\le  D\, (d_0 \, \theta_{X}+  {\h(f_0)}).
\end{equation}
This implies that $ {(p \, f_0 - a)^{|\bfalpha|+1} \equiv 0}
\mod{(f_{1}^{\alpha_{1}+1},\dots, f_{r}^{\alpha_{r}+1})} $ on $X$.
Hence
\begin{displaymath}
   \frac{1}{f_0}\equiv \frac{{q}}{a^{|\bfalpha|+1}} 
   \mod{(f_{1}^{\alpha_{1}+1},\dots, f_{r}^{\alpha_{r}+1})} \text{ on
   }X \setminus V(f_0)
 \end{displaymath}
 with
 $ {q=\sum_{j=1}^{|\bfalpha|+1} {|\bfalpha|+1 \choose j} p^{j}
   (-f_{0})^{j-1} a^{|\bfalpha|+1-j} \in \Z[x_{1},\dots, x_{n}]}$.
 Consider the polynomial $r$-form defined over $\Z$ given by
 $\sigma= {q}\, \tau$.  Since
 $X \setminus V(f_0)$ is a neighborhood of $X\cap V(\bff)$, by
 Proposition \ref{prop:16},
\begin{equation}
\label{eq:81}
  \rho_{X,\bff}(\omega,\bfalpha)= \frac{1}{a^{|\bfalpha|+1}}\Res_{X_{\C}}
  \Big[\begin{matrix} \sigma
    \cr
\bff^{\bfalpha+\bfone} \end{matrix}\Big]. 
\end{equation}
The degree and the length of this polynomial $r$-form are bounded by
\begin{equation*} 
 \deg(\sigma) \le  \deg(\tau) +
(|\bfalpha|+1)\deg(p)+|\bfalpha|\, d_0 \le  e+ (|\bfalpha|+1)\, d_{0}\, D
\end{equation*}
and
\begin{multline}
\label{eq:103}
   \h_{1}(\sigma) \le  \h_{1}(\tau)+ (|\bfalpha|+1) \h_{1}(p) +
  |\bfalpha| \h_{1}(f_0) 
\\ \le  \h_{1}(\tau) + (|\bfalpha|+1) \, (D\, ( d_0 \,  \theta_{X}+
 {\h(f_0)}){+\log(2)}),
\end{multline}
{using the submutiplicativity of the length.}  By Proposition
\ref{prop:1}, there are affine polynomials $\ell_i\in \Z[\bfx]$,
$i=1,...,n$, with $\h(\ell_{i}) \le 2\log (\deg(X)) + n \log (2)$ such
that, for every subset $I\subset \{1,\dots, n\}$ of cardinality $r$,
\begin{equation*}
  \# \Big(X_{\C}\cap \bigcap_{i\in I}V(\ell_{i})\Big)= \deg(X). 
\end{equation*}
Consider the invertible linear map
  \begin{equation}\label{eq:82}
    \bfell\colon  \A_{\Q}^{n}\longrightarrow \A_{\Q}^{n}, \quad \bfx\longmapsto
  (\ell_{1}(\bfx), \dots, \ell_{n}(\bfx))  
  \end{equation}
  and set $\wt X=\bfell(X)\subset \A^n_\Q$. Let $\bfm=\bfell^{-1}$ be
  the inverse map. We have that $\det (\bfell) \in \Z\setminus\{0\}$,
  and it follows from Cramer's formulae that $\det(\bfell) \, \bfm\in
  \Z^{n\times n}$. Moreover,
\begin{multline} 
\label{eq:79}
\log|\det (\bfell)| , \h(\det (\bfell) \,\bfm) 
\leq
n \max_{i}\h(\ell_{i}) +  \log (n!) \\ \le 2n\log(\deg(X))+
2n^{2}\log(2),
\end{multline}
where $\h(\det (\bfell) \,\bfm) $ stands for the logarithm of the maximum
of the absolute values of the entries of this matrix with
integer coefficients.
    
By the
invariance of the residue under change of variables (Proposition~\ref{changebasis}),
\begin{equation}\label{eq:108}
  \rho_{X,\bff}(\omega,\bfalpha)= \frac{1}{a^{|\bfalpha|+1}} \Res_{\wt X_{\C}}
  \Big[\begin{matrix} \bfm^{*} \sigma \cr 
     (\bfm^{*} f_{1})^{\alpha_{1}+1}, \dots, (\bfm^{*} f_{r})^{\alpha_{r}+1} \end{matrix}\Big].
\end{equation}

Let $\bfy=(y_{1},\dots, y_{n})$ be a group of $n$ variables,
corresponding to the coordinates of the second $\A_{\Q}^{n}$ in
\eqref{eq:82}.  For each pair of subsets $I,J \subset \{1,\dots, n\}$
of cardinality $r$, put $\bfm_{I,J}= (m_{i,j})_{i\in I, j\in J}$ for
the corresponding $r\times r$-submatrix of $\bfm$. Write
$\sigma= \sum_{I}q_{I}\dd\bfx_{I}$ with $q_{I}\in \Z[\bfx]$.  For each
multi-index $J$, set
  \begin{equation}
    \label{eq:91}
\wt q_{J}=\det (\bfell)^{\deg(\sigma)+r}\, \sum_{I} \det (\bfm_{I,J}) \,
\bfm^{*}q_{I} \in
\Z[\bfy]    
  \end{equation}
  and $\wt \sigma= \sum_{J} { \wt q_{J}} \dd \bfy_{J}$, so that
  $\bfm^{*}\sigma = {\det (\bfell)^{-\deg (\sigma) -r}} \, \wt
  \sigma$.
  Set also
  $\wt f_{i}=\det (\bfell)^{d_{i}}\, \bfm^{*}f_{i} \in \Z[\bfy]$. It
  follows from \eqref{eq:108} that
\begin{align}\label{eq:86}
  \rho_{X,\bff}(\omega,\bfalpha)= &{a^{-|\bfalpha|-1}}
  {\det (\bfell)^{-\deg (\sigma) -r + \langle \bfalpha+\bfone,\bfd\rangle}}\,  \Res_{\wt X_{\C}}
  \Big[\begin{matrix} \wt \sigma \cr
    \wt \bff^{\bfalpha+\bfone} \end{matrix}\Big] \\
\nonumber    =  &
 {a^{-|\bfalpha|-1}}
  {\det (\bfell)^{-\deg (\sigma) -r + \langle \bfalpha+\bfone,\bfd\rangle}} 
  \sum_{J} 
\Res_{\wt X_{\C}}
  \Big[\begin{matrix} \wt q_J \dd \bfy_J  \cr
    \wt \bff^{\bfalpha+\bfone} \end{matrix}\Big]           
\end{align}

We next bound the height of each residue in the right-hand side of \eqref{eq:86}. 
We have that $\deg({\wt
    X}) = \deg(X)$ and, using \cite[Lemma~2.7]{KPS:sean} and
   \cite[Proposition~2.39(5)]{DKS:hvmsaN}, the height of $\wt X$ can be
  bounded by
\begin{equation*}
\hcan({\wt X} )
\leq \hcan(X) + (r+1)\big( 2\log(\deg(X)) + n \log (2) 
+ 12\, \log(n+1)\big)\deg(X). 
\end{equation*}  
We have that $\deg(\wt f_{i})=d_{i}$ and, by
\cite[Lemma~1.2(c)]{KPS:sean} and the bounds in \eqref{eq:79},
\begin{align*}
\nonumber   \h(\wt f_{i})&\le \h(f_{i})+ d_{i}(2n\log(\deg(X))+
2n^{2}\log(2)+2\log(n+1))\\ & \le \h(f_{i})+ 2\, n\, d_{i}\log(2^{n+1}\deg(X)).
\end{align*}

Now fix a multi-index $J$ and let
$\wt q_{J}$ be as in  \eqref{eq:91}. We have that
\begin{equation}\label{eq:90bis}
\deg(\wt q_{J}) \le \deg(\sigma) \le e+ (|\bfalpha|+1)\, d_{0}\, D . 
\end{equation}
Applying again \cite[Lemma~1.2(c)]{KPS:sean} and the
bounds in \eqref{eq:79}, we get, for each multi-index $I$ in \eqref{eq:91}, 
\begin{equation*}
\h_{1}\big(  \det (\bfell)^{\deg (\sigma)+r} \det (\bfm_{I,J}) \,
\bfm^{*}q_{I}\big)\le  \h_{1}(q_{I})+ 2\, n\, 
  \log(2^{n+1}\deg(X)) \, \deg(q_{I}).
\end{equation*}
Hence, for each $J$,
\begin{align}
\label{eq:90ter} 
\nonumber \h_1(\wt q_J ) 
\le &\h_{1}(\tau) + (|\bfalpha|+1)\, (D\, (d_0 \, \theta_X +  \h(f_0)){+\log(2)}) \\ &+ 2\, n\,  \big(e + (|\bfalpha|+1)\, d_0\, D\big) \log(2^{n+1}\deg(X)) \\ 
\nonumber  \le & \h_1(\tau) + 2\, n\, e\, \log(2^{n+1}\deg(X)) 
\\ 
\nonumber & + (|\bfalpha|+1) \big(\big(
2\, n\, D  \log(2^{n+1}\deg(X)) + D\, \theta_{X}\big)\, d_{0}+ D\deg(X)
\h(f_0) {+\log(2)}\big) \\
\leq & \h_1(\tau) + (2\, n+1)\, e\,
  \log(2^{n+1}\deg X) + (|\bfalpha|+1)\, D\, \kappa_{1}
\end{align}
with
\begin{equation*}
 \kappa_{1}=  d_{0}\Big(\frac{\hcan(X)}{ \deg(X)} 
+\sum_{i=1}^{r} \frac{\h(f_{i})}{d_{i}} + 2n\log(2^{n+1}\deg(X)) +
(4r+{9}) \log(n+3)\Big)
+ \h(f_{0}) {+\log(2)}.
\end{equation*}

The result follows then from Theorem \ref{thm:8}. To abridge the rest
of the proof, we state the obtained bounds without detailing the
involved computations.  Set
\begin{equation*}
  \kappa_{2}  = \frac{\hcan( X)}{\deg( X)} + 
\sum_{j=1}^{r}\frac{\h( f_{j})}{d_{j}}+ 2(n+1)^{2}\log(\deg(X))
+ 3(n+2)(n+5)\log (n+2),
\end{equation*}
which is a quantity that bounds from above the constant $\kappa^{0}$
given by this theorem applied to the variety $\widetilde X$ and the
$ \widetilde f_{j}$'s.  Then, this theorem together with the bounds
in~\eqref{eq:90bis} and \eqref{eq:90ter} implies that there exists
$\zeta_{J}\in \Z\setminus\{0\}$ such that
$\zeta_J\cdot \rho_{\wt X, \wt \bff}( \wt q_{J} \dd
\bfy_{J},\bfalpha)\in \Z$ and
\begin{multline}
  \label{eq:98}
\log|\zeta_J|, \log |\zeta_J\cdot \rho_{\wt X, \wt \bff}( \wt
q_{J}\dd\bfy_{J},\bfalpha)| \le \h_1(\tau) + 2n\, e\, \log(2^{n+1}\deg(X)) \\ 
+ (r+1) ( e \, D\, \kappa_{2} + (|\bfalpha|+1)(D\, \kappa_{1} + (d_0 +r+1)\, D^{2}\, \kappa_{2})). 
\end{multline}

It follows from \eqref{eq:86} that we  can take the wanted denominator
$\zeta\in \Z\setminus \{0\}$ as 
$$
\zeta = 
{a^{|\bfalpha|+1}}
  {\det (\bfell)^{\deg (\sigma) + r - \langle \bfalpha+\bfone,\bfd\rangle}}
\prod_{J} \zeta_J,
$$
the product being over the subsets $J\subset\{1,\dots, n\}$ of
cardinality $r$. 
We deduce from~\eqref{eq:75}, \eqref{eq:79} and \eqref{eq:98} that
\begin{align*}
\log &|\zeta| + \h_1(\tau),  
\log \big|\zeta \cdot \rho_{X,\bff}(\omega,\bfalpha)\big|  \le
       d_{0}\, D\, \theta_{X} + D  \h(f_{0}) + 
2n\log(2^{n+1}\deg(X)) \, e\\
&+ 
\binom{n}{r} \Big( \h_1(\tau)  + \Big( \binom{n}{r} +1\Big)  2\,n
\log(2^{n+1}\, \deg(X)) \,e + (r+1) \binom{n}{r} e\, D\, 
\kappa_{2}+ \log \binom{n}{r}\\
 &+ (|\bfalpha|+1)
\, \Big ( (r+1) \binom{n}{r}  \Big(D\,\kappa_{1} + (d_0 +r+1)\, D^{2}\,
\kappa_{2}\Big) + 
2\, n\log(2^{n+1}\deg(X)) \, d_{0} \,D \Big). 
\end{align*}
The bound in Theorem \ref{thm:10} is a simpler (and rougher) version of
this  upper bound. 

In the case $X=\A^{n}_\Q$, we apply Theorem \ref{thm:9} to the
residues associated to the polynomial $n$-form
$a^{-|\bfalpha|-1} {q}\, g \dd \bfx =
a^{-|\bfalpha|-1}\sigma $
appearing in \eqref{eq:81}. With notation as in this result, set
\begin{displaymath}
  \zeta= a^{|\bfalpha|+1}\vartheta^{\deg(\sigma) + (|\bfalpha|+1) (n D+1)}\in
  \Z\setminus \{0\}.
\end{displaymath}
By Theorem \ref{thm:9}, we have that
$\zeta \cdot \rho_{\A_{\Q}^{n},\bff} (a^{-|\bfalpha|-1} \sigma,
\bfalpha)\in \Z$ and, using~\eqref{eq:75} and
\eqref{eq:103}, 
\begin{align*}
  \log|\zeta| & \le   (|\bfalpha|+1) \log|a|+ \big(\deg(\sigma)  +
  (|\bfalpha|+1)\, (n D+1)\big)\, \log|\vartheta|\\
& \le (|\bfalpha|+1)
 (d_{0} \, D\, \kappa + D \,  {\h(f_0)}) + \big(e+(|\bfalpha|+1)\, 
 ((d_0+n)\, D +1)\big)\, n\, D\, \kappa \\
& \le e\, n\, D\, \kappa + (|\bfalpha|+1)\, \big(D\, \h(f_0) +
  (d_0+n)(n D+1)\, D\, \kappa\big). 
\end{align*}
with $\kappa=\theta_{\A^{n}_{\Q}}$ as in the statement.  Similarly,
$\log|\zeta\cdot \rho_{\A_{\Q}^{n},\bff} (a^{-|\bfalpha|-1} \sigma,
\bfalpha)| $ is bounded by
\begin{multline*}
 \h_{1}(g) + \big(\deg (\sigma) + (|\bfalpha|+1)\, n\,
                 D)\big)\, n\, D\, \kappa 
 + (|\bfalpha|+1)\, (d_{0} \, D\, \kappa + D \,  {\h(f_0)}  {+\log(2)}) \\
\le \h_1(g) + e\, n\, D\, \kappa  + (|\bfalpha|+1)\, \big(D\,
      \h(f_0) + (d_0+n)(n D+1)\, D\, \kappa\big), 
\end{multline*}
yielding the bounds in Theorem \ref{thm:11}.
\end{proof}


\newcommand{\noopsort}[1]{} \newcommand{\printfirst}[2]{#1}
  \newcommand{\singleletter}[1]{#1} \newcommand{\switchargs}[2]{#2#1}
  \def\cprime{$'$}
\providecommand{\bysame}{\leavevmode\hbox to3em{\hrulefill}\thinspace}
\providecommand{\MR}{\relax\ifhmode\unskip\space\fi MR }
\providecommand{\MRhref}[2]{%
  \href{http://www.ams.org/mathscinet-getitem?mr=#1}{#2}
}
\providecommand{\href}[2]{#2}

\end{document}